\documentclass[reqno,10pt]{amsart}

\usepackage[utf8]{inputenc}

\usepackage{amssymb,amsmath,amsthm,amsfonts,color}
\usepackage{mathrsfs,dsfont,comment,mathscinet}
\usepackage{todonotes}
\usepackage{tikz}
\usepackage{chngcntr}

\usepackage{enumerate,esint}
\usepackage[left=2.8cm,right=2.8cm,top=2.8cm,bottom=2.8cm]{geometry}
\usepackage{mathtools}
\usepackage{graphicx,tikz}
\usepackage[format=hang,labelfont=bf]{caption}
\usepackage{ wasysym }

\usepackage{epstopdf}
\usepackage[colorlinks=true, pdfstartview=FitV, linkcolor=blue, 
            citecolor=blue, urlcolor=blue]{hyperref}

\allowdisplaybreaks
\numberwithin{equation}{section}
\theoremstyle{plain}
\newtheorem{theorem}{Theorem}[section]
\newtheorem{proposition}[theorem]{Proposition}
\newtheorem{lemma}[theorem]{Lemma}
\newtheorem{example}[theorem]{Example}
\newtheorem{corollary}[theorem]{Corollary}
\theoremstyle{definition}
\newtheorem{definition}[theorem]{Definition}
\newtheorem{remark}[theorem]{Remark}

\newtheorem*{theorem*}{Theorem}

\makeatletter
\def\th@plain{%
  \thm@notefont{}
  \itshape 
}
\def\th@definition{%
  \thm@notefont{}
  \normalfont 
}
\makeatother

\definecolor{mblue}{HTML}{13439b}

\newcommand\R{\mathbb R}

\newcommand\N{\mathbb N}
\newcommand\Z{\mathbb Z}
\renewcommand\S{\mathbb S}

\renewcommand{\d}{\mathrm{d}}

\newcommand{\restr}[1]{|_{#1}}

\def\Xint#1{\mathchoice
{\XXint\displaystyle\textstyle{#1}}%
{\XXint\textstyle\scriptstyle{#1}}%
{\XXint\scriptstyle\scriptscriptstyle{#1}}%
{\XXint\scriptscriptstyle\scriptscriptstyle{#1}}%
\!\int}
\def\XXint#1#2#3{{\setbox0=\hbox{$#1{#2#3}{\int}$ }
\vcenter{\hbox{$#2#3$ }}\kern-.6\wd0}}

\def\dashint{\Xint-}

\definecolor{bblue}{HTML}{3C3C9F}

\newcommand{\loc}{\mathrm{loc}}
\newcommand{\var}{\mathrm{Var}}
\newcommand{\tr}{\mathrm{tr}}
\newcommand{\cof}{\mathrm{cof}}
\newcommand{\adj}{\mathrm{adj}}

\renewcommand{\deg}{\mathrm{deg}}
\newcommand{\dom}{\mathrm{dom}}
\newcommand{\im}{\mathrm{im}}
\renewcommand{\div}{\mathrm{div}}
\newcommand{\mult}{\mathrm{mult}}

\newcommand{\dist}{\mathrm{dist}}
\newcommand\wk{\rightharpoonup}

\newcommand{\rnn}{\R^{N\times N}}

\newcommand*\closure[1]{\overline{#1}}

\newcommand{\leb}{\mathscr{L}^N}
\newcommand{\haus}{\mathscr{H}^{N-1}}

\DeclareMathOperator*{\aplim}{\mathrm{ap\,lim}}

\DeclareMathOperator*{\esssup}{\mathrm{ess\,sup}}

\DeclareMathOperator*{\argmin}{\mathrm{argmin}}

\def\Xint#1{\mathchoice
{\XXint\displaystyle\textstyle{#1}}%
{\XXint\textstyle\scriptstyle{#1}}%
{\XXint\scriptstyle\scriptscriptstyle{#1}}%
{\XXint\scriptscriptstyle\scriptscriptstyle{#1}}%
\!\int}
\def\XXint#1#2#3{{\setbox0=\hbox{$#1{#2#3}{\int}$ }
\vcenter{\hbox{$#2#3$ }}\kern-.6\wd0}}

\def\dashint{\Xint-}

\title[Quasistatic evolution  in magnetoelasticity]{Quasistatic evolution in magnetoelasticity under\\ subcritical coercivity assumptions}
\author[M. Bresciani]{Marco Bresciani}
\address[M. Bresciani]{Institute of Analysis and Scientific Computing,
		Technische Universit\"{a}t Wien,
		Wiedner Hauptstrasse 8--10, 1040 Vienna, Austria.
}
\email{marco.bresciani@tuwien.ac.at}

\date{\today}

\begin{document}

\setlength\parindent{0pt}

\vskip .2truecm
\begin{abstract}
 We study a variational model of magnetoelasticity both in the static and in the quasistatic setting. The model features a mixed Eulerian-Lagrangian formulation, as magnetizations are defined on the deformed configuration in the actual space. The magnetic saturation constraint is formulated in the reference configuration and involves the Jacobian determinant of deformations. These belong to the class of possibility discontinuous deformations excluding cavitation introduced by Barchiesi, Henao and Mora-Corral. We establish a compactness result which, in particular, yields the convergence of the compositions of magnetizations with deformations. In the static setting, this enables us to prove the existence of minimizers by means of classical lower semicontinuity methods. Our compactness result also allows us to address the  analysis in the quasistatic setting, where we examine rate-independent evolutions driven by applied loads and boundary conditions. In this case, we prove  the existence of energetic solutions. 
\end{abstract}
\maketitle

\section{Introduction}

Magnetoelastic materials are characterized by their tendency to experience mechanical deformations in response to external magnetic fields. This peculiar behaviour is termed magnetostriction and constitutes the foundation of the technology behind many devices such as sensors and actuators. 

A first phenomenological theory of magnetoelasticity has been proposed by Brown \cite{brown0,brown} in the form of a variational principle. 
The theory takes as independent variables the deformation and the magnetization. The latter should be interpreted as the local density of magnetic dipoles per unit volume. While the first variable is classically defined on the reference configuration (Lagrangian), the second one is naturally defined on the deformed configuration in the actual space (Eulerian). This setting constitutes one of the main features of the theory.

Given its sound variational structure, the theory of Brown has been the  subject of rigorous analytical investigations \cite{desimone.dolzmann,desimone.james,james.kinderlehrer}, although the problem of the existence of equilibrium configurations has been addressed only in recent years, see the brief review of the literature below. This is because the mixed Eulerian-Lagrangian formulation entails several substantial difficulties that make various variational techniques inapplicable. 
Overall, the mathematical modeling of magnetoelasticity poses very challenging problems that can be qualified as: nonlinear, as magnetostrictive materials can experience large deformations; nonconvex, as such problems are subjected to nonconvex constraint due to magnetic saturation; and nonlocal, as the magnetic response of the material depends on the shape assumed by the deformed body.

Let $\Omega\subset\R^N$ represent the reference configuration of a magnetoelastic body subjected to elastic deformations $\boldsymbol{y}\colon \Omega \to \R^N$ and magnetizations $\boldsymbol{m}\colon \boldsymbol{y}(\Omega)\to \R^N$.
According to the theory of Brown, equilibrium configurations  correspond to minimizers of the following energy functional:
\begin{equation}
    \label{eqn:intro-magnetoelastic-energy}
    (\boldsymbol{y},\boldsymbol{m})\mapsto \int_\Omega W(D\boldsymbol{y},\boldsymbol{m}\circ\boldsymbol{y})\,\d\boldsymbol{x}+\int_{\boldsymbol{y}(\Omega)} |D\boldsymbol{m}|^2\,\d\boldsymbol{\xi}+\frac{1}{2}\int_{\R^N} |Du_{\boldsymbol{m}}|^2\,\d\boldsymbol{\xi}.
\end{equation}
The first term in \eqref{eqn:intro-magnetoelastic-energy} stands for the elastic energy of the system. The expression of the nonlinear elastic energy density $W$ should exhibit a strong coupling between its two variables in order to  enhance magnetostrictive effects. A prototypical example inspired by the theory of liquid crystals \cite{desimone.teresi} is given by
\begin{equation*}
    W(\boldsymbol{F},\boldsymbol{\lambda})\coloneqq \Phi\big ((\alpha \boldsymbol{\lambda}(\det\boldsymbol{F})\otimes \boldsymbol{\lambda}(\det\boldsymbol{F})+\beta (\boldsymbol{I}-\boldsymbol{\lambda}(\det\boldsymbol{F})\otimes \boldsymbol{\lambda}(\det\boldsymbol{F})))^{-1}\boldsymbol{F}\big ), 
\end{equation*}
where $\Phi$ is a frame-indifferent density that is minimized at the identity and that blows up as the determinant of its argument approaches zero and $\alpha,\beta>0$ are material parameters. However, in the present work, no specific structure of $W$ will be assumed. The second term in \eqref{eqn:intro-magnetoelastic-energy} is the exchange energy determined by the pairwise interaction of magnetic dipoles, which favours their alignment. Eventually, the third term in \eqref{eqn:intro-magnetoelastic-energy} accounts for the magnetostatic energy; the stray-field potential $u_{\boldsymbol{m}}\colon \R^N \to \R$ is given by a solution to the Maxwell equation
\begin{equation}
\label{eqn:intro-maxwell}
    \text{$\div \left (-Du_{\boldsymbol{m}}-\chi_{\boldsymbol{y}(\Omega)}\boldsymbol{m} \right)=0$ in $\R^N$,}
\end{equation}
where $\chi_{\boldsymbol{y}(\Omega)}\boldsymbol{m}$ simply denotes the extension of $\boldsymbol{m}$ by zero outside of $\boldsymbol{y}(\Omega)$. 

In the case of rigid bodies at sufficiently low constant temperature, magnetizations are subjected to a magnetic saturation constraint which, up to normalization, requires them to be sphere-valued. Indeed, the modulus of magnetizations corresponds to the spontaneous magnetic moment per unit mass that, for simplicity, is assumed to depend only on the temperature. In the case of deformable bodies, especially at large strains, this observation leads to a reformulation of the magnetic saturation constraint. In this regard, employing our notation, Brown writes \cite[p.\,73]{brown}: \emph{``since perfect alignment of spins produces a definite magnetic moment per unit mass, not per unit volume, it is $|\boldsymbol{m}\circ \boldsymbol{y}|\det D \boldsymbol{y}$ and not $|\boldsymbol{m}|$ that must be supposed constant’’.} Therefore, the magnetic saturation constraint should take the following form:
\begin{equation}
    \label{eqn:intro-magnetic-saturation}
    \text{$|\boldsymbol{m}\circ \boldsymbol{y}|\det D \boldsymbol{y}=1$ in $\Omega$.}
\end{equation}
Trivially, for incompressible materials, the constraint \eqref{eqn:intro-magnetic-saturation} is equivalent to the requirement of magnetizations being sphere-valued, namely
\begin{equation}
    \label{eqn:sphere-valued}
    \text{$|\boldsymbol{m}|=1$ in $\boldsymbol{y}(\Omega)$.}
\end{equation}

In the present work, we investigate the existence of solutions for the variational model of Brown both in the static and in the quasistatic setting. In the first case, solutions corresponds to minimizers of the energy \eqref{eqn:intro-magnetoelastic-energy} subjected to \eqref{eqn:intro-maxwell}-\eqref{eqn:intro-magnetic-saturation}. In the second case, these are understood in the energetic sense according to the theory of rate-independent processes \cite{mielke.roubicek}. This framework appears to be  adequate in view the  hysteretic character of magnetostrictive phenomena \cite{davino.visone}.

Before presenting our results, we briefly review the most relevant literature. For the static problem, the existence of minimizers has been first established in \cite{rybka.luskin} for nonsimple materials. The case of simple and incompressible materials has been addressed in \cite{barchiesi.desimone} and \cite{kruzik.stefanelli.zeman} under critical and supercritical coercivity assumptions on the elastic energy density, respectively. These assumptions ensure the continuity of admissible deformations.
Subsequently, the existence of minimizers  for compressible materials under subcritical coercivity assumptions has been proved in \cite{barchiesi.henao.moracorral} by considering a suitable class of of possibly discontinuous admissible deformations excluding cavitation.
The results in \cite{barchiesi.henao.moracorral}  have been further extended in \cite{henao.stroffolini}, by enlarging the class of admissible deformations to the scale of Sobolev-Orlicz spaces. A different proof of the existence of minimizers for compressible materials under supercritical coercivity assumptions has been recently provided in \cite{bresciani.davoli.kruzik}.
In the quasistatic setting, few results are available only in the case  of  continuous deformations. The existence of energetic solutions has been first achieved in  \cite{kruzik.stefanelli.zeman} for incompressible materials. Then, this result has been extended to compressible materials in \cite{bresciani.davoli.kruzik}. The analysis in \cite{bresciani.davoli.kruzik} actually contemplates a different notion of dissipation with respect to the one in \cite{kruzik.stefanelli.zeman} and relies on some higher-order regularization of the energy; however, the arguments in \cite{bresciani.davoli.kruzik} ensure the existence of energetic solutions for compressible materials in absence of any regularization when the dissipation takes the same form as  in \cite{kruzik.stefanelli.zeman}. We observe that in all the contributions mentioned, except for \cite{rybka.luskin},  the magnetic saturation constraint takes the form \eqref{eqn:sphere-valued} rather than  \eqref{eqn:intro-magnetic-saturation}.

It is worth to notice that, apart for the sake of mathematical generality,  relaxing the coercivity assumptions on the elastic energy density is extremely important from the physical point of view; indeed, this makes the analysis compatible with a broader range of mechanical models acknowledged in the literature in which the elastic energy density has subcritical, even quadratic, growth.   

The aim of this paper is to extend the analysis in \cite{barchiesi.henao.moracorral} from the static setting to the quasistatic one by establishing the existence of energetic solutions under subcritical coercivity assumptions on the elastic energy density. Admissible deformations are supposed to belong to the same class of  maps in $W^{1,p}(\Omega;\R^N)$ with $p>N-1$ for which cavitation is excluded which has been proposed in \cite{barchiesi.henao.moracorral}. Such maps represent the counterpart of orientation-preserving local homeomorphisms in the context of Sobolev spaces. Moreover, in the present work, we treat the more realistic constraint \eqref{eqn:intro-magnetic-saturation}. This constitutes a difference with respect to the analysis in \cite{barchiesi.henao.moracorral}, where the constraint \eqref{eqn:sphere-valued} is imposed.

To accomplish our goal, we first sharpen the analysis of the variational model in the static setting. We provide a compactness result for sequences of admissible states with uniformly bounded magnetoelastic energy. The limiting states are shown to satisfy the constraint \eqref{eqn:intro-magnetic-saturation}.
In particular, our compactness result yields the convergence of the compositions of magnetizations with deformations to the corresponding limiting quantity. This fact constitutes a very delicate issue as both deformations and magnetizations are generally discontinuous.
The convergence of compositions allows us to deduce the lower semicontinuity of the elastic energy, which represents the problematic term, by means of a standard application of the classical Eisen Selection Lemma \cite{eisen}. From this, the existence of minimizers of the magnetoelastic energy under the constraint \eqref{eqn:intro-magnetic-saturation} is achieved by employing the Direct Method as in \cite{kruzik.stefanelli.zeman}.
In \cite{barchiesi.henao.moracorral}, where the convergence of compositions  had  not been proved, the lower semicontinuity of the elastic energy  had  been established by working on the deformed configuration exploiting the convergence of the Jacobian minors of inverse deformations. Compared with the  arguments  in \cite{barchiesi.henao.moracorral}, our  proof strategy seems to be more direct.

Subsequently, we study  quasistatic evolutions driven by time-dependent boundary conditions and applied loads. These include mechanical body and surface forces, but also external magnetic fields, whose energetic contribution is of Eulerian type. 
Since the mechanical response of the material is assumed to be purely elastic, energy dissipation results from the magnetic reorientation only. In view of our mixed Eulerian-Lagrangian structure and the expression of the elastic energy in \eqref{eqn:intro-magnetoelastic-energy}, the natural candidate as dissipative variable is given by the composition of  magnetization and deformation. The same modeling choice has been taken in \cite{kruzik.stefanelli.zeman}. Therefore, the convergence of compositions provided by our compactness result is essential to address  the analysis of the quasistatic model. As already mentioned, in this setting, we seek for solutions in the energetic sense. These, we recall, are formulated on the two principles of global miniminality and energy-dissipation balance.
Thanks to our compactness result, the dissipation distance turns out to be continuous on the sublevel sets of the total energy. Hence, existence of energetic solutions is established by following the standard time-discretization scheme \cite{francfort.mielke,mainik.mielke,mielke,mielke.theil,mielke.theil.levitas,mielke.roubicek}. Minor adaptations, similar to the ones in \cite{bresciani.davoli.kruzik},  are needed to show the compactness of time-discrete solutions.

We remark that our model contemplates deformations that are locally invertible but not necessarily globally invertible. However, given the significance of this last requirement from the modeling point of view, we show how to incorporate the global invertibiliy constraint into our analysis. This is done in a fairly simple way. Our approach is based on the Ciarlet-Ne\v{c}as condition \cite{ciarlet.necas}, see also \cite{giacomini.ponsiglione,qi}. However, other equivalent possibilities, such as the (INV) condition \cite{mueller.spector}, can be considered.

Our analysis builds upon the one in \cite{barchiesi.henao.moracorral}, which, in turn, is based on previous works \cite{mueller.spector,mueller.qi.spector,mueller.qi.yan,sverak}.  A central role in our arguments is played by  the notions of geometric and topological image. The former allows us to employ the Change-of-variable Formula  from the deformed configuration to the reference one and backwards, whereas  the latter enjoys some continuity properties with respect to the convergence of deformations. In this regard, the topological degree constitutes an indispensable tool. The local invertibility of deformations as well as the stability of the local inverse with respect to the weak convergence, both established in \cite{barchiesi.henao.moracorral}, are also essential ingredients in our proofs.  Key difficulties in the analysis are the possible lack of continuity of admissible deformations together with the fact that these may violate the Lusin condition (N), i.e. may map sets of Lebesgue measure zero to set of positive Lebesgue measure. To overcome them, we exploit the fine properties of admissible deformations and we make use of refined versions of the Change-of-variable Formula. 

From the technical point of view, the main novelty of the paper lies in the proof of the compactness result, namely Theorem \ref{thm:compactness}. The most delicate points are to show that the constraint \eqref{eqn:intro-magnetic-saturation} is preserved in the limit and to prove the convergence of the compositions of magnetizations with deformations. The strategy to tackle these issues  is based on a localization argument involving the topological image of nested balls. Given a converging sequence of deformations, this allows us to consider the sequence of the local inverses on the topological image of a fixed ball under the limiting deformation. This argument is presented in Lemma \ref{lem:topological-images-nested}. This last result relies on the fine properties of  admissible deformations, precisely on their regular approximate differentiability \cite{goffman.ziemer}.

We expect that our results can be extended to the case in which admissible deformations belong to suitable Sobolev-Orlicz spaces with the help of the techniques developed in \cite{henao.stroffolini}. A possible variant of our analysis consists in taking the dissipation distance to be of the form proposed in \cite{bresciani.davoli.kruzik}. In that case, this is defined by means of a pull-back of magnetizations to the reference configuration involving the deformation gradient; this modeling choice has the advantage of producing a frame-indifferent dissipation which should be more selective among all possible evolutions.  However, this choice seem to require some higher-order regularization of the energy in order to proceed with the analysis. We refer to \cite{bresciani.davoli.kruzik} for more details.

The outline of the paper is as follows. In Section \ref{sec:preliminaries}, we present some preliminary results from the literature, mostly taken from \cite{barchiesi.henao.moracorral}, and we complement those with some simple observations. In Section \ref{sec:static}, we address the analysis in the static setting; the main results of the section are Theorem \ref{thm:compactness}, our compactness result, and Theorem \ref{thm:existence-minimizers}, the existence of minimizers. Quasistatic evolutions are studied in Section \ref{sec:quasistatic} and the existence of energetic solutions is established in Theorem \ref{thm:existence-energetic-solution}. Eventually, in the Appendix, we recall the definition about Sobolev maps on the boundary of regular domains.

\section{Preliminaries}
\label{sec:preliminaries}
In this section we recall various preliminary notions and results. Most of those are taken  from \cite{barchiesi.henao.moracorral}.  

Henceforth, $\Omega \subset \R^N$ is a bounded Lipschitz domain. The spatial dimension $N \in \N$ satisfies $N\geq 2$.  Generic points in the reference space and in the actual space are  denoted by $\boldsymbol{x}$ and $\boldsymbol{\xi}$, respectively. Accordingly, integration with respect to  $\leb$ is indicated by $\d \boldsymbol{x}$ or $\d\boldsymbol{\xi}$, while integration with respect to $\haus$ in the reference space is indicated by $\d\boldsymbol{a}$. Here, $\mathscr{L}^m$ and $\mathscr{H}^s$ denote the $m$-dimensional Lebesgue measure and the $s$-dimensional Hausdorff measure, where $m\in \N$ and $s>0$. Expressions like ``almost everywhere'' or ``almost every'' are referred to one of these measures depending on the context without any explicit specification. We do not identify maps that coincide almost everywhere. The exponent $p>N-1$ is fixed and we denote by $p'\coloneqq p/(p-1)$ its conjugate exponent.

A measurable map $\boldsymbol{y}\colon \Omega \to \R^N$ is termed to have the Lusin property (N) if $\leb(\boldsymbol{y}(A))=0$ for every $A\subset \Omega$ with $\leb(A)=0$, whereas it is termed to have the Lusin property (N${}^{-1}$) if $\leb(\boldsymbol{y}^{-1}(B))=0$ for every $B\subset \R^N$ with $\leb(B)=0$. For every set $U\subset \subset \Omega$ of class $C^2$, we denote by $\boldsymbol{n}_U$ the outer unit normal on $\partial U$.

\subsection{Lebesgue points and precise representative} Let $\boldsymbol{y}\in W^{1,p}(\Omega;\R^N)$. We denote by $L_{\boldsymbol{y}}$ the set of Lebesgue points of $\boldsymbol{y}$. Namely, $\boldsymbol{x}_0\in L_{\boldsymbol{y}}$ if there exists $\boldsymbol{y}^*(\boldsymbol{x}_0)\in\R^N$ such that
\begin{equation*}
    \lim_{r \to 0^+} \dashint_{B(\boldsymbol{x}_0,r)}|\boldsymbol{y}(\boldsymbol{x})-\boldsymbol{y}^*(\boldsymbol{x}_0)|\,\d\boldsymbol{x}=0.
\end{equation*}
In this case, $\boldsymbol{y}^*(\boldsymbol{x}_0)$ is termed the Lebesgue value of $\boldsymbol{y}$ at $\boldsymbol{x}_0$ and it is characterized as
\begin{equation*}
    \boldsymbol{y}^*(\boldsymbol{x}_0)=\lim_{r \to 0^+} \dashint_{B(\boldsymbol{x}_0,r)} \boldsymbol{y}(\boldsymbol{x})\,\d\boldsymbol{x}.
\end{equation*}
It is proved that $L_{\boldsymbol{y}}\subset \Omega$ is a Borel set and  $\leb(\Omega\setminus L_{\boldsymbol{y}})=0$. Actually, we have a precise estimate for the Hausdorff dimension of this set. For the following result, we refer to \cite[Theorem 2, Section 1.1, Chapter 3]{cartesian.currents}.

\begin{proposition}[Exceptional points]
	\label{prop:exceptional-points}
	Let $\boldsymbol{y}$ belong to $W^{1,p}(\Omega;\R^N)$. Then, for every $s>N-p$, there holds $\mathscr{H}^s(\Omega\setminus L_{\boldsymbol{y}})=0$. In particular, we have $\mathscr{H}^1(\Omega\setminus L_{\boldsymbol{y}})=0$.
\end{proposition} 

We will use the following consequence of the previous result.

\begin{corollary}
	\label{cor:lebesgue-points-balls}
	Let $\boldsymbol{y}\in W^{1,p}(\Omega;\R^N)$. Then, for every $\boldsymbol{x}_0\in \Omega$ and for almost every $r\in (0,\dist(\boldsymbol{x}_0;\partial \Omega))$, there holds $\partial B(\boldsymbol{x}_0,r)\subset L_{\boldsymbol{y}}$.
\end{corollary}
\begin{proof}
	For simplicity, set $\overline{r}\coloneqq \dist(\boldsymbol{x}_0;\partial \Omega)$. 	
	By \cite[Corollary 2.10.11]{federer}, we have
	\begin{equation*}
	\int_0^{\overline{r}} \# (\partial B(\boldsymbol{x}_0,r) \setminus L_{\boldsymbol{y}})\,\d r\leq \mathscr{H}^1(B(\boldsymbol{x}_0,\overline{r})\setminus L_{\boldsymbol{y}}),
	\end{equation*}
	where the right-hand side equals zero by Proposition \ref{prop:exceptional-points}. Therefore, we have $\# (\partial B(\boldsymbol{x}_0,r) \setminus L_{\boldsymbol{y}})=0$ for almost every $0<r<\overline{r}$, which proves the claim.   
\end{proof}

 Note that the set $L_{\boldsymbol{y}}$ as well as the Lebesgue values of $\boldsymbol{y}$ depend only on the equivalence class of $\boldsymbol{y}$. We will regard $\boldsymbol{y}^*$ as a map from $L_{\boldsymbol{y}}$ to $\R^N$, termed the precise representative of $\boldsymbol{y}$. In this case, it is shown that $\boldsymbol{y}=\boldsymbol{y}^*$ almost everywhere in $L_{\boldsymbol{y}}$. 
 
\subsection{Approximate differentiabiliy and Change-of-variable Formula.} \label{subsec:approximate-differentiability-geometric-image}
For the notions of density and approximate differentiability, we refer to \cite{cartesian.currents}.
The $N$-dimensional density of a measurable set $A\subset \R^N$ at  $\boldsymbol{x}_0\in \R^N$ is denoted by $\Theta^N(A,\boldsymbol{x}_0)$.
The right one-dimensional densities of a measurable set $A\subset \R$ at $x_0\in \R$ is defined as
\begin{equation*}
    \Theta^1_+(A,x_0)\coloneqq \lim_{r \to 0^+} \frac{\mathscr{L}^1(A\cap (x_0,x_0+r))}{r}.
\end{equation*}

For $\boldsymbol{y}\in W^{1,p}(\Omega;\R^N)$, we denote the weak gradient of $\boldsymbol{y}$ by $D\boldsymbol{y}$. The map $\boldsymbol{y}$ is almost everywhere approximately differentiable \cite[Theorem 2, Section 1.4, Chapter 3]{cartesian.currents} and we denote its approximate gradient by $\nabla \boldsymbol{y}$. In particular, there holds $\nabla \boldsymbol{y}=D\boldsymbol{y}$ almost everywhere.

\begin{definition}[Geometric domain and geometric image]
\label{def:geometric-image}
Let $\boldsymbol{y}\colon \Omega \to \R^N$ be almost everywhere approximately differentiable in $\Omega$ with $\det \nabla\boldsymbol{y}\neq 0$ almost everywhere. We define the geometric domain of $\boldsymbol{y}$ as the set $\dom_{\rm G}(\boldsymbol{y},\Omega)$ of points $\boldsymbol{x}_0\in \Omega$ such that $\boldsymbol{y}$ is approximately differentiable in $\boldsymbol{x}_0$ with $\det\nabla\boldsymbol{y}(\boldsymbol{x}_0)\neq 0$ and there exist a compact set $K\subset \Omega$ with $\boldsymbol{x}_0\in K$ and $\Theta^N(K,\boldsymbol{x}_0)=1$, and a map $\boldsymbol{w}\in C^1(\R^N;\R^N)$ satisfying $\boldsymbol{w}\restr{K}=\boldsymbol{y}\restr{K}$ and $D\boldsymbol{w}\restr{K}=\nabla\boldsymbol{y}\restr{K}$. For every measurable set $A\subset \Omega$, we set
\begin{equation*}
	\dom_{\rm G}(\boldsymbol{y},A)\coloneqq A \cap \dom_{\rm G}(\boldsymbol{y},\Omega)
\end{equation*}
and we define
the geometric image of $A$ under $\boldsymbol{y}$  as
\begin{equation*}
    \im_{\rm G}(\boldsymbol{y},A)\coloneqq \boldsymbol{y}(A \cap \dom_{\rm G}(\boldsymbol{y},\Omega)).
\end{equation*}
\end{definition}

Note that the definition of geometric domain might be slightly different compared with others available in the literature. In particular, additionally to $\Theta^N(K,\boldsymbol{x}_0)=1$, we also require that $\boldsymbol{x}_0\in K$. This requirement is not restrictive and will be exploited in  Lemma \ref{lem:inverse-approximately-differentiable}.

In the following result, we collect some basic properties of the geometric domain and the geometric image.

\begin{lemma}
\label{lem:geometric-image}
Let $\boldsymbol{y}\colon \Omega \to \R^N$ be almost everywhere approximately differentiable with $\det \nabla \boldsymbol{y}\neq 0$ almost everywhere and let $A\subset \Omega$ be measurable. Then, the following hold:
\begin{enumerate}[(i)]
    \item $\dom_{\rm G}(\boldsymbol{y},A)$ is measurable and $\leb(A \setminus \dom_{\rm G}(\boldsymbol{y},A))=0$;
    \item $\boldsymbol{y}\restr{\dom_{\rm G}(\boldsymbol{y},A)}$ has the Lusin property \emph{(N)} and the set  $\im_{\rm G}(\boldsymbol{y},A)$ is measurable;
    \item for every $\boldsymbol{x}_0\in \Omega$  such that $\Theta^N(A,\boldsymbol{x}_0)=1$, there holds $\Theta^N(\im_{\rm G}(\boldsymbol{y},A),\boldsymbol{y}(\boldsymbol{x}_0))=1.$
\end{enumerate}
\end{lemma}
\begin{proof}
Denote by $\Omega_{\rm d}(\boldsymbol{y})$ the set of points where $\boldsymbol{y}$ is approximately differentiable. Then
\begin{equation*}
	\leb(\Omega \setminus \Omega_{\rm d}(\boldsymbol{y}))=0,
\end{equation*}
while by \cite[Lemma 2.5]{mueller.spector} we have that 
\begin{equation*}
\leb(\Omega_{\rm d}(\boldsymbol{y})\setminus \dom_{\rm G}(\boldsymbol{y},\Omega))=0.	
\end{equation*}
Thus claim (i) follows. For (ii), the Lusin property (N) follows from the fact that the same property holds for  $\boldsymbol{y}\restr{\Omega_{\mathrm{d}}(\boldsymbol{y})}$, see \cite[Proposition 1, Section 1.5, Chapter 3]{cartesian.currents}. Also, this property entails the measurability of the geometric image  by \cite[Theorem 3.6.9]{bogachev}. Claim (iii) is proved in \cite[Lemma 2.5]{mueller.spector}.
\end{proof}

The geometric domain and, in turn, the geometric image clearly depend on the representative of $\boldsymbol{y}$. However, if $\boldsymbol{y}=\widetilde{\boldsymbol{y}}$ almost everywhere, then 
\begin{equation*}
	\leb(\dom_{\rm G}(\boldsymbol{y},A)\, \triangle \, \dom_{\rm G}(\widetilde{\boldsymbol{y}},A))=0, \qquad \leb(\im_{\rm G}(\boldsymbol{y},A)\, \triangle \, \im_{\rm G}(\widetilde{\boldsymbol{y}},A))=0,
\end{equation*}
thanks to claims (i)--(ii) of Lemma \ref{lem:geometric-image}.

We present the Change-of-variable Formula in the following form. 

\begin{proposition}[Change-of-variable Formula]
\label{prop:change-of-variable}
Let $\boldsymbol{y}\colon \Omega \to \R^N$ be almost everywhere approximately differentiable with $\det \nabla\boldsymbol{y}\neq 0$ almost everywhere. For every measurable set $A\subset \Omega$, the multiplicity function $\mult(\boldsymbol{y},A,\cdot)\colon \R^N \to \N \cup \{0\}$ defined by
\begin{equation*}
\mult(\boldsymbol{y},A,\boldsymbol{\xi})\coloneqq \# \{\boldsymbol{x}\in \dom_{\rm G}(\boldsymbol{y},A):\:\boldsymbol{y}(\boldsymbol{x})=\boldsymbol{\xi}\}    
\end{equation*}
is measurable. Also, for every measurable function $\boldsymbol{\psi}\colon \im_{\rm G}(\boldsymbol{y},A)\to\R^M$, there holds
\begin{equation*}
    \int_A \boldsymbol{\psi}\circ \boldsymbol{y}(\boldsymbol{x})\,\det\nabla \boldsymbol{y}(\boldsymbol{x})\,\d\boldsymbol{x}=\int_{\im_{\rm G}(\boldsymbol{y},A)} \boldsymbol{\psi}(\boldsymbol{\xi})\,\mult(\boldsymbol{y},A,\boldsymbol{\xi})\,\d\boldsymbol{\xi}
\end{equation*}
whenever one of the two integrals exists.
\end{proposition}
Note that, in the previous formula, the map $\boldsymbol{y}$ has the Lusin propety (N${}^{-1}$) \cite[Lemma 2.8, Claim (c)]{barchiesi.henao.moracorral}, so that the composition $\boldsymbol{\psi}\circ \boldsymbol{y}$ is measurable \cite[Lemma 2.9]{barchiesi.henao.moracorral}.
\begin{proof}
Again, denote by $\Omega_{\rm d}(\boldsymbol{y})$ the set of points of approximate differentiability of $\boldsymbol{y}$.	
By the Federer Change-of-variable Formula \cite[Theorem 1, Section 1.5, Chapter 3]{cartesian.currents}, we have 
\begin{equation*}
\int_A \boldsymbol{\psi}\circ \boldsymbol{y}(\boldsymbol{x})\,\det\nabla \boldsymbol{y}(\boldsymbol{x})\,\d\boldsymbol{x}=\int_{\boldsymbol{y}(A \cap \Omega_{\rm d}(\boldsymbol{y}))} \boldsymbol{\psi}(\boldsymbol{\xi})\, \#\{\boldsymbol{x}\in A \cap \Omega_{\rm d}(\boldsymbol{y}):\,\boldsymbol{y}(\boldsymbol{x})=\boldsymbol{\xi}\}\,\d\boldsymbol{\xi}.
\end{equation*}
First, recall that $\boldsymbol{y}\restr{\Omega_{\rm d}(\boldsymbol{y})}$ has the Lusin property (N). Thus $$\leb(\boldsymbol{y}(A \cap \Omega_{\mathrm{d}}(\boldsymbol{y})))=\leb(\im_{\rm G}(\boldsymbol{y},A))$$ since $\leb(\Omega_{\mathrm{d}}(\boldsymbol{y}) \setminus \dom_{\rm G}(\boldsymbol{y},\Omega))=0$. Second, note that 
\begin{equation*}
\#\{\boldsymbol{x}\in A \cap \Omega_{\rm d}(\boldsymbol{y}):\,\boldsymbol{y}(\boldsymbol{x})=\boldsymbol{\xi}\}=\mult(\boldsymbol{y},A,\boldsymbol{\xi})
\end{equation*}
for every $\boldsymbol{\xi}\in \im_{\rm G}(\boldsymbol{y},A)\setminus \boldsymbol{y}((A\cap \Omega_{\mathrm{d}}(\boldsymbol{y}))\setminus \dom_{\rm G}(\boldsymbol{y},A))$. Hence, the desired formula follows in view of these two observations.
\end{proof}

We recall the notion of regular approximate differentiability. 

\begin{definition}[Regular approximate differentiability]
\label{def:regular-approximate-differentiability}
Let $\boldsymbol{y}\colon \Omega \to \R^N$ be measurable and let $\boldsymbol{x}_0\in \Omega$. The map $\boldsymbol{y}$ is termed regularly approximately differentiable  at $\boldsymbol{x}_0$ if there exists $\nabla \boldsymbol{y}(\boldsymbol{x}_0)\in \rnn$ such that
\begin{equation}
\label{eqn:regular-approximate-differentiability-aplim}
    \aplim_{r \to 0^+} \esssup_{\boldsymbol{x}\in \partial B(\boldsymbol{x}_0,r)} \frac{|\boldsymbol{y}(\boldsymbol{x})-\boldsymbol{y}(\boldsymbol{x}_0)- \nabla \boldsymbol{y}(\boldsymbol{x}_0)(\boldsymbol{x}-\boldsymbol{x}_0)|}{|\boldsymbol{x}-\boldsymbol{x}_0|}=0.
\end{equation}
In this case,  $\nabla \boldsymbol{y}(\boldsymbol{x}_0)$  is termed the regular approximate gradient of $\boldsymbol{y}$ at $\boldsymbol{x}_0$.
\end{definition}

In the previous definition, \eqref{eqn:regular-approximate-differentiability-aplim} requires the existence of a measurable set $R \subset (0,\dist(\boldsymbol{x}_0;\partial \Omega))$ with $\Theta^1_+(R,0)=1$ such that
\begin{equation}
\label{eqn:regular-approximate-differentiability-lim}
    \lim_{\substack{r \to 0 \\ r\in R}} \esssup_{\boldsymbol{x}\in \partial B(\boldsymbol{x}_0,r)} \frac{|\boldsymbol{y}(\boldsymbol{x})-\boldsymbol{y}(\boldsymbol{x}_0)-\nabla \boldsymbol{y}(\boldsymbol{x}_0)(\boldsymbol{x}-\boldsymbol{x}_0)|}{|\boldsymbol{x}-\boldsymbol{x}_0|}=0.
\end{equation}
Here, the essential supremum is meant with respect to the $(N-1)$-dimensional Hausdorff measure.

The regular approximate gradient is uniquely defined whenever it exists. Also, in the assumptions of Definition \ref{def:regular-approximate-differentiability}, $\nabla \boldsymbol{y}(\boldsymbol{x}_0)$ coincides with the approximate gradient  of $\boldsymbol{y}$ at $\boldsymbol{x}_0$, which justifies the notation. To see this, for every $r \in R$, let $T_r\subset \partial B(\boldsymbol{x}_0,r)$ with $\mathscr{H}^{N-1}(T_r)=0$ be such that 
\begin{equation*}
    \esssup_{\boldsymbol{x}\in \partial B(\boldsymbol{x}_0,r)} \frac{|\boldsymbol{y}(\boldsymbol{x})-\boldsymbol{y}(\boldsymbol{x}_0)-\nabla \boldsymbol{y}(\boldsymbol{x}_0)(\boldsymbol{x}-\boldsymbol{x}_0)|}{|\boldsymbol{x}-\boldsymbol{x}_0|}=\sup_{\boldsymbol{x}\in \partial B(\boldsymbol{x}_0,r)\setminus T_r}\frac{|\boldsymbol{y}(\boldsymbol{x})-\boldsymbol{y}(\boldsymbol{x}_0)-\nabla \boldsymbol{y}(\boldsymbol{x}_0)(\boldsymbol{x}-\boldsymbol{x}_0)|}{|\boldsymbol{x}-\boldsymbol{x}_0|}
\end{equation*}
and set $A\coloneqq \bigcup \left \{\partial B(\boldsymbol{x}_0,r)\setminus T_r:\:r \in R \right \}$. Then, $\Theta^N(A,\boldsymbol{x}_0)=1$ and $\boldsymbol{y}\restr{A}$ is differentiable at $\boldsymbol{x}_0$, which proves the claim.

Note that Definition \ref{def:regular-approximate-differentiability} is slightly different from the definition of regular approximate differentiability available in the literature, see \cite[Section 3]{goffman.ziemer} or \cite[Definition 8.2]{mueller.spector}. Specifically, in \eqref{eqn:regular-approximate-differentiability-aplim}, we use the essential supremum instead of the supremum. With this choice, as stated in the following result, any map in $W^{1,p}(\Omega;\R^N)$ is almost everywhere regularly approximately differentiable independently of the representative chosen. 

\begin{proposition}
\label{prop:regular-approximate-differentiability}
Let $\boldsymbol{y}\in W^{1,p}(\Omega;\R^N)$. Then, 
$\boldsymbol{y}$ is almost everywhere regularly approximately differentiable with $\nabla \boldsymbol{y}=D\boldsymbol{y}$ almost everywhere.
\end{proposition}
\begin{proof}
Recall Corollary \ref{cor:lebesgue-points-balls}. By \cite[Section 3]{goffman.ziemer}, for almost every $\boldsymbol{x}_0 \in L_{\boldsymbol{y}}$, there exists $R^* \subset (0,\dist(\boldsymbol{x}_0;\partial \Omega))$ with $\Theta^1_+(R^*,0)=1$ such that $\partial B(\boldsymbol{x}_0,r)\subset L_{\boldsymbol{y}}$ for every $r \in R^*$ and there holds
\begin{equation*}
    \lim_{\substack{r \to 0 \\ r\in R^*}} \sup_{\boldsymbol{x}\in \partial B(\boldsymbol{x}_0,r)} \frac{|\boldsymbol{y}^*(\boldsymbol{x})-\boldsymbol{y}^*(\boldsymbol{x}_0)-\nabla \boldsymbol{y}^*(\boldsymbol{x}_0)(\boldsymbol{x}-\boldsymbol{x}_0)|}{|\boldsymbol{x}-\boldsymbol{x}_0|}=0.
\end{equation*}
Set $R\coloneqq \{r \in R^*:\:\text{$\boldsymbol{y}=\boldsymbol{y}^*$ a.e. on $\partial B(\boldsymbol{x}_0,r)$}\}$. In this case, $\mathscr{L}^1(R)=\mathscr{L}^1(R^*)$ and, in turn, $\Theta^1_+(R,0)=\Theta^1_+(R^*,0)=1$. Also, if $\boldsymbol{y}(\boldsymbol{x}_0)=\boldsymbol{y}^*(\boldsymbol{x}_0)$, then we have
\begin{equation*}
\begin{split}
    \lim_{\substack{r \to 0 \\ r \in R}} \esssup_{\boldsymbol{x}\in \partial B(\boldsymbol{x}_0,r)} &\frac{|\boldsymbol{y}(\boldsymbol{x})-\boldsymbol{y}(\boldsymbol{x}_0)-\nabla\boldsymbol{y}^*(\boldsymbol{x}_0)(\boldsymbol{x}-\boldsymbol{x}_0)|}{|\boldsymbol{x}-\boldsymbol{x}_0|}\\
    &\leq \lim_{\substack{r \to 0 \\ r \in R}}  \sup_{\boldsymbol{x}\in \partial B(\boldsymbol{x}_0,r)}  \frac{|\boldsymbol{y}^*(\boldsymbol{x})-\boldsymbol{y}^*(\boldsymbol{x}_0)-\nabla \boldsymbol{y}^*(\boldsymbol{x}_0)(\boldsymbol{x}-\boldsymbol{x}_0)|}{|\boldsymbol{x}-\boldsymbol{x}_0|}=0.
\end{split}
\end{equation*}
Thus, $\boldsymbol{y}$ is regularly approximately differentiable at $\boldsymbol{x}_0$ with $\nabla \boldsymbol{y}(\boldsymbol{x}_0)=\nabla\boldsymbol{y}^*(\boldsymbol{x}_0)$. As 
$$\leb(\Omega \setminus \{\boldsymbol{x}_0 \in L_{\boldsymbol{y}}:\:\boldsymbol{y}(\boldsymbol{x}_0)=\boldsymbol{y}^*(\boldsymbol{x}_0),\;D\boldsymbol{y}(\boldsymbol{x}_0)=\nabla\boldsymbol{y}^*(\boldsymbol{x}_0)\})=0,$$ this concludes the proof.
\end{proof}

\subsection{Topological degree.} For the definition and the main properties of the topological degree of a continuous map defined on the closure of a bounded open set, we refer to \cite{fonseca.gangbo}.
Let $U\subset \subset \Omega$ be open and $\boldsymbol{y}\in C^0(\partial U;\R^N)$. The topological degree of $\boldsymbol{y}$ on $U$ is defined as the topological degree of any extension ${\boldsymbol{Y}}\in C^0(\closure{U};\R^N)$ of $\boldsymbol{y}$. This definition is well posed. Indeed, such an extension always exists by Tiezte Theorem and the topological degree of ${\boldsymbol{Y}}$ on $U$ depends only on ${\boldsymbol{Y}}\restr{\partial U}$. In particular, the map $\deg(\boldsymbol{y},U,\cdot)\colon \R^N \setminus \boldsymbol{y}(\partial U)\to \Z$ is continuous and we define the topological image of $U$ under $\boldsymbol{y}$ as
\begin{equation*}
    \im_{\rm T}(\boldsymbol{y},U)\coloneqq \left \{\boldsymbol{\xi}\in \R^N \setminus \boldsymbol{y}(\partial U):\:\deg(\boldsymbol{y},U,\boldsymbol{\xi})\neq 0 \right \}.
\end{equation*}
This set is bounded since $\deg(\boldsymbol{y},U,\cdot)=0$ on the unique unbounded component of $\R^N \setminus \boldsymbol{y}(\partial U)$. Moreover, by continuity, $\im_{\rm T}(\boldsymbol{y},U)$ is open and $\partial\, \im_{\rm T}(\boldsymbol{y},U)\subset \boldsymbol{y}(\partial U)$. 

In the next two lemmas, we collect some elementary observations.

\begin{lemma}
\label{lem:topological-image-contained}
Let $U\subset \subset \Omega$ be open  and let $\boldsymbol{y}\in C^0(\partial U;\R^N)$. Suppose that $\boldsymbol{y}(\partial U)\subset O$ for some  $O \subset \subset \R^N$ open with $\R^N \setminus O$ connected. Then, $\im_{\rm T}(\boldsymbol{y},U)\subset \subset  O$.
\end{lemma}
\begin{proof}
Since $\R^N \setminus O \subset \R^N \setminus \boldsymbol{y}(\partial U)$ is connected and unbounded,  there holds $\R^N \setminus O \subset V$, where $V$ is the unique unbounded component of $\R^N \setminus \boldsymbol{y}(\partial U)$. Thus, $\im_{\rm T}(\boldsymbol{y},U)\subset \R^N \setminus V \subset O$.
\end{proof}

\begin{lemma}
\label{lem:connected}
Let $U\subset \subset \Omega$ be open  and let $\boldsymbol{y}\in C^0(\partial U;\R^N)$. Also, let $\{V_i:\:i \in \N\}$ be the collection of the connected components of  $\R^N \setminus \boldsymbol{y}(\partial U)$. Then:
\begin{enumerate}[(i)]
    \item for every $i \in \N$, there holds $\partial V_i \subset \boldsymbol{y}(\partial U)$;
    \item if $\partial U$ is connected, then, for every $I \subset \N$, the set $\bigcup_{i \in I} V_i \cup \boldsymbol{y}(\partial U)$ is connected and, in particular, $\im_{\rm T}(\boldsymbol{y},U)\cup \boldsymbol{y}(\partial U)$ and $\R^N \setminus \im_{\rm T}(\boldsymbol{y},U)$ are connected.
\end{enumerate}
\end{lemma}
\begin{proof}
To prove (i), suppose that there exists $\boldsymbol{\xi}\in \partial V_i \setminus \boldsymbol{y}(\partial U)$. Then, $\boldsymbol{\xi}\in V_j$ for some $j \in \N$ with $j \neq i$. As $V_j$ is open, $B(\boldsymbol{\xi},r)\subset V_j$ for some $r>0$. However, as $\boldsymbol{\xi}\in \partial V_i$, there holds $B(\boldsymbol{\xi},r)\cap V_i\neq \emptyset$ and, in turn, $V_i \cap V_j \neq \emptyset$. This provides a contradiction. To prove (ii), note that, if $\partial U$ is connected, then $\boldsymbol{y}(\partial U)$ is connected. For every $i \in \N$, the set $\closure{V}_i$ is connected being the closure of a connected set. By (i),  we have $\closure{V}_i \cap \boldsymbol{y}(\partial U)=\partial V_i$, so that $V_i \cup \boldsymbol{y}(\partial U)=\closure{V}_i \cup \boldsymbol{y}(\partial U)$ is connected being the union of two connected sets with nonempty intersection. For the same reason, given $I \subset \N$, the set $\bigcup_{i \in I} V_i \cup \boldsymbol{y}(\partial U)$ is connected since $\bigcap_{i \in I} (V_i \cap \boldsymbol{y}(\partial U))=\boldsymbol{y}(\partial U)$. 
\end{proof}

The continuity of the topological degree allows to describe the asymptotic behaviour of topological images of uniformly converging sequences of maps.

\begin{lemma}
\label{lem:topological-image-uniform-convergence}
Let $U \subset \subset \Omega$ be open and let $(\boldsymbol{y}_n)\subset C^0(\partial U;\R^N)$ and $\boldsymbol{y}\in C^0(\partial U;\R^N)$. Suppose that $\boldsymbol{y}_n \to \boldsymbol{y}$ uniformly on $\partial U$ and that $\leb(\boldsymbol{y}_n(\partial U))=\leb(\boldsymbol{y}(\partial U))=0$ for every $n \in \N$. Then:
\begin{enumerate}[(i)]
    \item for every $K\subset \im_{\rm T}(\boldsymbol{y},U)$ compact, there holds $K\subset \im_{\rm T}(\boldsymbol{y}_n,U)$ for $n \gg1$ depending on $K$;
    \item for every $K\subset \R^N \setminus ( \im_{\rm T}(\boldsymbol{y},U)\cup \boldsymbol{y}(\partial U))$ compact,  there holds $K\subset \R^N \setminus ( \im_{\rm T}(\boldsymbol{y}_n,U) \cup \boldsymbol{y}_n(\partial U))$ for $n \gg1$ depending on $K$;
    \item for every $O \subset \subset \R^N$ open with $\R^N \setminus O$ connected such that $\boldsymbol{y}(\partial U)\subset O$, there hold $\im_{\rm T}(\boldsymbol{y},U)\subset O$ and $\im_{\rm T}(\boldsymbol{y}_n,U)\cup \boldsymbol{y}(\partial U) \subset O$ for $n \gg 1$ depending on $O$;
    \item $\chi_{\im_{\rm T}(\boldsymbol{y}_n,U)} \to \chi_{\im_{\rm T}(\boldsymbol{y},U)}$ in $L^1(\R^N)$ and a.e. in $\R^N$.
\end{enumerate}
\end{lemma}
\begin{proof}
Claims (i)--(ii) are shown in \cite[Lemma 3.6]{barchiesi.henao.moracorral}. To prove (iii), observe that, by uniform convergence, $\boldsymbol{y}_n(\partial U)\subset O$ for  $n\gg 1$. Then, the claim follows by Lemma \ref{lem:topological-image-contained}. To prove (iv), note first that \begin{equation*}
    \text{$\chi_{\im_{\rm T}(\boldsymbol{y}_n,U)} \to \chi_{\im_{\rm T}(\boldsymbol{y},U)}$ a.e. in $\R^N\setminus \boldsymbol{y}(\partial U)$}
\end{equation*}
by (i)--(ii). Let $r>0$ be such that $\im_{\rm T}(\boldsymbol{y},U)\subset B(\boldsymbol{0},r)$. Then, by (iii), $\im_{\rm T}(\boldsymbol{y}_n,U)\subset \subset B(\boldsymbol{0},r)$ for $n\gg 1$, so that (iv) follows by applying the Lebesgue Dominated Convergence Theorem.
\end{proof}

Let $\boldsymbol{y}\in W^{1,p}(\Omega;\R^N)$. Let $U\subset \subset \Omega$ be a domain of class $C^2$ and suppose that $\boldsymbol{y} \restr{\partial U}\in W^{1,p}(\partial U;\R^N)$. As $p>N-1$, by the Morrey embedding, the map $\boldsymbol{y}\restr{\partial U}$ admits a continuous representative in $W^{1,p}(\partial U;\R^N)$. In this situation, \emph{we adopt the convention of identifying ${\boldsymbol{y}\restr{\partial U}}$ with its representative in $C^0(\partial U;\R^N)$}. We will often omit the restriction to $\partial U$ in the notation. 

We need some geometric preliminaries. Let $U \subset \subset \Omega$ be a domain of class $C^2$. In this case, 
$\partial U$ is a compact hypersurface of class $C^2$ without boundary. The signed distance function $d_U\colon \R^N \to \R$ is defined by setting
\begin{equation*}
d_U(\boldsymbol{x})\coloneqq \begin{cases}
\dist(\boldsymbol{x};\partial U), & \boldsymbol{x}\in U, \\
0, & \boldsymbol{x}\in \partial U, \\
-\dist(\boldsymbol{x};\partial U), & \boldsymbol{x}\in \R^N \setminus \closure{U}.
\end{cases}
\end{equation*}
For every $\ell \in \R$, we set $U_\ell\coloneqq \{\boldsymbol{x}\in \R^N:\:d_U(\boldsymbol{x})>\ell\}$. As $d_U$ is continuous, the set $U_\ell$ is open and $\partial U_\ell=\{\boldsymbol{x}\in \R^N:\:d_U(\boldsymbol{x})=\ell\}$. Given $\delta>0$, we consider the tubular neighborhood
\begin{equation*}
T(\partial U,\delta)\coloneqq \{\boldsymbol{x}\in \R^N:\:-\delta< d_U(\boldsymbol{x})<\delta\}.
\end{equation*}
For $\delta \ll 1$ depending on $U$, the following properties hold \cite[Section 4]{ambrosio}:
\begin{enumerate}[(i)]
	\item the function $d_U$ is of class $C^2$ on $T(\partial U,\delta)$;
	\item the projection onto the boundary $\boldsymbol{\Pi}_U\colon T(\partial U,\delta) \to \partial U$ given by
	\begin{equation*}
	\boldsymbol{\Pi}_U(\boldsymbol{x})\coloneqq \argmin_{{\boldsymbol{z}}\in \partial U} |\boldsymbol{x}-{\boldsymbol{z}}|     
	\end{equation*}
	is well defined;
	\item for every $\boldsymbol{x}\in T(\partial U,\delta)$ there hold
	\begin{equation*}
	\boldsymbol{x}=\boldsymbol{\Pi}_U(\boldsymbol{x})+d_U(\boldsymbol{x})\boldsymbol{n}_U(\boldsymbol{\Pi}_U(\boldsymbol{x})), \qquad D d_U(\boldsymbol{x})=-\boldsymbol{n}_U(\boldsymbol{\Pi}_U(\boldsymbol{x})),
	\end{equation*}
	and $\boldsymbol{n}_{U_\ell}(\boldsymbol{x})=\boldsymbol{n}_U(\boldsymbol{\Pi}_U(\boldsymbol{x}))$ whenever $\boldsymbol{x}\in \partial U_\ell$.
\end{enumerate}

The following result is available in the literature. For the convenience of the reader, we provide a detailed proof in the Appendix.

\begin{lemma}[{\cite[Lemma 2.9]{mueller.spector}}]
\label{lem:weak-convergence-boundary}
Let $(\boldsymbol{y}_n)\subset W^{1,p}(\Omega;\R^N)$ and $\boldsymbol{y}\in W^{1,p}(\Omega;\R^N)$. Suppose that $\boldsymbol{y}_n \to \boldsymbol{y}$ in $W^{1,p}(\Omega;\R^N)$. Let $U\subset \subset \Omega$ be a domain of class $C^2$. Then, there exists $\delta>0$ such that for almost every $\ell \in (-\delta,\delta)$ there holds
\begin{equation}
\label{eqn:sobolev-regularity-boundary}
    \text{$\forall\,n \in \N,\quad\boldsymbol{y}_n\restr{\partial U_\ell}\in W^{1,p}(\partial U_\ell;\R^N),$  \quad $\boldsymbol{y}\restr{\partial U_\ell}\in W^{1,p}(\partial U_l;\R^N)$}.
\end{equation}
Moreover, for any such $\ell\in (-\delta,\delta)$, there exists a subsequence $(\boldsymbol{y}_{n_k})$, possibly depending on $\ell$, such that
\begin{equation}
\label{eqn:weak-convergence-boundary}
    \text{$\boldsymbol{y}_{n_k}\wk \boldsymbol{y}$ in $W^{1,p}(\partial U_\ell;\R^N)$, \quad  ${\boldsymbol{y}}_{n_k}\to {\boldsymbol{y}}$ uniformly on $\partial U_\ell$.}
\end{equation}
\end{lemma}

\begin{remark}
	Let $\boldsymbol{y}\in W^{1,p}(\Omega;\R^N)$ and let $\boldsymbol{x}_0\in \Omega$. By Lemma \ref{lem:weak-convergence-boundary}, for almost every $r\in (0,\dist(\boldsymbol{x}_0;\partial \Omega))$, we have $\boldsymbol{y}\restr{\partial B(\boldsymbol{x}_0,r)}\in W^{1,p}(\partial B(\boldsymbol{x}_0,r);\R^N)$. In view of Corollary \ref{cor:lebesgue-points-balls}, without loss of generality, we can assume that $\partial B(\boldsymbol{x}_0,r)\subset L_{\boldsymbol{y}}$. In this case, by exploiting \cite[Claim (iii), Proposition 2.8]{mueller.spector} it is possible to show that the continuous representative of $\boldsymbol{y}\restr{\partial B(\boldsymbol{x}_0,r)}$ is given by $\boldsymbol{y}^*\restr{\partial B(\boldsymbol{x}_0,r)}$. 
\end{remark}

For technical reasons, following \cite[Definition 6]{henao.moracorral.fracture}, we introduce the class of regular subdomains.  

\begin{definition}[Regular subdomains]
\label{def:regular-subdomains}
Let $\boldsymbol{y}\in W^{1,p}(\Omega;\R^N)$. We define  $\mathcal{U}_{\boldsymbol{y}}$ as the class of domains $U\subset \subset \Omega$ of class $C^2$ satisfying the following properties:
\begin{enumerate}[(i)]
    \item $\boldsymbol{y}\restr{\partial U}\in W^{1,p}(\partial U;\R^N)$ and $(\cof\,\nabla \boldsymbol{y})\restr{\partial U}\in L^1(\partial U;\rnn)$;
    \item $\haus(\partial U \setminus \Omega_{\boldsymbol{y}})=0$ and $\nabla^{\partial U}{\boldsymbol{y}}\restr{\partial U}=\nabla \boldsymbol{y}(\boldsymbol{I}-\boldsymbol{n}_U\otimes \boldsymbol{n}_U)$  a.e. on $\partial U$;
    \item there holds
    \begin{equation*}
        \displaystyle \lim_{\delta \to 0^+} \dashint_{-\delta}^\delta \left |\int_{\partial U_\ell} |\cof\hspace{1pt}\nabla \boldsymbol{y}|\,\d\boldsymbol{a}- \int_{\partial U} |\cof\hspace{1pt}\nabla \boldsymbol{y}|\,\d\boldsymbol{a}\right |\,\d \ell=0;
    \end{equation*}
    \item for every $\boldsymbol{\psi}\in C^1_c(\R^N;\R^N)$, there holds
    \begin{equation*}
    \begin{split}
        &\lim_{\delta \to 0^+} \dashint_{-\delta}^\delta \bigg | \int_{\partial U_\ell} \boldsymbol{\psi}\circ\boldsymbol{y}\cdot (\cof \hspace{1pt}\nabla\boldsymbol{y})\boldsymbol{n}_{U_\ell}\,\d\boldsymbol{a}- \int_{\partial U} \boldsymbol{\psi}\circ\boldsymbol{y}\cdot (\cof \hspace{1pt}\nabla\boldsymbol{y})\boldsymbol{n}_{U}\,\d\boldsymbol{a}          \bigg |\,\d \ell=0.
    \end{split}
    \end{equation*}
\end{enumerate}
\end{definition}

Property (i) in the previous definition ensures the regularity of the map $\boldsymbol{y}\restr{\partial U}$. In (ii), we denote by $\nabla^{\partial U} {\boldsymbol{y}}\restr{\partial U}$ the approximate tangential gradient of ${\boldsymbol{y}}\restr{\partial U}$, see \cite[Definition 4]{henao.moracorral.fracture} or \cite[Paragraph 3.2.16]{federer}. Since we are not going to use it explicitly, we do not introduce this notion. For our purposes, the only important fact is that, in view of (ii), there holds $\leb({\boldsymbol{y}}(\partial U))=0$ \cite[Remark 2.16, Claim (b)]{barchiesi.henao.moracorral}. This is a consequence of the Area Formula for surface integrals, see \cite[Proposition 2]{henao.moracorral.fracture} or \cite[Corollary 3.2.20]{federer}, that, for the same reason as before, we do not present. Eventually, properties (iii) and (iv) are crucial for the validity of Theorem \ref{thm:deg=m} below.

The following result ensures the abundance of regular subdomains as in Definition \ref{def:regular-subdomains}. We use again the notation  from the Appendix.

\begin{lemma}[{\cite[Lemma 2]{henao.moracorral.fracture}}]
\label{lem:abundance}
Let $\boldsymbol{y}\in W^{1,p}(\Omega;\R^N)$ be such that $\det D\boldsymbol{y}>0$ almost everywhere and let $U \subset \subset \Omega$ be a domain of class $C^2$. Then, there exists $\delta>0$  such that , for almost every $\ell\in (-\delta,\delta)$, there holds $U_\ell\in\mathcal{U}_{\boldsymbol{y}}$.
\end{lemma}

The following result is based on Proposition \ref{prop:regular-approximate-differentiability} and complements \cite[Lemma 3.2]{barchiesi.henao.moracorral} with an elementary observation.

\begin{lemma}
\label{lem:regular-approximate-differentiability}
Let $\boldsymbol{y}\in W^{1,p}(\Omega;\R^N)$ and let $\boldsymbol{x}_0\in \Omega$. Suppose that $\boldsymbol{y}$ is regularly approximately differentiable at $\boldsymbol{x}_0$ with $\det \nabla \boldsymbol{y}(\boldsymbol{x}_0)\neq 0$. Then, there exists $P\subset (0,\dist(\boldsymbol{x}_0;\partial \Omega))$ with $\Theta^1_+(P,0)=1$ such that the following hold:
\begin{enumerate}[(i)]
    \item for every $r\in P$, we have
    \begin{equation*}
        B(\boldsymbol{x}_0,r)\in \mathcal{U}_{\boldsymbol{y}}, \quad \boldsymbol{y}(\boldsymbol{x}_0)\notin {\boldsymbol{y}}(\partial B(\boldsymbol{x}_0,r)), \quad \deg(\boldsymbol{y},B(\boldsymbol{x}_0,r),\boldsymbol{y}(\boldsymbol{x}_0))=\mathrm{sgn} \det \nabla \boldsymbol{y}(\boldsymbol{x}_0);
    \end{equation*}
    \item for every $r,r'\in P$ with $r'<r/2$, the set ${\boldsymbol{y}}(\partial B(\boldsymbol{x}_0,r'))$ is included in the connected component of $\R^N\setminus {\boldsymbol{y}}(\partial B(\boldsymbol{x}_0,r))$ containing $\boldsymbol{y}(\boldsymbol{x}_0)$ and, in turn, we have
    \begin{equation*}
        \im_{\rm T}(\boldsymbol{y},B(\boldsymbol{x}_0,r')) \cup {\boldsymbol{y}}(\partial B(\boldsymbol{x}_0,r')) \subset  \im_{\rm T}(\boldsymbol{y},B(\boldsymbol{x}_0,r)). 
    \end{equation*}
\end{enumerate}
\end{lemma}
\begin{proof}
Claim (i) is proved  in \cite[Lemma 3.2]{barchiesi.henao.moracorral}. For claim (ii), let $V$ denote the connected component of $\R^N \setminus {\boldsymbol{y}}(\partial B(\boldsymbol{x}_0,r))$ such that $\boldsymbol{y}(\boldsymbol{x}_0)\in V$. The inclusion ${\boldsymbol{y}}(\partial B(\boldsymbol{x}_0,r'))\subset V$ is  shown in \cite[Lemma 3.2]{barchiesi.henao.moracorral}. As $\deg(\boldsymbol{y},B(\boldsymbol{x}_0,r),\boldsymbol{y}(\boldsymbol{x}_0))\neq 0$, we have $V \subset \im_{\rm T}(\boldsymbol{y},B(\boldsymbol{x}_0,r))$. Therefore, we deduce that $\im_{\rm T}(\boldsymbol{y},B(\boldsymbol{x}_0,r')) \subset \im_{\rm T}(\boldsymbol{y},B(\boldsymbol{x}_0,r))$ by Lemma \ref{lem:topological-image-contained} because $\R^N \setminus \im_{\rm T}(\boldsymbol{y},B(\boldsymbol{x}_0,r))$ is connected by Lemma \ref{lem:connected}. 
\end{proof}

\subsection{Deformations excluding cavitation}
Let $\boldsymbol{y}\in W^{1,p}(\Omega;\R^N)$. Following the terminology in \cite{mueller.qi.yan}, the Divergence Identities read as follows:
\begin{equation}
    \tag{DIV}
    \label{eqn:DIV}
    \forall\,\boldsymbol{\psi}\in C^1_{\rm c}(\R^N;\R^N),\quad \mathrm{Div}((\adj\hspace{1pt}D\boldsymbol{y})\,\boldsymbol{\psi}\circ\boldsymbol{y})=\div\boldsymbol{\psi}\circ \boldsymbol{y}\,\det D\boldsymbol{y} \quad \text{in $\Omega$.}
\end{equation}
Such identities have been examined in \cite{mueller.div} in relation with the problem of the weak continuity of the Jacobian determinant.
The divergence on the left-hand side of the previous equation is intendend in the sense of distributions. Namely, $\boldsymbol{y}$ satisfies \eqref{eqn:DIV}  whenever
\begin{equation*}
    \forall\,\boldsymbol{\psi}\in C^1_{\rm c}(\R^N;\R^N),\:\forall\,\varphi \in C^1_{\rm c}(\Omega),\quad -\int_\Omega ((\adj\hspace{1pt}D\boldsymbol{y})\,\boldsymbol{\psi}\circ\boldsymbol{y})\cdot D\varphi\,\d \boldsymbol{x}= \int_\Omega \div\boldsymbol{\psi}\circ \boldsymbol{y}\,\det D\boldsymbol{y}\,\varphi\,\d\boldsymbol{x}.
\end{equation*}

We will consider the following class of admissible deformations:
\begin{equation}
\label{eqn:admissible-deformations}
    \mathcal{Y}_p(\Omega)\coloneqq \left \{\boldsymbol{y}\in W^{1,p}(\Omega;\R^N):\:\det D\boldsymbol{y}\in L^1(\Omega), \: \text{$\det D \boldsymbol{y}>0$ a.e. in $\Omega$}, \: \text{$\boldsymbol{y}$ satisfies \eqref{eqn:DIV}} \right \}.
\end{equation}

This class has been first introduced in \cite{barchiesi.henao.moracorral}, where the authors carried out a thorough study of fine and local invertibility properties of such maps. As a result of their analysis, deformations in this class turn out to enjoy a surprising degree of regularity.

The Divergence Identities entail the coincidence of the topological image and the geometric image of regular subdomains up to sets of Lebesgue measure zero. Recall Definition \ref{def:geometric-image} and the  definition of multiplicity function in Proposition \ref{prop:change-of-variable} as well as the convention established in Subsection \ref{subsec:approximate-differentiability-geometric-image}. 

\begin{theorem}[{\cite[Theorem 4.1]{barchiesi.henao.moracorral}}]
\label{thm:deg=m}
Let $\boldsymbol{y}\in \mathcal{Y}_p(\Omega)$. Then, for every $U\in \mathcal{U}_{\boldsymbol{y}}$, there holds
\begin{equation*}
    \text{$\deg(\boldsymbol{y},U,\cdot)=\mult(\boldsymbol{y},U,\cdot)$ a.e. in $\R^N\setminus {\boldsymbol{y}}(\partial U)$.}
\end{equation*}
Therefore, $\leb\left (\im_{\rm T}(\boldsymbol{y},U)\hspace{1pt}\triangle \hspace{1pt}\im_{\rm G}(\boldsymbol{y},U)\right)=0$. In particular, $\boldsymbol{y}\in L^\infty_{\loc}(\Omega;\R^N)$.   
\end{theorem}

The following result collects some consequences of the previous theorem. 

\begin{corollary}
Let $\boldsymbol{y}\in \mathcal{Y}_p(\Omega)$. Then:
\begin{enumerate}[(i)]
    \item for every $U \in \mathcal{U}_{\boldsymbol{y}}$, there holds $\deg(\boldsymbol{y},U,\cdot)\geq 0$ almost everywhere in $\R^N \setminus \overline{\boldsymbol{y}}(\partial U)$;
    \item for every $U_1,U_2 \in \mathcal{U}_{\boldsymbol{y}}$ such that $U_1 \subset \subset U_2$, there holds $\closure{\im_{\rm T}(\boldsymbol{y},U_1)}\subset \closure{\im_{\rm T}(\boldsymbol{y},U_2)} $;
    \item for every $U \in \mathcal{U}_{\boldsymbol{y}}$, there holds $\im_{\rm G}(\boldsymbol{y},U)\subset \closure{\im_{\rm T}(\boldsymbol{y},U)}$.
\end{enumerate}
\end{corollary}
\begin{proof}
Claim (i) is immediate. Claims (ii) and (iii) are proved as in \cite[Proposition 4.3, Claim (d)]{barchiesi.henao.moracorral} and \cite[Proposition 4.3, Claim (e)]{barchiesi.henao.moracorral}, respectively, by exploiting (i). 
\end{proof}

For technical reasons, we give the following definition.

\begin{definition}[Singularity set]
Let $\boldsymbol{y}\in W^{1,p}(\Omega;\R^N)$. We define the singularity set of $\boldsymbol{y}$ as
\begin{equation*}
    S_{\boldsymbol{y}}\coloneqq \left \{\boldsymbol{x}\in \Omega:\:\limsup_{\substack{r \to 0 \\ B(\boldsymbol{x},r)\in \mathcal{U}_{\boldsymbol{y}}}} \mathrm{diam}(\im_{\rm T}(\boldsymbol{y},B(\boldsymbol{x},r)))>0 \right\}.
\end{equation*}
\end{definition}

The set $S_{\boldsymbol{y}}$ does not depend on the specific representative of $\boldsymbol{y}$ but only on its equivalence class \cite[Remark 5.7, Claim (c)]{barchiesi.henao.moracorral}. Moreover, one can prove that $\mathscr{H}^{\mathrm{max}\{N-p,0\}}(S_{\boldsymbol{y}})=0$, $\Omega \setminus S_{\boldsymbol{y}} \subset L_{\boldsymbol{y}}$ and that $\boldsymbol{y}^*\restr{\Omega \setminus S_{\boldsymbol{y}}}$ is continuous \cite[Proposition 5.9]{barchiesi.henao.moracorral}.

We denote by $\mathcal{U}_{\boldsymbol{y}}^*$ the class of regular subdomains $U \in \mathcal{U}_{\boldsymbol{y}}$ such that $\partial U \cap S_{\boldsymbol{y}}=\emptyset$. Combining Lemma \ref{lem:abundance} with the estimate on the Hausdorff dimension of the singularity set, we obtain the following result.

\begin{lemma}[{\cite[Lemma 5.12]{barchiesi.henao.moracorral}}]
\label{lem:abundance-star}
Let $\boldsymbol{y}\in \mathcal{Y}_p(\Omega)$ and let $U \subset \subset \Omega$ be a domain of class $C^2$. Then, there exists $\delta>0$ such that, for almost every $\ell \in (-\delta,\delta)$, there holds $U_\ell\in \mathcal{U}_{\boldsymbol{y}}^*$.
\end{lemma}

We define the topological image according to \cite[Definition 5.17]{barchiesi.henao.moracorral}.

\begin{definition}[Topological image]
\label{def:topological-image}
Let $\boldsymbol{y}\in \mathcal{Y}_p(\Omega)$. 
The topological image of $\Omega$ under $\boldsymbol{y}$ is defined as
\begin{equation*}
    \im_{\rm T}(\boldsymbol{y},\Omega)\coloneqq \bigcup_{U \in \mathcal{U}_{\boldsymbol{y}}^*} \im_{\rm T}(\boldsymbol{y},U).
\end{equation*}
\end{definition}

Being the union of open sets, the set $\im_{\rm T}(\boldsymbol{y},\Omega)$ is open. Notably, as shown in \cite[Lemma 5.18, Claim (b)]{barchiesi.henao.moracorral},  this set depends only on the equivalence class of $\boldsymbol{y}$. This fact relies on the restriction taken in Definition \ref{def:topological-image} to the regular subdomains whose boundary does not intersect the singularity set.

Combining Theorem \ref{thm:deg=m} with Lemma \ref{lem:abundance-star}, one obtains the following.

\begin{corollary}[{\cite[Lemma 5.18, Claim (c)]{barchiesi.henao.moracorral}}]
\label{cor:topological-geometric-image}
Let $\boldsymbol{y}\in \mathcal{Y}_p(\Omega)$. Then, there holds $$\leb\left (\im_{\rm T}(\boldsymbol{y},\Omega)\hspace{1pt}\triangle \hspace{1pt}\im_{\rm G}(\boldsymbol{y},\Omega)\right)=0.$$
\end{corollary}

The class of maps defined in \eqref{eqn:admissible-deformations} is closed with respect to the weak convergence in $W^{1,p}(\Omega;\R^N)$ and, on this class, the Jacobian determinant is continuous with respect to the weak convergence in $L^1(\Omega)$.

\begin{theorem}[{\cite[Theorem 4]{mueller.div}}]
\label{thm:closedness}
Let $(\boldsymbol{y}_n)\subset \mathcal{Y}_p(\Omega)$. Suppose that there exist $\boldsymbol{y}\in W^{1,p}(\Omega;\R^N)$ and $h\in L^1(\Omega)$ such that
\begin{equation*}
    \text{$\boldsymbol{y}_n \wk \boldsymbol{y}$ in $W^{1,p}(\Omega;\R^N)$, \quad $\det D \boldsymbol{y}_n \wk h$ in $L^1(\Omega)$.}
\end{equation*}
Then, $\boldsymbol{y}$ satisfies \eqref{eqn:DIV} and $h=\det D \boldsymbol{y}$ almost everywhere. In particular, if $\det D \boldsymbol{y}>0$ almost everywhere, then $\boldsymbol{y}\in \mathcal{Y}_p(\Omega)$.
\end{theorem}

The following result will be instrumental in the proof of our compactness theorem. For the validity of the first claim,  the fact that in Definition \ref{def:topological-image} we only considered regular subdomains whose boundaries do not intersect the singularity set  is essential. This is the reason why this restriction is taken in the present paper.

\begin{proposition}
\label{prop:weak-convergence}
Let $(\boldsymbol{y}_n)\subset \mathcal{Y}_p(\Omega)$ and let $\boldsymbol{y}\in \mathcal{Y}_p(\Omega)$. Suppose that $\boldsymbol{y}_n \wk \boldsymbol{y}$ in $W^{1,p}(\Omega;\R^N)$.
\begin{enumerate}[(i)]
    \item for every $U \in \mathcal{U}_{\boldsymbol{y}}^*$ and for every $V \subset \subset \im_{\rm T}(\boldsymbol{y},U)$ open,  there holds $V \subset \subset \im_{\rm T}(\boldsymbol{y}_n,\Omega)$ for $n \gg 1$;
    \item if $(\det D\boldsymbol{y}_n)$ is equi-integrable, then, up to subsequences, there holds 
    \begin{equation*}
        \text{$\chi_{\im_{\rm T}(\boldsymbol{y}_n,\Omega)} \to \chi_{\im_{\rm T}(\boldsymbol{y},\Omega)}$ in $L^1(\R^N)$,}
    \end{equation*}
    and also
    \begin{equation*}
        \text{$\chi_{\im_{\rm G}(\boldsymbol{y}_n,\Omega)} \to \chi_{\im_{\rm G}(\boldsymbol{y},\Omega)}$ in $L^1(\R^N)$.}
    \end{equation*}
\end{enumerate}
\end{proposition}
\begin{proof}
Claim (i) has been proved in \cite[Theorem 6.3, Claim (a)]{barchiesi.henao.moracorral}. For (ii), the first convergence has been proved in \cite[Theorem 6.3, Claim (d)]{barchiesi.henao.moracorral}, and the second one is  deduced from the first one thanks to Corollary \ref{cor:topological-geometric-image}.
\end{proof}

\subsection{Local invertibility}

Let $\boldsymbol{y}\colon \Omega \to\R^N$ be measurable and let $A\subset \Omega$ be measurable. The map $\boldsymbol{y}$ is termed almost everywhere injective in $A$ if there exists a subset $X\subset A$ with $\leb(X)=0$ such that $\boldsymbol{y}\restr{A \setminus X}$ is an injective map. 
We observe that almost everywhere injectivity depends only on the equivalence class of $\boldsymbol{y}$.
The following result constitutes one of the main motivations for the definition of the geometric domain.

\begin{lemma}[{\cite[Lemma 3]{henao.moracorral.fracture}}]
\label{lem:injective}
Let $\boldsymbol{y}\colon \Omega \to \R^N$ be almost everywhere approximately differentiable with $\det \nabla\boldsymbol{y}\neq 0$ almost everywhere. Suppose that $\boldsymbol{y}$ is almost everywhere injective on $A$ for some measurable set $A \subset \Omega$. Then, $\boldsymbol{y}\restr{\dom_{\rm G}(\boldsymbol{y},A)}$ is injective.
\end{lemma}

Let $\boldsymbol{y}\in W^{1,p}(\Omega;\R^N)$ and recall Definition \ref{def:regular-subdomains}. We denote by $\mathcal{U}_{\boldsymbol{y}}^{\rm inj}$ the class of regular subdomains $U \in \mathcal{U}_{\boldsymbol{y}}$ such that $\boldsymbol{y}$ is almost everywhere injective in $U$. Also,  we set $\mathcal{U}_{\boldsymbol{y}}^{*,\,\mathrm{inj}}\coloneqq\mathcal{U}_{\boldsymbol{y}}^{*}\cap \mathcal{U}_{\boldsymbol{y}}^{\mathrm{inj}}$.

The following local invertibility result has been proven in \cite{barchiesi.henao.moracorral} and is going to be instrumental in our proofs.

\begin{theorem}[{\cite[Proposition 4.5, Claim (d)]{barchiesi.henao.moracorral}}]
	\label{thm:local-invertibility}
	Let $\boldsymbol{y} \in  \mathcal{Y}_p(\Omega)$ and let $\boldsymbol{x}_0\in \Omega$ be such that $\boldsymbol{y}$ is regularly approximately differentiable at $\boldsymbol{x}_0$ with $\det \nabla \boldsymbol{y}(\boldsymbol{x}_0)>0$. Let 
	$P\subset (0,\dist(\boldsymbol{x}_0;\partial \Omega))$ be the set given by Lemma \ref{lem:regular-approximate-differentiability}.
	Then, for every $r\in P$ such that there exists $\widetilde{r}\in P$ with $r<\widetilde{r}/2$
	, we have $B(\boldsymbol{x}_0,r)\in \mathcal{U}^{\rm inj}_{\boldsymbol{y}}$.
\end{theorem}

Let $\boldsymbol{y}\colon \Omega \to \R^N$ be almost everywhere approximately differentiable with $\det \nabla \boldsymbol{y}\neq 0$ almost everywhere and let $A \subset \Omega$ be measurable. Suppose that $\boldsymbol{y}$ is almost everywhere injective in $A$. By Lemma \ref{lem:injective}, the map $\boldsymbol{y}\restr{\dom_{\rm G}(\boldsymbol{y},A)}\colon \dom_{\rm G}(\boldsymbol{y},A)\to \im_{\rm G}(\boldsymbol{y},A)$ is injective and, in turn, its inverse $(\boldsymbol{y}\restr{\dom_{\rm G}(\boldsymbol{y},A)})^{-1}\colon \im_{\rm G}(\boldsymbol{y},A)\to \dom_{\rm G}(\boldsymbol{y},A)$ is defined. Concerning the approximate differentiablity of this map, we make the following observation.

\begin{lemma}[Approximate differentiability of the inverse]
\label{lem:inverse-approximately-differentiable}
Let $\boldsymbol{y}\colon \Omega \to \R^N$ be almost everywhere approximately differentiable with $\det \nabla \boldsymbol{y}\neq 0$ almost everywhere and let $A\subset \Omega$ be measurable.
Suppose that $\boldsymbol{y}$ is almost everywhere injective in $A$. Then, for every $\boldsymbol{x}_0\in \dom_{\rm G}(\boldsymbol{y},A)$ with $\Theta^N(A,\boldsymbol{x}_0)=1$, the map $(\boldsymbol{y}\restr{\dom_{\rm G}(\boldsymbol{y},A)})^{-1}$ is approximately differentiable at $\boldsymbol{y}(\boldsymbol{x}_0)$.
\end{lemma}
\begin{proof}
Consider $\boldsymbol{x}_0\in \dom_{\rm G}(\boldsymbol{y},A)$ with $\Theta^N(A,\boldsymbol{x}_0)=1$. By Definition \ref{def:geometric-image}, there exists a compact set $K \subset \Omega$ with $\boldsymbol{x}_0\in K$ and $\Theta^N(K,\boldsymbol{x}_0)=1$, and a map $\boldsymbol{w}\in C^1(\R^N;\R^N)$ such that $\boldsymbol{y}\restr{K}=\boldsymbol{w}\restr{K}$ and $\nabla \boldsymbol{y}\restr{K}=D\boldsymbol{w}\restr{K}$. Since $\det D\boldsymbol{w}(\boldsymbol{x}_0)=\det \nabla \boldsymbol{y}(\boldsymbol{x}_0)\neq 0$, by the Local Diffeomorphism Theorem, there exists $r>0$ such that, setting $B\coloneqq B(\boldsymbol{x}_0,r)$, the map $\boldsymbol{w}\restr{B}$ is a diffeomorphism of class $C^1$ onto its image. In particular, $D(\boldsymbol{w}\restr{B})^{-1}=(D\boldsymbol{w})^{-1}\circ (\boldsymbol{w}\restr{B})^{-1}$ on the open set $\boldsymbol{w}(B)$. Thus, setting $\boldsymbol{\xi}_0\coloneqq \boldsymbol{y}(\boldsymbol{x}_0)$, there holds
\begin{equation}
\label{eqn:differentiability}
    \lim_{\boldsymbol{\xi}\to \boldsymbol{\xi}_0} \frac{|(\boldsymbol{w}\restr{B})^{-1}(\boldsymbol{\xi})-(\boldsymbol{w}\restr{B})^{-1}(\boldsymbol{\xi}_0)-(D\boldsymbol{w}(\boldsymbol{x}_0))^{-1}(\boldsymbol{\xi}-\boldsymbol{\xi}_0)|}{|\boldsymbol{\xi}-\boldsymbol{\xi}_0|}=0.
\end{equation}
We have $\Theta^N(K\cap B \cap \dom_{\rm G}(\boldsymbol{y},A),\boldsymbol{x}_0)=1$ so that, by Lemma \ref{lem:geometric-image},  there holds
\begin{equation*}
	\Theta^N(\boldsymbol{y}(K\cap  B)\cap \im_{\rm G}(\boldsymbol{y},A),\boldsymbol{\xi}_0)=1.
\end{equation*}
Observing that $(\boldsymbol{w}\restr{B})^{-1}=(\boldsymbol{y}\restr{\dom_{\rm G}(\boldsymbol{y},A)})^{-1}$ and $D(\boldsymbol{w}\restr{B})^{-1}=(\nabla \boldsymbol{y})^{-1}\circ (\boldsymbol{y}\restr{\dom_{\rm G}(\boldsymbol{y},A)})^{-1}$ on $\boldsymbol{y}(K \cap B) \cap \im_{\rm G}(\boldsymbol{y},A)$, from \eqref{eqn:differentiability}, we obtain
\begin{equation*}
\lim_{\substack{\boldsymbol{\xi}\to \boldsymbol{\xi}_0 \\ \boldsymbol{\xi}\in \boldsymbol{y}(K \cap B) \cap \im_{\rm G}(\boldsymbol{y},A)}}\hspace*{-7mm} \frac{|(\boldsymbol{y}\restr{\dom_{\rm G}(\boldsymbol{y},A)})^{-1}(\boldsymbol{\xi})-(\boldsymbol{y}\restr{\dom_{\rm G}(\boldsymbol{y},A)})^{-1}(\boldsymbol{\xi}_0)-(\nabla \boldsymbol{y}(\boldsymbol{x}_0))^{-1}(\boldsymbol{\xi}-\boldsymbol{\xi}_0)|}{|\boldsymbol{\xi}-\boldsymbol{\xi}_0|}=0.
\end{equation*}
Therefore, the claim is proved.
\end{proof}

Let $\boldsymbol{y}\in W^{1,p}(\Omega;\R^N)$ be such that $\det D\boldsymbol{y}\neq 0$ almost everywhere and let $U\in\mathcal{U}_{\boldsymbol{y}}^{\rm inj}$. As before, we can consider the inverse $(\boldsymbol{y}\restr{\dom_{\rm G}(\boldsymbol{y},U)})^{-1}:\im_{\rm G}(\boldsymbol{y},U)\to \R^N$. Recalling Theorem \ref{thm:deg=m}, we give the following definition.

\begin{definition}[Local inverse]
\label{def:local-inverse}
Let $\boldsymbol{y}\in\mathcal{Y}_p(\Omega)$ and let $U\in\mathcal{U}_{\boldsymbol{y}}^{\rm inj}$. We define $\boldsymbol{y}_U^{-1}\colon \im_{\rm T}(\boldsymbol{y},U)\to\R^N$ to be an arbitrary extension of the map $(\boldsymbol{y}\restr{\dom_{\rm G}(\boldsymbol{y},U)})^{-1}\restr{\im_{\rm T}(\boldsymbol{y},U) \cap \im_{\rm G}(\boldsymbol{y},U)}$.
\end{definition}

For definiteness, one can set $\boldsymbol{y}_U^{-1}$ to be the extension by zero. However, for every other choice for the extension, the resulting map $\boldsymbol{y}_{U}^{-1}$ is measurable thanks to the completeness of the Lebesgue measure. 
As a consequence of \eqref{eqn:DIV}, the map $\boldsymbol{y}_{U}^{-1}$ turns out to have Sobolev regularity. This is meaningful since the set $\im_{\rm T}(\boldsymbol{y},U)$ is open. 

\begin{proposition}[{\cite[Proposition 5.3]{barchiesi.henao.moracorral}}]
\label{prop:regularity-inverse}
Let $\boldsymbol{y}\in\mathcal{Y}_p(\Omega)$ and let $U\in\mathcal{U}_{\boldsymbol{y}}^{\rm inj}$. Then, $\boldsymbol{y}_{U}^{-1}\in W^{1,1}(\im_{\rm T}(\boldsymbol{y},U);\R^N)$ and $D\boldsymbol{y}_{U}^{-1}=(D\boldsymbol{y})^{-1}\circ \boldsymbol{y}_{U}^{-1}$ almost everywhere in $\im_{\rm T}(\boldsymbol{y},U)$. Moreover, each Jacobian minor of $\boldsymbol{y}_{U}^{-1}$ belongs to $L^1(\im_{\rm T}(\boldsymbol{y},U))$.
\end{proposition}

From Proposition \ref{prop:regularity-inverse}, we deduce that $\boldsymbol{y}_{U}^{-1}$ is almost everywhere approximately differentiable \cite[Theorem 2, Section 1.4, Chapter 3]{cartesian.currents}, so that the Change-of-variable Formula holds. More precisely, we have the following.

\begin{proposition}[Change-of-variable Formula for the inverse]
\label{prop:inverse-change-of-variable}
Let $\boldsymbol{y}\in \mathcal{Y}_p(\Omega)$ and let $U\in\mathcal{U}_{\boldsymbol{y}}^{\rm inj}$. Then, for every  $\boldsymbol{\varphi}\colon U \to \R^M$ measurable, there holds
\begin{equation*}
    \int_U \boldsymbol{\varphi}(\boldsymbol{x})\,\d\boldsymbol{x}=\int_{\im_{\rm T}(\boldsymbol{y},U)} \boldsymbol{\varphi} \circ \boldsymbol{y}_{U}^{-1}(\boldsymbol{\xi})\,\det D\boldsymbol{y}_{U}^{-1}(\boldsymbol{\xi})\,\d\boldsymbol{\xi}.
\end{equation*}
\end{proposition}

In the previous formula, we observe that $\boldsymbol{\varphi}\circ \boldsymbol{y}_{U}^{-1}$ is measurable since $(\boldsymbol{y}\restr{U})^{-1}$ has the Lusin property (N${}^{-1}$) as $\det D\boldsymbol{y}_{U}^{-1}>0$ almost everywhere, see \cite[Lemma 2.8, Claim (c)]{barchiesi.henao.moracorral} and \cite[Lemma 2.9]{barchiesi.henao.moracorral}.

\begin{proof}
	By Lemma \ref{lem:inverse-approximately-differentiable}, the map $(\boldsymbol{y}\restr{\dom_{\rm G}(\boldsymbol{y},U)})^{-1}$ is approximately differentiable on $\im_{\rm G}(\boldsymbol{y},U)$. The result follows by applying the Federer Change-of-variable Formula \cite[Theorem 1, Section 1.5, Chapter 3]{cartesian.currents} to this map taking into account Theorem \ref{thm:deg=m} and Definition \ref{def:local-inverse}.
\end{proof}

The next result show that the radius of local invertibility given by Theorem \ref{thm:local-invertibility} can be chosen uniformly for  sequences of deformations converging weakly in $W^{1,p}(\Omega;\R^N)$. The result is taken from \cite{barchiesi.henao.moracorral} and it is reformulated according to our needs.

\begin{proposition}[{\cite[Theorem 6.3, Claim (b)]{barchiesi.henao.moracorral}}]
	\label{prop:local-invertibility-sequence}
	Let $\boldsymbol{y}\in \mathcal{Y}_p(\Omega)$ and let $\boldsymbol{x}_0\in \Omega$ be such that $\boldsymbol{y}$ is regularly approximately differentiable at $\boldsymbol{x}_0$ with $\det \nabla \boldsymbol{y}(\boldsymbol{x}_0)>0$. Let $P\subset (0,\dist(\boldsymbol{x}_0;\partial \Omega))$ be the set given by Lemma \ref{lem:regular-approximate-differentiability}.
	Also, let $(\boldsymbol{y}_n)\subset \mathcal{Y}_p(\Omega)$ and suppose that $\boldsymbol{y}_n \wk \boldsymbol{y}$ in $W^{1,p}(\Omega;\R^N)$.  Then, for every $r\in P$ such that there exist $\widetilde{r}\in P$ with $r<\widetilde{r}/2$ and a subsequence $(\boldsymbol{y}_{n_k})$ for which we have 
	\begin{equation*}
	\text{$\boldsymbol{y}_{n_k}\to \boldsymbol{y}$ uniformly on $\partial B(\boldsymbol{x}_0,r)\cup \partial B(\boldsymbol{x}_0,\widetilde{r})$,}
	\end{equation*}
	there holds
	\begin{equation*}
	B(\boldsymbol{x}_0,r)\in \bigcap_{k=1}^{\infty} \mathcal{U}^{\rm inj}_{\boldsymbol{y}_{n_k}}\cap \mathcal{U}^{\rm inj}_{\boldsymbol{y}}
	\end{equation*}
\end{proposition}

The following result yields the stability of the local inverse with respect to the weak convergence of deformations in $W^{1,p}(\Omega;\R^N)$.

\begin{theorem}[{\cite[Theorem 6.3, Claim (c)]{barchiesi.henao.moracorral}}]
	\label{thm:stability-inverse}
	Let $(\boldsymbol{y}_n)\subset \mathcal{Y}_p(\Omega)$ and $\boldsymbol{y}\in \mathcal{Y}_p(\Omega)$ be such that $\boldsymbol{y}_n \wk \boldsymbol{y}$ in $W^{1,p}(\Omega;\R^N)$. Let $U \in \mathcal{U}_{\boldsymbol{y}}^{\rm inj} \cap \bigcap_{n=1}^\infty \mathcal{U}_{\boldsymbol{y}_n}^{\rm inj}$ and let $V \subset \subset \im_{\rm T}(\boldsymbol{y},U)$ be open, so that $V \subset \subset \im_{\rm T}(\boldsymbol{y}_n,U)$ for $n \gg 1$. Then, there holds
	\begin{equation*}
	\text{$\boldsymbol{y}_{n,U}^{-1}\to \boldsymbol{y}_{U}^{-1}$ a.e. in $V$ and strongly in $L^1(V;\R^N)$.}   
	\end{equation*}
	Moreover, each Jacobian minor of $\boldsymbol{y}_{n,U}^{-1}$ converges weakly-$*$ in $\mathcal{M}_{\rm b}(V)$ to the corresponding Jacobian minor of $\boldsymbol{y}_{U}^{-1}$.
	In particular, if $(\det D \boldsymbol{y}_{n,U}^{-1})\subset L^1(V)$ is equi-integrable, then the Jacobian minors converge also weakly in $L^1(V)$. In this case, we have
	\begin{equation*}
	\text{$\boldsymbol{y}_{n,U}^{-1}\wk \boldsymbol{y}_{U}^{-1}$ in $W^{1,1}(V;\R^N)$.}
	\end{equation*}
\end{theorem}

\subsection{Global invertibility} The requirement of global injectivity of admissible deformations can be easily incorporated in our analysis. Our approach relies on the Ciarlet-Ne\v{c}as condition \cite{ciarlet.necas}. 

Let $\boldsymbol{y}\colon \Omega \to \R^N$ be almost everywhere approximately differentiable and recall Definition \ref{def:geometric-image}. The map $\boldsymbol{y}$ satisfies the Ciarlet-Ne\v{c}as condition  when the following holds:
\begin{equation*}
    \int_\Omega |\det \nabla \boldsymbol{y}|\,\d\boldsymbol{x}\leq \leb(\im_{\rm G}(\boldsymbol{y},\Omega)).
\end{equation*}
Note that, by Proposition \ref{prop:change-of-variable}, the opposite inequality is automatically satisfied, so that, in the previous equation, we could equivalently replace the inequality  with an equality. Also, we observe that the validity of the Ciarlet-Ne\v{c}as condition depends only on the equivalence class of $\boldsymbol{y}$.

The following result clarifies the relationship between almost everywhere injectivity and the Ciarlet-Ne\v{c}as condition and it is extremely useful in applications.

\begin{proposition}[{\cite[Proposition 1.5]{giacomini.ponsiglione}}]
\label{prop:ciarlet-necas}
Let $\boldsymbol{y}\colon \Omega \to \R^N$ be almost everywhere approximately differentiable. Then:
\begin{enumerate}[(i)]
    \item if $\boldsymbol{y}$ is almost everywhere injective in $\Omega$, then it satisfies the Ciarlet-Ne\v{c}as condition;
    \item if $\boldsymbol{y}$ satisfies the Ciarlet-Ne\v{c}as condition and also $\det \nabla \boldsymbol{y}\neq 0$ almost everywhere, then $\boldsymbol{y}$ is almost everywhere injective.
\end{enumerate}
\end{proposition}

Recalling \eqref{eqn:admissible-deformations}, the class of globally injective admissible deformations is defined as
\begin{equation}
\label{eqn:admissible-deformations-injective}
    \mathcal{Y}_p^{\rm inj}(\Omega)\coloneqq \left \{\boldsymbol{y}\in \mathcal{Y}_p(\Omega):\:\text{$\boldsymbol{y}$ a.e. injective in $\Omega$} \right \}.
\end{equation}

The following simple result provides the stability of almost everywhere injectivity with respect to the weak convergence in $W^{1,p}(\Omega;\R^N)$ and the weak convergence in $L^1(\Omega)$ of the Jacobian determinant.

\begin{corollary}
\label{cor:closure-injectivity}
Let $(\boldsymbol{y}_n)\subset \mathcal{Y}_p^{\rm inj}(\Omega)$. Suppose that there exists $\boldsymbol{y}\in \mathcal{Y}_p(\Omega)$ such that
\begin{equation*}
    \text{$\boldsymbol{y}_n \wk \boldsymbol{y}$ in $W^{1,p}(\Omega;\R^N)$, \quad $\det D \boldsymbol{y}_n \wk \det D \boldsymbol{y}$ in $L^1(\Omega)$.}
\end{equation*}
Then, $\boldsymbol{y}$ is almost everywhere injective, so that $\boldsymbol{y}\in \mathcal{Y}_p^{\rm inj}(\Omega)$.
\end{corollary}
\begin{proof}
By Proposition \ref{prop:ciarlet-necas}, for every $n \in \N$, there holds
\begin{equation}
\label{eqn:cn-seq}
    \int_\Omega \det D \boldsymbol{y}_n\,\d\boldsymbol{x}\leq \leb(\im_{\rm G}(\boldsymbol{y}_n,\Omega)).
\end{equation}
Thanks to claim (ii) of Proposition \ref{prop:weak-convergence}, passing to the limit, as $n \to \infty$, in \eqref{eqn:cn-seq}, we obtain
\begin{equation}
    \int_\Omega \det D \boldsymbol{y}\,\d\boldsymbol{x}\leq \leb(\im_{\rm G}(\boldsymbol{y},\Omega)).
\end{equation}
Since $\det D\boldsymbol{y}>0$ almost everywhere by assumption, this implies that $\boldsymbol{y}$ is almost everywhere injective by Proposition \ref{prop:ciarlet-necas}. 
\end{proof}

In the case of global injectivity, it is natural to wonder if the global inverse enjoys the same kind of regularity of the local inverse established in Theorem \ref{prop:regularity-inverse}. The answer is actually affermative.
Given $\boldsymbol{y}\in \mathcal{Y}^{\rm inj}_p(\Omega)$, it follows from \cite[Theorem 3.3]{henao.moracorral.regularity} that the global inverse $\boldsymbol{y}_{\Omega}^{-1}$, given by an arbitrary extension of $(\boldsymbol{y}\restr{\dom_{\rm G}(\boldsymbol{y},\Omega)})^{-1}\restr{\im_{\rm G}(\boldsymbol{y},\Omega)\cap \im_{\rm T}(\boldsymbol{y},\Omega)} $ to $\im_{\rm T}(\boldsymbol{y},\Omega)$, belongs to $W^{1,1}(\im_{\rm T}(\boldsymbol{y},\Omega);\R^N)$. However, this fact will not be exploited in our analysis.

\section{Static setting}
\label{sec:static}

In this section, we address the analysis of the variational model in the static setting. In the first subsection, we describe the mechanical model; in the second subsection, we present our compacntess result, namely Theorem \ref{thm:compactness}; finally, in the third subsection, we apply the compactness result to establish the existence of minimizers in the static case in Theorem \ref{thm:existence-minimizers}. The results of this section will also be instrumental for the analysis in the quasistatic setting.

\subsection{The static model}
Recall that $\Omega\subset \R^N$ is a bounded Lipschitz domain and $p>N-1$. Recall also the class of admissible deformations introduced in \eqref{eqn:admissible-deformations} and Definition \ref{def:topological-image}. For each $\boldsymbol{y}\in \mathcal{Y}_p(\Omega)$, magnetizations are given by maps $\boldsymbol{m}\in W^{1,2}(\im_{\rm T}(\boldsymbol{y},\Omega);\R^N)$ such that 
\begin{equation}
\label{eqn:magnetic-saturation}
    \text{$|\boldsymbol{m}\circ \boldsymbol{y}|\det D\boldsymbol{y}=1$ a.e. in $\Omega$.}    
\end{equation}
Thus, admissible states belong to the class
\begin{equation}
\label{eqn:admissible-states}
    \mathcal{Q}\coloneqq \left \{(\boldsymbol{y},\boldsymbol{m}):\:\boldsymbol{y}\in \mathcal{Y}_p(\Omega),\: \boldsymbol{m}\in W^{1,2}(\im_{\rm T}(\boldsymbol{y},\Omega);\R^{N}),\:\text{$|\boldsymbol{m}\circ \boldsymbol{y}|\det D\boldsymbol{y}=1$ a.e. in $\Omega$} \right \}.
\end{equation}
This class is endowed with the topology that makes the map
\begin{equation*}
    (\boldsymbol{y},\boldsymbol{m})\mapsto (\boldsymbol{y},\,\chi_{\im_{\rm T}(\boldsymbol{y},\Omega)}\boldsymbol{m},\,\chi_{\im_{\rm T}(\boldsymbol{y},\Omega)}D\boldsymbol{m})
\end{equation*}
from $\mathcal{Q}$ to $W^{1,p}(\Omega;\R^N)\times L^2(\R^N;\R^N)\times L^2(\R^N;\rnn)$ a homeomorphism onto its image, where the latter space is equipped with the product weak topology.
When considering only globally injective deformations, we restrict ourselves to the class
\begin{equation}
    \mathcal{Q}^{\rm inj}\coloneqq \left \{(\boldsymbol{y},\boldsymbol{m})\in \mathcal{Q}:\:\boldsymbol{y}\in \mathcal{Y}_p^{\rm inj}(\Omega) \right \},
\end{equation}
where we recall \eqref{eqn:admissible-deformations-injective}.

The magnetoelastic energy functional $E\colon \mathcal{Q}\to [0,+\infty]$ is defined, for $\boldsymbol{q}=(\boldsymbol{y},\boldsymbol{m})$, by setting
\begin{equation}
\label{eqn:magnetoelastic-energy}
    E(\boldsymbol{q})\coloneqq \int_\Omega W(D\boldsymbol{y},\boldsymbol{m}\circ \boldsymbol{y})\,\d\boldsymbol{x}+\int_{\im_{\rm T}(\boldsymbol{y},\Omega)} |D\boldsymbol{m}|^2\,\d\boldsymbol{\xi}+\frac{1}{2}\int_{\R^N} |D u_{\boldsymbol{m}}|^2\,\d\boldsymbol{\xi}.
\end{equation}
We denote the functionals corresponding to the three terms in the previous formula by $E^{\rm el}, E^{\rm exc}, E^{\rm mag}\colon \mathcal{Q} \to [0,+\infty)$, respectively.
The first term in \eqref{eqn:magnetoelastic-energy} represents the elastic energy and the elastic energy density $W\colon \R^{N\times N}\times \R^N \to [0,+\infty)$ is assumed to be continuous. Note that the composition $\boldsymbol{m}\circ \boldsymbol{y}$ is measurable and does not depend on the representatives of both $\boldsymbol{m}$ and $\boldsymbol{y}$. This comes from the Lusin property (N${}^{-1}$) of $\boldsymbol{y}$. On the function $W$, we take the following two assumptions:
\begin{enumerate}[(i)]
    \item \textbf{Coercivity:} there exist a constant $C>0$ and a Borel  function $\gamma\colon (0,+\infty)\to (0,+\infty)$ satisfying
    \begin{equation}
    \label{eqn:gamma}
        \lim_{h \to 0^+}h\,\gamma(h)=\lim_{h \to +\infty}\frac{\gamma(h)}{h}=+\infty
    \end{equation}
    such that there holds
    \begin{equation}
    \label{eqn:coercivity}
        \forall\,\boldsymbol{F}\in \rnn_+,\:\forall\,\boldsymbol{\lambda}\in \R^N, \quad W(\boldsymbol{F},\boldsymbol{\lambda})\geq C |\boldsymbol{F}|^p+\gamma(\det \boldsymbol{F});
    \end{equation}
    \item \textbf{Partial polyconvexity:} there exists a continuous function 
    \begin{equation*}
    \widehat{W}\colon \rnn_+\times \dots\times \rnn_+ \times (0,+\infty) \times \R^N\setminus \{\boldsymbol{0}\}\to [0,+\infty),  
    \end{equation*}
    such that the map $\boldsymbol{F}\mapsto \widehat{W}(\boldsymbol{F},\mathbf{M}_2(\boldsymbol{F}),\dots,\mathbf{M}_{N-2}(\boldsymbol{F}),\cof\boldsymbol{F},h,\boldsymbol{\lambda})$ is convex for every $h>0$ and $\boldsymbol{\lambda}\in \R^N\setminus \{\boldsymbol{0}\}$, and, for every $\boldsymbol{F}\in\rnn_+$ and $\boldsymbol{\lambda}\in \R^N \setminus \{\boldsymbol{0}\}$ with $|\boldsymbol{\lambda}|\det\boldsymbol{F}=1$,  there holds
    \begin{equation}
    \label{eqn:W-polyconvex} W(\boldsymbol{F},\boldsymbol{\lambda})=\widehat{W}(\boldsymbol{F},\mathbf{M}_2(\boldsymbol{F}),\dots,\mathbf{M}_{N-2}(\boldsymbol{F}),\cof\boldsymbol{F},\det\boldsymbol{F},\boldsymbol{\lambda}).
    \end{equation}
\end{enumerate}

Note that, the first limit condition in \eqref{eqn:gamma} entails the more natural one
\begin{equation}
\label{eqn:extreme-compression}
    \lim_{h \to 0^+} \gamma(h)=+\infty.
\end{equation}
This condition is fundamental from the modeling point of view, as it ensures that the elastic  energy blows up in case of extreme compressions. Additionally, in view of \eqref{eqn:extreme-compression}, we can extend the map $W$ by continuity as $W\colon \rnn\times \R^N \to [0,+\infty]$ by setting $W(\boldsymbol{F},\boldsymbol{\lambda})=+\infty$ for every $\boldsymbol{F}\in \rnn$ with $\det\boldsymbol{F}\leq 0$ and $\boldsymbol{\lambda}\in\R^N$.  Condition (ii) is close to the assumption of $W(\cdot,\boldsymbol{\lambda})$ being polyconvex for every $\boldsymbol{\lambda}\in\R^N$. On the one hand, the function $\widehat{W}$ does not depend on $\boldsymbol{\lambda}\in\R^N$. On the other hand,  the convexity requirement does not involve the determinant of $\boldsymbol{F}$. In (ii), for every $r\in \{1,\dots,N\}$,
we denote by $\mathbf{M}_r(\boldsymbol{F})$ the square matrix of order $\binom{N}{r}$ whose terms are all the minors of order $r$ of $\boldsymbol{F}$ with a given sign and order. A possible choice in this sense is described in \cite[Section 5.4]{dacorogna}. In particular, there hold $\mathbf{M}_1(\boldsymbol{F})=\boldsymbol{F}$, $\mathbf{M}_{N-1}(\boldsymbol{F})=\cof \boldsymbol{F}$ and $\mathrm{M}_N(\boldsymbol{F})=\det \boldsymbol{F}$.
To clarify, in the case $N=2$, the identity \eqref{eqn:W-polyconvex} reads
\begin{equation*}
W(\boldsymbol{F},\boldsymbol{\lambda})=\widehat{W}(\boldsymbol{F},\det\boldsymbol{F},\boldsymbol{\lambda}),
\end{equation*}
while, for $N=3$, we have
\begin{equation*}
W(\boldsymbol{F},\boldsymbol{\lambda})=\widehat{W}(\boldsymbol{F},\cof\boldsymbol{F},\det\boldsymbol{F},\boldsymbol{\lambda}).
\end{equation*}
An example of function $W$ satisfying conditions (i)--(ii) is provided in Example \ref{ex:W} below.

Note that the continuity of $\widehat{W}$ entails the one of $W$.  Since $\boldsymbol{y}$ has the Lusin property (N${}^{-1}$), the composition $\boldsymbol{m}\circ \boldsymbol{y}$ is measurable and well defined, in the sense that the equivalence class of the composition does not depend on the representative of $\boldsymbol{m}$ nor of the one of $\boldsymbol{y}$.

\begin{example}
\label{ex:W}
Define $\boldsymbol{L}\colon \S^{N-1}\to \rnn_+$ by setting $\boldsymbol{L}(\boldsymbol{z})\coloneqq \alpha \boldsymbol{z}\otimes \boldsymbol{z}+\beta(\boldsymbol{I}-\boldsymbol{z}\otimes \boldsymbol{z})$, where $\alpha,\beta>0$. Observe that $\det \boldsymbol{L}(\boldsymbol{z})=\alpha \beta^2$, so that  $\boldsymbol{L}(\boldsymbol{z})$ is invertible for every $\boldsymbol{z}\in \S^{N-1}$.  Let $\Phi\colon \rnn \to [0,+\infty)$ and define $W\colon \rnn \times \R^N \to [0,+\infty]$ by setting
\begin{equation*}
	W(\boldsymbol{F},\boldsymbol{\lambda})\coloneqq \Phi \big (\boldsymbol{L}(\boldsymbol{\lambda}\det \boldsymbol{F})^{-1}\boldsymbol{F} \big )
\end{equation*}
if $|\boldsymbol{\lambda}|\det\boldsymbol{F}=1$ and $W=+\infty$ otherwise. Suppose that $\Phi$ satisfies the following two conditions:
\begin{enumerate}[(i)]
	\item there exist two constants $C>0$ and $a>1$ such that
	\begin{equation*}
		\forall\,\boldsymbol{\Xi}\in\rnn, \quad \Phi(\boldsymbol{\Xi})\geq C |\boldsymbol{\Xi}|^p+\frac{1}{|\det \boldsymbol{\Xi}|^a};
	\end{equation*}
	\item there exists a continuous function
	\begin{equation*}
	\widehat{\Phi}\colon \rnn_+\times \dots\times \rnn_+ \times (0,+\infty) \times \R^N\setminus \{\boldsymbol{0}\}\to [0,+\infty),  
	\end{equation*}
	such that the map $\boldsymbol{\Xi}\mapsto \widehat{\Phi}(\boldsymbol{\Xi},\mathbf{M}_2(\boldsymbol{\Xi}),\dots,\mathbf{M}_{N-2}(\boldsymbol{\Xi}),\cof\boldsymbol{\Xi},d)$ is convex for every $d>0$, and, for every $\boldsymbol{\Xi}\in\rnn_+$,  there holds
	\begin{equation*}
 	\Phi(\boldsymbol{\Xi})=\widehat{\Phi}(\boldsymbol{\Xi},\mathbf{M}_2(\boldsymbol{\Xi}),\dots,\mathbf{M}_{N-2}(\boldsymbol{\Xi}),\cof\boldsymbol{\Xi},\det\boldsymbol{\Xi}).
	\end{equation*} 
\end{enumerate}
In this case, assumptions \eqref{eqn:gamma}--\eqref{eqn:W-polyconvex} are satisfied. Assumptions \eqref{eqn:gamma}--\eqref{eqn:coercivity} are checked by exploiting the uniform boundedness of  $\boldsymbol{L}^{-1}(\boldsymbol{z})$ for every $\boldsymbol{z}\in \S^{N-1}$. Property \eqref{eqn:W-polyconvex} is verified by using the Cauchy-Binet Formula for the minors of the product of two matrices similarly to \cite[Lemma 8.1]{barchiesi.henao.moracorral}.
\end{example}
 
The second term in \eqref{eqn:magnetoelastic-energy} constitutes the exchange energy. Note that the corresponding functional is well defined as the set $\im_{\rm T}(\boldsymbol{y},\Omega)$ depends only on the equivalence class of $\boldsymbol{y}$. The last term in \eqref{eqn:magnetoelastic-energy} represents the magnetostatic energy and involves the stray field potential $u_{\boldsymbol{m}}\colon \R^N\to\R^N$. This is given by a weak solution of the Maxwell equation:
\begin{equation}
\label{eqn:maxwell}
    \div \left (-Du_{\boldsymbol{m}}-\chi_{\im_{\rm T}(\boldsymbol{y},\Omega)}\boldsymbol{m} \right)=0\quad \text{in $\R^N$.}
\end{equation}
This means that $u_{\boldsymbol{m}}\in V^{1,2}(\R^N)$ satisfies
\begin{equation}
    \forall\,v\in V^{1,2}(\R^N),\quad \int_{\R^N} \left (-D u_{\boldsymbol{m}}-\chi_{\im_{\rm T}(\boldsymbol{y},\Omega)}\boldsymbol{m} \right) \cdot D v\,\d\boldsymbol{\xi}=0,
\end{equation}
where 
\begin{equation*}
V^{1,2}(\R^N)\coloneqq \left \{v\in L^2_{\loc}(\R^N):\: D v\in L^2(\R^N;\R^N)\right\}.    
\end{equation*}
Such a weak solution $u_{\boldsymbol{m}}$ of \eqref{eqn:maxwell} exists and is unique up to additive constants \cite[Proposition 8.8]{barchiesi.henao.moracorral}, so that the stray field $Du_{\boldsymbol{m}} \in L^2(\R^N;\R^N)$ and, in turn, the magnetostatic energy are well defined.

\subsection{Compactness} The following compactness result constitutes the first main result of the section. In particular, after having identified a limiting admissible state, we show that the magnetic saturation constraint is preserved and we prove the strong convergence in $L^1(\Omega;\R^N)$ of the compositions of magnetizations with deformations.  

\begin{theorem}[Compactness]
\label{thm:compactness}
Let $(\boldsymbol{q}_n)\subset \mathcal{Q}$ with $\boldsymbol{q}_n=(\boldsymbol{y}_n,\boldsymbol{m}_n)$. Suppose that 
\begin{equation}
\label{eqn:compactness-deformation}
\text{$\boldsymbol{y}_n\wk\boldsymbol{y}$ in $W^{1,p}(\Omega;\R^N)$, \qquad $\det D\boldsymbol{y}_n \wk \det D \boldsymbol{y}$ in $L^1(\Omega)$.}
\end{equation}
for some $\boldsymbol{y}\in \mathcal{Y}_p(\Omega)$.
Also, suppose that 
\begin{equation}
	\label{eqn:equiintegrability-det}
	\text{$(1/\det D\boldsymbol{y}_n)$ is equi-integrable,}
\end{equation}
\begin{equation}
	\label{eqn:gradient-magnetization-bounded}
	\text{$(\chi_{\im_{\rm T}(\boldsymbol{y}_n,\Omega)}D\boldsymbol{m}_n)$ is bounded in $L^2(\R^N;\rnn)$.}
\end{equation}
Then, there exists $\boldsymbol{m}\in W^{1,2}(\im_{\rm T}(\boldsymbol{y},\Omega);\R^N)$ such that $\boldsymbol{q}=(\boldsymbol{y},\boldsymbol{m})\in \mathcal{Q}$ and, up to subsequences, the following convergences hold: 
\begin{equation}
    \label{eqn:compactness-extension-m}
    \text{$\chi_{\im_{\rm T}(\boldsymbol{y}_n,\Omega)}\boldsymbol{m}_n \wk \chi_{\im_{\rm T}(\boldsymbol{y},\Omega)}\boldsymbol{m}$ in $L^2(\R^N;\R^N)$;}
\end{equation}
\begin{equation}
    \label{eqn:compactness-extension-Dm}
    \text{$\chi_{\im_{\rm T}(\boldsymbol{y}_n,\Omega)}D\boldsymbol{m}_n \wk \chi_{\im_{\rm T}(\boldsymbol{y},\Omega)}D\boldsymbol{m}$ in $L^2(\R^N;\rnn)$;}
\end{equation}
\begin{equation}
\label{eqn:compactness-composition}
    \text{$\boldsymbol{m}_n\circ \boldsymbol{y}_n \to \boldsymbol{m}\circ \boldsymbol{y}$ a.e. in $\Omega$ and in $L^1(\Omega;\R^N)$.}
\end{equation}
Moreover, up to subsequences, there hold
\begin{equation}
\label{eqn:improved-determinant}
    \text{$\det D \boldsymbol{y}_n \to \det D \boldsymbol{y}$ in $L^1(\Omega)$, \qquad $\frac{1}{\det D \boldsymbol{y}_n} \to \frac{1}{\det D \boldsymbol{y}}$ in $L^1(\Omega)$;}
\end{equation}
\begin{equation}
\label{eqn:compactness-lagrangian}
\boldsymbol{m}_n\circ \boldsymbol{y}_n\det D\boldsymbol{y}_n\to \boldsymbol{m}\circ \boldsymbol{y}\det D\boldsymbol{y}  \text{ in $L^q(\Omega;\R^N)$ for $1\leq q<\infty$;} 
\end{equation}
\begin{equation}
\label{eqn:improved-extension}
    \text{$\chi_{\im_{\rm T}(\boldsymbol{y}_n,\Omega)}\boldsymbol{m}_n \to \chi_{\im_{\rm T}(\boldsymbol{y},\Omega)}\boldsymbol{m}$ a.e. in $\R^N$ and in  $L^q(\R^N;\R^N)$ for every $1 \leq q<2$.}
\end{equation}
In particular, if $(\boldsymbol{q}_n)\subset \mathcal{Q}^{\,\rm inj}$, then $\boldsymbol{q}\in \mathcal{Q}^{\,\rm inj}$ and additionally, up to subsequences, there holds
\begin{equation}
\label{eqn:improved-injective}
    \text{$\chi_{\im_{\rm T}(\boldsymbol{y}_n,\Omega)}\boldsymbol{m}_n \to \chi_{\im_{\rm T}(\boldsymbol{y},\Omega)}\boldsymbol{m}$  in  $L^2(\R^N;\R^N)$.}
\end{equation}
\end{theorem}

The proof exploits the following result involving the topological image of nested balls.

\begin{lemma}[Topological image of nested balls]
\label{lem:topological-images-nested}
Let $\boldsymbol{y}\in \mathcal{Y}_p(\Omega)$ and let $\boldsymbol{x}_0\in \Omega$ be such that $\boldsymbol{y}$ is regularly approximately differentiable at $\boldsymbol{x}_0$ with $\det \nabla \boldsymbol{y}(\boldsymbol{x}_0)>0$. Then, there exist ${r}_{\boldsymbol{x}_0}>0$
with the following property: for every $(\boldsymbol{y}_n)\subset \mathcal{Y}_p(\Omega)$ such that $\boldsymbol{y}_n \wk \boldsymbol{y}$ in $W^{1,p}(\Omega;\R^N)$, there exist $r,r',r''>0$ with
\begin{equation*}
        {r}_{\boldsymbol{x}_0}<r''<r'<r<\dist(\boldsymbol{x}_0;\partial \Omega)
    \end{equation*}
and a subsequence $(\boldsymbol{y}_{n_k})$, possibly depending on $r,r'$ and $r''$, such that there hold:
    \begin{equation}
\label{eqn:regularity-nested}
    B(\boldsymbol{x}_0,r),B(\boldsymbol{x}_0,r'),B(\boldsymbol{x}_0,r'')\in  \bigcap_{k=1}^\infty \mathcal{U}_{\boldsymbol{y}_{n_k}}^{*,\,\rm inj} \cap \mathcal{U}_{\boldsymbol{y}}^{*,\,\rm inj},
\end{equation}
\begin{equation}
\label{eqn:uniform-convergence-nested}
    \text{${\boldsymbol{y}}_{n_k} \to {\boldsymbol{y}}$ uniformly on $\partial B(\boldsymbol{x}_0,r)\cup \partial B(\boldsymbol{x}_0,r')\cup \partial B(\boldsymbol{x}_0,r'')$,}
\end{equation}
\begin{equation}
    \label{eqn:topological-images-nested}
        \im_{\rm T}(\boldsymbol{y},B(\boldsymbol{x}_0,r''))\subset \subset \im_{\rm T}(\boldsymbol{y},B(\boldsymbol{x}_0,r')) \subset \subset \im_{\rm T}(\boldsymbol{y},B(\boldsymbol{x}_0,r)),    
    \end{equation}
\begin{equation}
    \label{eqn:topological-images-nested-sequence}
        \forall\,k \in \N, \quad \im_{\rm T}(\boldsymbol{y}_{n_k},B(\boldsymbol{x}_0,r''))\subset \subset \im_{\rm T}(\boldsymbol{y},B(\boldsymbol{x}_0,r'))\subset \subset \im_{\rm T}(\boldsymbol{y}_{n_k},B(\boldsymbol{x}_0,r)).
\end{equation}
In this case, up to subsequences, we have
\begin{equation*}
    \text{$\boldsymbol{y}_{n_k,B(\boldsymbol{x}_0,r)}^{-1} \to \boldsymbol{y}_{B(\boldsymbol{x}_0,r)}^{-1}$ a.e. in $\im_{\rm T}(\boldsymbol{y},B(\boldsymbol{x}_0,r'))$ and strongly in $L^1(\im_{\rm T}(\boldsymbol{y},B(\boldsymbol{x}_0,r'));\R^N)$.}
\end{equation*}
\end{lemma}
We stress that, in the previous statement, the radius $r_{\boldsymbol{x}_0}$ does not depend on the sequence $(\boldsymbol{y}_n)$, but only on $\boldsymbol{y}$ and $\boldsymbol{x}_0$. A graphical representation of \eqref{eqn:topological-images-nested}--\eqref{eqn:topological-images-nested-sequence} is provided in Figure \ref{fig:nested}.

\begin{figure}
	\centering
	\begin{tikzpicture}
	\node[anchor=south west,inner sep=0] at (0,0) {\includegraphics[width=.8\textwidth]{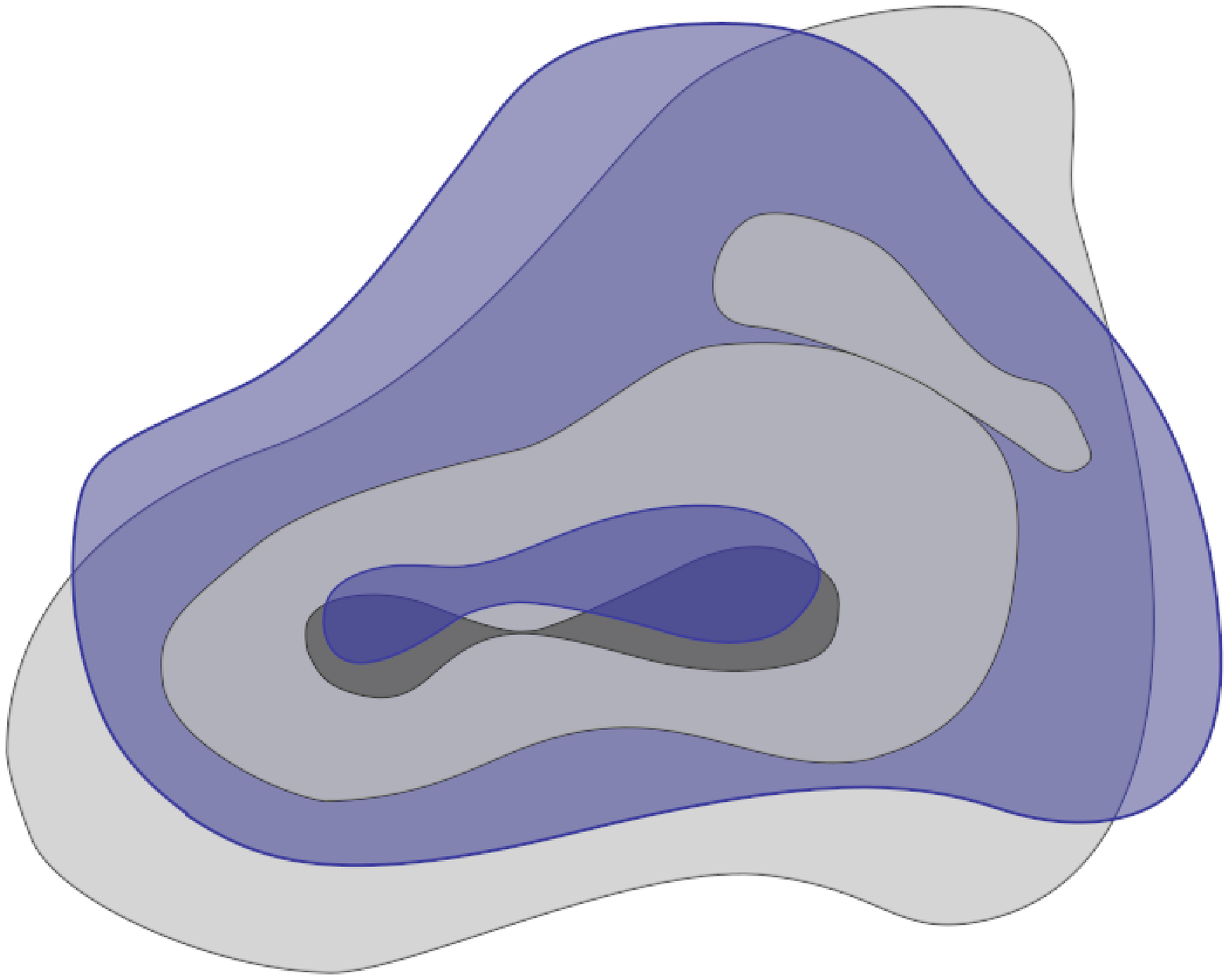}};
	\node at (8.2,5.3) { \color{bblue}  $\im_{\rm T}(\boldsymbol{y}_{n_k},B(\boldsymbol{x}_0,r''))$};
	\node at (7.4,3.3) { $\im_{\rm T}(\boldsymbol{y},B(\boldsymbol{x}_0,r''))$};
	\node at (7.7,5.8) {$\im_{\rm T}(\boldsymbol{y},B(\boldsymbol{x}_0,r'))$};
	\node at (12.8,6.2) {\color{bblue} $\im_{\rm T}(\boldsymbol{y}_{n_k},B(\boldsymbol{x}_0,r))$};
	\node at (12.3, 7.2) {$\im_{\rm T}(\boldsymbol{y},B(\boldsymbol{x}_0,r))$};
	\end{tikzpicture}
	\caption{The sets $\im_{\rm T}(\boldsymbol{y},B(\boldsymbol{x}_0,r''))$ and $\im_{\rm T}(\boldsymbol{y}_{n_k},B(\boldsymbol{x}_0,r''))$ are both contained in $\im_{\rm T}(\boldsymbol{y},B(\boldsymbol{x}_0,r'))$ which, in turn, is contained in the intersection of  $\im_{\rm T}(\boldsymbol{y},B(\boldsymbol{x}_0,r))$ and $\im_{\rm T}(\boldsymbol{y}_{n_k},B(\boldsymbol{x}_0,r))$.}
	\label{fig:nested}
\end{figure}

\begin{proof}[Proof of Lemma \ref{lem:topological-images-nested}]
	We begin by checking the following claim:
	\begin{equation}
	\label{eqn:claim-radii}
	\forall\,\alpha>0,\:\exists\,r,r'>0:\:0<r'<\frac{r}{2}<r<\alpha, \quad \mathscr{L}^1(P\cap (r',r))>0.
	\end{equation}
	Recall that $\Theta^1_+(P,0)=1$. Thus, by definition, there holds
	\begin{equation}
	\label{eqn:epsilon-delta}
	\forall\,\varepsilon>0,\:\exists\,\delta(\varepsilon)>0:\:\forall\,0<\rho<\delta(\varepsilon),\quad (1-\varepsilon)\rho <\mathscr{L}^1(P\cap (0,\rho))<(1+\varepsilon)\rho.
	\end{equation}
	Let $\varepsilon>0$. Given $0<r<\alpha \wedge \delta(\varepsilon)$, by \eqref{eqn:epsilon-delta}, we have
	\begin{equation}
	\label{eqn:d1}
	(1-\varepsilon)r<\mathscr{L}^1(P\cap (0,r))<(1+\varepsilon)r. 
	\end{equation}
	Similarly, for $0<r'<r/2\wedge \delta(\varepsilon)$, we obtain
	\begin{equation}
	\label{eqn:d2}
	(1-\varepsilon)r'<\mathscr{L}^1(P\cap (0,r'))<(1+\varepsilon)r'. 
	\end{equation}
	Combining \eqref{eqn:d1}--\eqref{eqn:d2}, we estimate
	\begin{equation*}
	\begin{split}
	\mathscr{L}^1(P\cap (r',r))&=\mathscr{L}^1(P \cap [r',r))=\mathscr{L}^1(P \cap (0,r))-\mathscr{L}^1(P \cap (0,r'))\\
	&>(1-\varepsilon)r-(1+\varepsilon)r'>\left(\frac{1}{2}-\frac{3}{2}\varepsilon\right)r.
	\end{split}
	\end{equation*}
	As we may assume $\varepsilon<1/3$, this proves \eqref{eqn:claim-radii}.
	
	First, by applying \eqref{eqn:claim-radii} with $\alpha=\dist(\boldsymbol{x}_0;\partial \Omega)$, we find $\rho_1,\rho_2>0$ such that 
	\begin{equation}
	\label{eqn:radii-start}
	{\rho}_2<\frac{\rho_1}{2}<\rho_1 <\dist(\boldsymbol{x}_0;\partial \Omega), \quad \mathscr{L}^1(P \cap (\rho_2,\rho_1))>0.
	\end{equation}
	Analogously, by a repeated application of \eqref{eqn:claim-radii}, we find $\rho_3,\rho_4,\rho_5,\rho_6,\rho_7,\rho_8,>0$ satisfying
	\begin{equation}
	\label{eqn:nested-radii-choice}
	\rho_8<\frac{\rho_7}{2}<\rho_7<\frac{\rho_6}{2}<\rho_6<\frac{\rho_5}{2}<\rho_5<\frac{\rho_4}{2}<\rho_4<\frac{\rho_3}{2}<\rho_3<\frac{\rho_2}{2}
	\end{equation}
	and
	\begin{equation}
	\label{eqn:nested-radii-measure}
	\mathscr{L}^1(P \cap (\rho_4,\rho_3))>0, \quad \mathscr{L}^1(P \cap (\rho_6,\rho_5))>0, \quad \mathscr{L}^1(P \cap (\rho_8,\rho_7))>0 
	\end{equation}
	by taking $\alpha=\rho_2/2$, $\alpha=\rho_4/2$ and $\alpha=\rho_6/2$. Then, we set $r_{\boldsymbol{x}_0}\coloneqq \rho_8>0$.
	
	Now, let $(\boldsymbol{y}_n)\subset \mathcal{Y}_p(\Omega)$ be such that $\boldsymbol{y}_n \wk \boldsymbol{y}$ in $W^{1,p}(\Omega;\R^N)$. Thanks to \eqref{eqn:radii-start} and \eqref{eqn:nested-radii-measure}, by Lemma \ref{lem:weak-convergence-boundary} and Lemma \ref{lem:abundance-star}, there exist $\widetilde{r}\in P\cap (\rho_2,\rho_1)$, $r\in P\cap (\rho_4,\rho_3)$, and a subsequence $(\boldsymbol{y}_{n_k})$, possibly depending on $\widetilde{r}$ and $r$, such that 
	\begin{equation*}
	B(\boldsymbol{x}_0,\widetilde{r}),B(\boldsymbol{x}_0,r)\in \bigcap_{k=1}^\infty \mathcal{U}_{\boldsymbol{y}_{n_k}}^* \cap \mathcal{U}_{\boldsymbol{y}}^*
	\end{equation*}
	and
	\begin{equation*}
	\text{${\boldsymbol{y}}_{n_k}\to {\boldsymbol{y}}$ uniformly on $\partial B(\boldsymbol{x}_0,\widetilde{r}) \cup \partial B(\boldsymbol{x}_0,r)$.}
	\end{equation*}
	Note that, by \eqref{eqn:nested-radii-choice}, there holds $r<\widetilde{r}/2$. Therefore, thanks to Proposition \ref{prop:local-invertibility-sequence}, we conclude that each map $\boldsymbol{y}_{n_k}$ and also $\boldsymbol{y}$, are almost everywhere injective on $B(\boldsymbol{x}_0,r)$.  
	
	At this point, taking into account \eqref{eqn:nested-radii-measure},  we apply once more Lemma \ref{lem:weak-convergence-boundary} and Lemma \ref{lem:abundance-star} to select $r'\in P \cap (\rho_6,\rho_5)$ and $r''\in P \cap (\rho_8,\rho_7)$, as well as a further subsequence that we do not relabel, so that  \eqref{eqn:regularity-nested}--\eqref{eqn:uniform-convergence-nested} are satisfied.
	
	From \eqref{eqn:nested-radii-choice}, we see that $r'<r/2$ and $r''<r'/2$. Thus, by claim (ii) of Lemma \ref{lem:regular-approximate-differentiability}, there hold
	\begin{equation}
	\label{eqn:bbs1}
	\im_{\rm T}(\boldsymbol{y},B(\boldsymbol{x}_0,r'))\cup {\boldsymbol{y}}(\partial B(\boldsymbol{x}_0,r')) \subset \im_{\rm T}(\boldsymbol{y},B(\boldsymbol{x}_0,r))
	\end{equation}
	and
	\begin{equation}
	\label{eqn:bbs2}
	\im_{\rm T}(\boldsymbol{y},B(\boldsymbol{x}_0,r''))\cup {\boldsymbol{y}}(\partial B(\boldsymbol{x}_0,r'')) \subset \im_{\rm T}(\boldsymbol{y},B(\boldsymbol{x}_0,r')).
	\end{equation}
	These yield \eqref{eqn:topological-images-nested}. The first inclusion in \eqref{eqn:topological-images-nested-sequence} follows from \eqref{eqn:bbs1} by claim (iii) of Lemma \ref{lem:topological-image-uniform-convergence}, whereas the second inclusion in \eqref{eqn:topological-images-nested-sequence} is deduced from \eqref{eqn:bbs2} by claim (i) of Lemma \ref{lem:topological-image-uniform-convergence}.
	Finally, the last statement follows by Theorem \ref{thm:stability-inverse}.
\end{proof}

We are now ready to present the proof of our compactness result.

\begin{proof}[Proof of Theorem \ref{thm:compactness}]
For convenience of the reader, the proof is subdivided into six steps.

\textbf{Step 1 (Compactness of magnetizations).} In this step, we argue as in \cite[Proposition 7.1]{barchiesi.henao.moracorral}. For every $n \in \N$, set 
\begin{equation*}
    \boldsymbol{v}_n\coloneqq \chi_{\im_{\rm T}(\boldsymbol{y}_n,\Omega)}\boldsymbol{m}_n, \qquad \boldsymbol{V}_n\coloneqq \chi_{\im_{\rm T}(\boldsymbol{y}_n,\Omega)}D\boldsymbol{m}_n.
\end{equation*}
By \eqref{eqn:gradient-magnetization-bounded}, $(\boldsymbol{V}_n)$ is bounded in $L^2(\R^N;\rnn)$.  Applying Corollary \ref{cor:topological-geometric-image} and  Proposition \ref{prop:change-of-variable}, we compute
\begin{equation*}
        \int_{\im_{\rm T}(\boldsymbol{y}_n,\Omega)} |\boldsymbol{m}_n|^2\,\d\boldsymbol{\xi}\leq \int_\Omega |\boldsymbol{m}_n \circ \boldsymbol{y}_n|^2\det D \boldsymbol{y}_n\,\d\boldsymbol{x}=\int_\Omega \frac{1}{\det D \boldsymbol{y}_n}\,\d\boldsymbol{x},
\end{equation*}
where the last equality is justified by the identity
\begin{equation}
\label{eqn:saturation-sequence}
    \text{$|\boldsymbol{m}_n \circ \boldsymbol{y}_n|\det D \boldsymbol{y}_n=1$ a.e. in $\Omega$,}
\end{equation}
which holds for every $n \in \N$.
Thus, $(\boldsymbol{v}_n)$ is bounded in $L^2(\R^N;\R^N)$ thanks to \eqref{eqn:equiintegrability-det}. Therefore, there exist $\boldsymbol{v}\in L^2(\R^N;\R^N)$ and $\boldsymbol{V}\in L^2(\R^N;\rnn)$ such that, up to subsequences, we have
\begin{equation}
\label{eqn:weak-convergence-v-and-V}
    \text{$\boldsymbol{v}_n \wk \boldsymbol{v}$ in $L^2(\R^N;\R^N)$, \qquad $\boldsymbol{V}_n \wk \boldsymbol{V}$ in $L^2(\R^N;\rnn)$.}
\end{equation}
Let $U \in  \mathcal{U}_{\boldsymbol{y}}^*$ and let $V \subset \subset \im_{\rm T}(\boldsymbol{y},U)$ be a smooth domain. By claim (i) of Proposition \ref{prop:weak-convergence}, for $n \gg 1$ there holds $V \subset \subset \im_{\rm T}(\boldsymbol{y}_n,\Omega)$. In this case, we have
\begin{equation}
\label{eqn:constant}
    \sup_{n \gg 1}||\boldsymbol{m}_n||_{W^{1,2}(V;\R^N)}= \sup_{n \gg 1} \left \{||\boldsymbol{v}_n||^2_{L^2(V;\R^N)}+||\boldsymbol{V}_n||^2_{L^2(V;\rnn)} \right \}^{1/2}\leq C,
\end{equation}
for some constant $C>0$ which is independent on $V$. Thus, there exist a subsequence $(\boldsymbol{m}_{n_k})$ and a map $\boldsymbol{m}\in W^{1,2}(V;\R^N)$, both possibly depending on $V$, such that 
\begin{equation}
\label{eqn:compactness-m}
    \text{$\boldsymbol{m}_{n_k} \wk \boldsymbol{m}$ in $W^{1,2}(V;\R^N)$, \qquad $\boldsymbol{m}_{n_k} \to \boldsymbol{m}$ in $L^q(V;\R^N)$ for every $1\leq q <2^*$ and a.e. in $V$.}
\end{equation}
Here, we applied the Sobolev Embedding Theorem.
In view of \eqref{eqn:weak-convergence-v-and-V}, there hold $\boldsymbol{m}=\boldsymbol{v}$  and $\boldsymbol{V}=D\boldsymbol{m}$ almost everywhere in $V$. In particular, we deduce that the map $\boldsymbol{m}$ does not depend on $V$. As $U$ and $V$ are arbitrary and the constant on the right-hand side of \eqref{eqn:constant} does not depend on these sets, we infer that $\boldsymbol{m}\in W^{1,2}(\im_{\rm T}(\boldsymbol{y},\Omega);\R^N)$ and also $\boldsymbol{v}=\boldsymbol{m}$  and $\boldsymbol{V}=D \boldsymbol{m}$ almost everywhere in $\im_{\rm T}(\boldsymbol{y},\Omega)$.

Now, we write $\im_{\rm T}(\boldsymbol{y},\Omega)=\bigcup_{i=1}^\infty V_i$, where 
$(V_i)$ denotes a sequence of smooth domains with the property that for every $i \in \N$ there exists $U_i \in \mathcal{U}_{\boldsymbol{y}}^*$ such that $V_i \subset \subset \im_{\rm T}(\boldsymbol{y},U_i)$. This is possible thanks to the Lindel\"{o}f property.
Arguing as in \eqref{eqn:compactness-m} and applying a diagonal argument, we select a subsequence of $(\boldsymbol{m}_n)$, that we do not relabel, such that
\begin{equation*}
     \forall\,i \in \N,\quad \text{$\boldsymbol{m}_{n} \wk \boldsymbol{m}$ in $W^{1,2}(V_i;\R^N)$, \quad $\boldsymbol{m}_{n} \to \boldsymbol{m}$ in $L^q(V_i;\R^N)$ for every $1\leq q<2^*$ and a.e. in $V_i$,}
\end{equation*}
This entails the following
\begin{equation}
\label{eqn:compactness-m-local}
\begin{split}
    &\text{$\forall\,V \subset \subset \im_{\rm T}(\boldsymbol{y},\Omega)$ open,\qquad $\boldsymbol{m}_{n}\wk\boldsymbol{m}$ in $W^{1,2}(V;\R^N)$,}\\
    &\hspace{-2mm}\text{$\boldsymbol{m}_{n}\to \boldsymbol{m}$ in $L^q(V;\R^N)$ for every $1\leq q <2^*$ and a.e. in $V$.}
\end{split}
\end{equation}
Note that, in the previous equation, the set $V$ does not have to be smooth.

\textbf{Step 2 (Magnetic saturation constraint).} In order to prove that $\boldsymbol{q}=(\boldsymbol{y},\boldsymbol{m})\in \mathcal{Q}$, we have to show that \eqref{eqn:magnetic-saturation} holds true. 
Let $\boldsymbol{x}_0\in \Omega$ be such that $\boldsymbol{y}$ is regularly approximately differentiable at $\boldsymbol{x}_0$ with $\det \nabla \boldsymbol{y}(\boldsymbol{x}_0)>0$. Let $r,r',r''>0$ and $(\boldsymbol{y}_{n_k})$ be given by Lemma \ref{lem:topological-images-nested}. For convenience, we set $B\coloneqq B(\boldsymbol{x}_0,r)$, $B'\coloneqq B(\boldsymbol{x}_0,r')$ and $B''\coloneqq B(\boldsymbol{x}_0,r'')$. Thus
\begin{equation}
\label{eqn:nested}
    \im_{\rm T}(\boldsymbol{y},B'')\subset \subset \im_{\rm T}(\boldsymbol{y},B')\subset \subset \im_{\rm T}(\boldsymbol{y},B)
\end{equation}
and, for every $k \in \N$, there holds
\begin{equation}
\label{eqn:nested-sequence}
    \im_{\rm T}(\boldsymbol{y}_{n_k},B'')\subset \subset \im_{\rm T}(\boldsymbol{y},B')\subset \subset \im_{\rm T}(\boldsymbol{y}_{n_k},B).
\end{equation}
Also, we have
\begin{equation}
\label{eqn:inverses-convergence-ae}
    \text{$\boldsymbol{y}_{n_k,B}^{-1}\to\boldsymbol{y}_{B}^{-1}$ a.e. in $\im_{\rm T}(\boldsymbol{y},B')$.}
\end{equation}
Let $\varphi\in C^0_{\rm c}(B'')$ and set $K\coloneqq \mathrm{supp}\,\varphi$. 
With the aid of Proposition \ref{prop:change-of-variable}, we compute
\begin{equation*}
    \begin{split}
        \int_{B''} \varphi\,\d\boldsymbol{x}&=\int_{B''} \varphi\,|\boldsymbol{m}_{n_k} \circ \boldsymbol{y}_{n_k}|\,\det D\boldsymbol{y}_{n_k}\,\d\boldsymbol{x}\\
        &=\int_{\im_{\rm G}(\boldsymbol{y}_{n_k},B'')} \varphi\circ \boldsymbol{y}_{n_k,B}^{-1}\,|\boldsymbol{m}_{n_k}|\,\d\boldsymbol{\xi}\\
        &=\int_{\im_{\rm T}(\boldsymbol{y},B')} \varphi\circ \boldsymbol{y}_{n_k,B}^{-1}\,|\boldsymbol{m}_{n_k}|\,\d\boldsymbol{\xi}.
    \end{split}
\end{equation*}
Here, we exploited \eqref{eqn:saturation-sequence} together with
the identities $\varphi \circ \boldsymbol{y}_{n_k,B}^{-1}=0$ on $\im_{\rm G}(\boldsymbol{y}_{n_k},B'')\setminus \im_{\rm G}(\boldsymbol{y}_{n_k},K)$ and $\im_{\rm G}(\boldsymbol{y}_{n_k},B'')\subset \im_{\rm T}(\boldsymbol{y},B')$, the latter being justified by  \eqref{eqn:nested-sequence} and claim (iii) of Corollary \ref{cor:topological-geometric-image}. Given \eqref{eqn:compactness-m-local} and \eqref{eqn:inverses-convergence-ae}, taking the limit, as $k \to \infty$, in the previous equation yields
\begin{equation*}
    \begin{split}
        \int_{B''} \varphi\,\d\boldsymbol{x}&=\int_{\im_{\rm T}(\boldsymbol{y},B')} \varphi\circ \boldsymbol{y}_{B}^{-1}\,|\boldsymbol{m}|\,\d\boldsymbol{\xi}\\
        &=\int_{\im_{\rm G}(\boldsymbol{y},B'')} \varphi\circ \boldsymbol{y}_{B}^{-1}\,|\boldsymbol{m}|\,\d\boldsymbol{\xi}\\
        &=\int_{B''} \varphi\,|\boldsymbol{m}\circ \boldsymbol{y}|\,\det D\boldsymbol{y}\,\d\boldsymbol{x}.
    \end{split}
\end{equation*}
Here, we applied Proposition \ref{prop:change-of-variable} and we exploited the identities $\boldsymbol{\varphi} \circ \boldsymbol{y}_{B}^{-1}=0$ on $\im_{\rm G}(\boldsymbol{y},B'')\setminus \im_{\rm G}(\boldsymbol{y},K)$ and $\im_{\rm G}(\boldsymbol{y},B'')\subset \im_{\rm T}(\boldsymbol{y},B')$, the latter being justified by \eqref{eqn:nested} and claim (iii) of Corollary \ref{cor:topological-geometric-image}. As $\varphi$ is arbitrary, this entails 
\begin{equation*}
    \text{$|\boldsymbol{m}\circ \boldsymbol{y}|\det D\boldsymbol{y}=1$ a.e. in $B''$.}
\end{equation*}
As almost every point in $\Omega$ has the same property of $\boldsymbol{x}_0$, this proves \eqref{eqn:magnetic-saturation}.

\textbf{Step 3 (Convergence of the trivial extensions).} In order to establish \eqref{eqn:compactness-extension-m}--\eqref{eqn:compactness-extension-Dm}, it is sufficient to show that $\boldsymbol{v}=\boldsymbol{0}$ and $\boldsymbol{V}=\boldsymbol{O}$ almost everywhere in $\R^N \setminus \im_{\rm T}(\boldsymbol{y},\Omega)$. To do this, we argue as in \cite[Proposition 7.1]{barchiesi.henao.moracorral}. Recall \eqref{eqn:weak-convergence-v-and-V} and let $Y\subset \R^N \setminus \im_{\rm T}(\boldsymbol{y},\Omega)$ be measurable with $\leb(Y)<+\infty$. Then,
\begin{equation*}
    \left |\int_Y \boldsymbol{v}\,\d\boldsymbol{\xi} \right |=\lim_{n \to \infty} \left |\int_Y \boldsymbol{v}_n\,\d\boldsymbol{\xi} \right |\leq \lim_{n \to \infty} \int_{\im_{\rm T}(\boldsymbol{y}_n,\Omega)\setminus \im_{\rm T}(\boldsymbol{y},\Omega)} |\boldsymbol{v}_n|\,\d\boldsymbol{\xi}=0,
\end{equation*}
where, in the last equality, we exploited  claim (ii) of Corollary \ref{cor:topological-geometric-image} and the equi-integrability of $(\boldsymbol{v}_n)$. This last property follows from the boundedness of $(\boldsymbol{v}_n)$ in $L^2(\R^N;\R^N)$. The identity $\boldsymbol{V}=\boldsymbol{O}$ almost everywhere in $\R^N \setminus \im_{\rm T}(\boldsymbol{y},\Omega)$ is proved in the same way.

\textbf{Step 4 (Convergence of compositions).} From the equi-integrability of $(1/\det D \boldsymbol{y}_n)$ and from \eqref{eqn:saturation-sequence}, we see that $(\boldsymbol{m}_n \circ \boldsymbol{y}_n)$ is equi-integrable. If we show that
\begin{equation}
\label{eqn:composition-convergence-ae}
    \text{$\boldsymbol{m}_n \circ \boldsymbol{y}_n \to \boldsymbol{m}\circ \boldsymbol{y}$ \:a.e. in $\Omega$,}
\end{equation}
then \eqref{eqn:compactness-composition} follows by the Vitali Convergence Theorem.

First, we argue locally. Let $\boldsymbol{x}_0\in \Omega$ be such that $\boldsymbol{y}$ is regularly approximately differentiable at $\boldsymbol{x}_0$ with $\det \nabla \boldsymbol{y}(\boldsymbol{x}_0)>0$. Let $r,r',r''>0$ and $(\boldsymbol{y}_{n_k})$ be given by Lemma \ref{lem:topological-images-nested}. As in Step 2, we set $B\coloneqq B(\boldsymbol{x}_0,r), B'\coloneqq B(\boldsymbol{x}_0,r')$ and $B''\coloneqq B(\boldsymbol{x}_0,r'')$, so that \eqref{eqn:nested}--\eqref{eqn:inverses-convergence-ae} hold true.
Fix $1<q<2^*-1$. By \eqref{eqn:compactness-m-local}, the sequence $(\boldsymbol{m}_{n_k})$ admits a majorant  in $L^{q+1}(\im_{\rm T}(\boldsymbol{y},B');\R^N)$. Exploiting \eqref{eqn:saturation-sequence} and applying Proposition \ref{prop:change-of-variable}, we compute
\begin{equation*}
    \begin{split}
        \int_{B''}|\boldsymbol{m}_{n_k} \circ \boldsymbol{y}_{n_k}|^q\,\d\boldsymbol{x}&=\int_{B''} |\boldsymbol{m}_{n_k} \circ \boldsymbol{y}_{n_k}|^{q+1}\,\det D\boldsymbol{y}_{n_k}\,\d\boldsymbol{x}\\
        &=\int_{\im_{\rm G}(\boldsymbol{y}_{n_k},B'')} |\boldsymbol{m}_{n_k}|^{q+1}\,\d\boldsymbol{\xi}\\
        &=\int_{\im_{\rm T}(\boldsymbol{y},B')} \chi_{\im_{\rm G}(\boldsymbol{y}_{n_k},B'')}\,|\boldsymbol{m}_{n_k}|^{q+1}\,\d\boldsymbol{\xi}.
    \end{split}
\end{equation*}
By Lemma \ref{lem:topological-images-nested}, there holds  ${\boldsymbol{y}}_{n_k}\to{\boldsymbol{y}}$ uniformly on $\partial B''$. Thus,  by claim (iv) of  Lemma \ref{lem:topological-image-uniform-convergence}, we have
\begin{equation*}
    \text{$\chi_{\im_{\rm G}(\boldsymbol{y}_{n_k},B'')}\to \chi_{\im_{\rm G}(\boldsymbol{y},B'')}$ a.e. in $\R^N$.}
\end{equation*}
This, together with \eqref{eqn:compactness-m-local}, allows us to pass to the limit, as $k \to \infty$, in the previous equation thanks to the  Dominated Convergence Theorem. Employing once again Proposition \ref{prop:change-of-variable} and recalling \eqref{eqn:magnetic-saturation}, we obtain
\begin{equation}
\label{eqn:convergence-composition-Lq-norm}
    \begin{split}
        \lim_{k \to \infty} \int_{B''}|\boldsymbol{m}_{n_k} \circ \boldsymbol{y}_{n_k}|^q\,\d\boldsymbol{x}&=\int_{\im_{\rm T}(\boldsymbol{y},B')} \chi_{\im_{\rm G}(\boldsymbol{y},B'')}\,|\boldsymbol{m}|^{q+1}\,\d\boldsymbol{\xi}\\
        &=\int_{\im_{\rm G}(\boldsymbol{y},B'')} |\boldsymbol{m}|^{q+1}\,\d\boldsymbol{\xi}\\
        &=\int_{B''} |\boldsymbol{m}\circ \boldsymbol{y}|^{q+1}\,\det D \boldsymbol{y}\,\d\boldsymbol{x}\\
        &=\int_{B''} |\boldsymbol{m} \circ \boldsymbol{y}|^q\,\d\boldsymbol{x}.
    \end{split}
\end{equation}
In particular, we see that $(\boldsymbol{m}_{n_k} \circ \boldsymbol{y}_{n_k})$ is bounded in $L^q(B'';\R^N)$. Hence, there exist a subsequence of $(\boldsymbol{m}_{n_k} \circ \boldsymbol{y}_{n_k})$, that we do not relabel, and a map $\boldsymbol{\lambda}\in L^q(B'';\R^N)$, both possibly depending on $B''$, such that $\boldsymbol{m}_{n_k}\circ \boldsymbol{y}_{n_k} \wk \boldsymbol{\lambda}$ in $L^q(B'';\R^N)$. We claim that $\boldsymbol{\lambda}=\boldsymbol{m}\circ \boldsymbol{y}$ almost everywhere in $B''$. To see this, let $\boldsymbol{\varphi}\in C^0_{\rm c}(B'';\R^N)$. 
With the aid of Proposition \ref{prop:change-of-variable} and exploiting \eqref{eqn:saturation-sequence}, we compute
\begin{equation*}
\begin{split}
    \int_{B''} \boldsymbol{m}_{n_k} \circ
    \boldsymbol{y}_{n_k} \cdot \boldsymbol{\varphi}\,\d\boldsymbol{x}&=\int_{B''} |\boldsymbol{m}_{n_k} \circ
    \boldsymbol{y}_{n_k}|\,\boldsymbol{m}_{n_k} \circ
    \boldsymbol{y}_{n_k} \cdot \boldsymbol{\varphi}\,\det D \boldsymbol{y}_{n_k}\,\d\boldsymbol{x}\\
    &=\int_{\im_{\rm G}(\boldsymbol{y}_{n_k},B'')} |\boldsymbol{m}_{n_k}|\,\boldsymbol{m}_{n_k} \cdot \boldsymbol{\varphi}\circ \boldsymbol{y}_{n_k,B}^{-1}\,\d\boldsymbol{\xi}\\
    &=\int_{\im_{\rm T}(\boldsymbol{y},B')} |\boldsymbol{m}_{n_k}|\,\boldsymbol{m}_{n_k} \cdot \boldsymbol{\varphi}\circ \boldsymbol{y}_{n_k,B}^{-1}\,\d\boldsymbol{\xi}
\end{split}
\end{equation*}
Recalling \eqref{eqn:compactness-m-local} and \eqref{eqn:inverses-convergence-ae}, applying the Dominated Convergence Theorem, we pass to the limit, as $k \to \infty$, in the previous equation. Note that this is justified as $(|\boldsymbol{m}_{n_k}|^2)$ admits a integrable majorant by \eqref{eqn:compactness-m-local} and $(\boldsymbol{\varphi}\circ \boldsymbol{y}_{n_k,B}^{-1})$ is uniformly bounded. 
We obtain
\begin{equation*}
    \begin{split}
        \int_{B''} \boldsymbol{\lambda} \cdot \boldsymbol{\varphi}\,\d\boldsymbol{x}&=\int_{\im_{\rm T}(\boldsymbol{y},B')} |\boldsymbol{m}|\,\boldsymbol{m}\cdot \boldsymbol{\varphi}\circ \boldsymbol{y}_{B}^{-1} \,\d\boldsymbol{\xi}\\
        &=\int_{\im_{\rm G}(\boldsymbol{y},B'')} |\boldsymbol{m}|\,\boldsymbol{m}\cdot \boldsymbol{\varphi}\circ \boldsymbol{y}_{B}^{-1} \,\d\boldsymbol{\xi}\\
        &=\int_{B''} |\boldsymbol{m}\circ \boldsymbol{y}|\,\boldsymbol{m}\circ \boldsymbol{y}\cdot \boldsymbol{\varphi}\,\det D \boldsymbol{y}\,\d\boldsymbol{x}\\
        &=\int_{B''} \boldsymbol{m}\circ \boldsymbol{y}\cdot \boldsymbol{\varphi}\,\d\boldsymbol{x}
    \end{split}
\end{equation*}
This entails $\boldsymbol{\lambda}=\boldsymbol{m}\circ \boldsymbol{y}$ almost everywhere in $B''$. Therefore, by the Urysohn property, the whole sequence $(\boldsymbol{m}_{n_k} \circ \boldsymbol{y}_{n_k})$ converges weakly in $L^q(B'';\R^N)$. 
As $q>1$, the convergence in \eqref{eqn:convergence-composition-Lq-norm} yields 
\begin{equation*}
    \text{$\boldsymbol{m}_{n_k} \circ \boldsymbol{y}_{n_k} \to \boldsymbol{m}\circ \boldsymbol{y}$ in $L^q(B'';\R^N)$.}
\end{equation*}
Thus, up to extraction of a further subsequence, we have
\begin{equation*}
    \text{$\boldsymbol{m}_{n_k} \circ \boldsymbol{y}_{n_k} \to \boldsymbol{m}\circ \boldsymbol{y}$ a.e. in $B''$.}
\end{equation*}
At this point, we stress that both the radius $r''>0$ of $B''$ and the sequence of indices $(n_k)$ depend on the point $\boldsymbol{x}_0$.

To achieve \eqref{eqn:composition-convergence-ae}, we employ a diagonal argument. Denote by $\Omega_0$ the set of points $\boldsymbol{x}_0\in \Omega$ such that $\boldsymbol{y}$ is regularly approximately differentiable at $\boldsymbol{x}_0$ with $\det \nabla \boldsymbol{y}(\boldsymbol{x}_0)>0$. In this case, $\leb(\Omega \setminus \Omega_0)=0$. For every $\boldsymbol{x}_0\in \Omega_0$, let ${r}_{\boldsymbol{x}_0}>0$ be given by Lemma \ref{lem:topological-images-nested}. Thanks to the Lindel\"{o}f property, there exists a sequence $(\boldsymbol{x}_i)\subset \Omega_0$ such that
\begin{equation}
\label{eqn:Omega_0}
    \Omega_0 \subset \bigcup_{i=1}^\infty B(\boldsymbol{x}_i,{r}_{\boldsymbol{x}_i}).
\end{equation}
Now, by means of the previous local argument, for every $i \in \N$, we  select a subsequence, indexed by $\left (n_k^{(i)}\right )$, such that
\begin{equation*}
\text{$\boldsymbol{m}_{n_k^{(i)}} \circ \boldsymbol{y}_{n_k^{(i)}}\to \boldsymbol{m}\circ \boldsymbol{y}$ a.e. in $B(\boldsymbol{x}_i,r_{\boldsymbol{x}_i})$.}
\end{equation*}
For each $i \in \N$, we choose $\left (n_k^{(i)}\right )$ to be a subsequence of $\left (n_k^{(i-1)}\right )$. In this way, we have
\begin{equation*}
\forall\,j \in \N:j\leq i,\quad \text{$\boldsymbol{m}_{n_k^{(i)}} \circ \boldsymbol{y}_{n_k^{(i)}}\to \boldsymbol{m}\circ \boldsymbol{y}$ a.e. in $B(\boldsymbol{x}_j,r_{\boldsymbol{x}_j})$.}
\end{equation*}
Thus, setting $n_j\coloneqq n^{(j)}_j$ for every $j \in \N$, there holds
\begin{equation*}
\forall\,i \in \N,\quad \text{$\boldsymbol{m}_{n_j} \circ \boldsymbol{y}_{n_j}\to \boldsymbol{m}\circ \boldsymbol{y}$ a.e. in $B(\boldsymbol{x}_i,r_{\boldsymbol{x}_i})$,}
\end{equation*}
which, in view of \eqref{eqn:Omega_0}, yields \eqref{eqn:composition-convergence-ae} for the subsequence indexed by $(n_j)$.

\textbf{Step 5 (Improved convergences).} From \eqref{eqn:compactness-composition}, we have
\begin{equation}
\label{eqn:L^1-distance-inverse-determinant}
    \int_\Omega \left | \frac{1}{\det D \boldsymbol{y}_n}-\frac{1}{\det D \boldsymbol{y}}   \right |\,\d\boldsymbol{x}=\int_\Omega \big | |\boldsymbol{m}_n \circ \boldsymbol{y}_n|- |\boldsymbol{m}\circ \boldsymbol{y}|\big |\,\d\boldsymbol{x}\leq \int_\Omega |\boldsymbol{m}_n \circ \boldsymbol{y}_n-\boldsymbol{m}\circ \boldsymbol{y}|\,\d\boldsymbol{x}.
\end{equation}
Since the right-hand side goes to zero, as $n \to \infty$, the second convergence in \eqref{eqn:improved-determinant} follows. Thus, up to subsequences, $\det D \boldsymbol{y}_n \to \det D \boldsymbol{y}$ almost everywhere and, by the Vitali Convergence Theorem, this entails the first convergence in  \eqref{eqn:improved-determinant}. 
Claim \eqref{eqn:compactness-lagrangian} immediately follows from \eqref{eqn:compactness-composition}--\eqref{eqn:improved-determinant} by taking into account \eqref{eqn:magnetic-saturation} and \eqref{eqn:saturation-sequence}.
To prove \eqref{eqn:improved-extension}, observe that, by claim (ii) of Proposition \ref{prop:weak-convergence} and by \eqref{eqn:compactness-m-local}, there holds
\begin{equation*}
    \text{$\chi_{\im_{\rm T}(\boldsymbol{y}_n,\Omega)}\boldsymbol{m}_n \to \chi_{\im_{\rm T}(\boldsymbol{y},\Omega)}\boldsymbol{m}$ a.e. in $\R^N$.}
\end{equation*}
As $(\chi_{\im_{\rm T}(\boldsymbol{y}_n,\Omega)}\boldsymbol{m}_n)$ is bounded in $L^2(\R^N;\R^N)$, by De la Vallée-Poussin Criterion \cite[Theorem 2.29]{fonseca.leoni}, the same sequence is $q$-equi-integrable for every $1\leq q<2$, so that \eqref{eqn:improved-extension} follows by applying the Vitali Convergence Theorem.

\textbf{Step 6 (Global injectivity).} If $(\boldsymbol{q}_n)\subset\mathcal{Q}^{\,\rm inj}$, then $\boldsymbol{q}\in \mathcal{Q}^{\,\rm inj}$ thanks to \eqref{eqn:compactness-deformation} and Corollary \ref{cor:closure-injectivity}. Finally, we show \eqref{eqn:improved-injective}. Applying Proposition \ref{prop:change-of-variable} and Theorem \ref{thm:deg=m}, for every $n \in \N$ we compute
\begin{equation*}
    \int_{\im_{\rm T}(\boldsymbol{y}_n,\Omega)}|\boldsymbol{m}_n|^2\,\d\boldsymbol{\xi}=\int_\Omega |\boldsymbol{m}_n \circ \boldsymbol{y}_n|^2\det D \boldsymbol{y}_n\,\d\boldsymbol{x}=\int_\Omega |\boldsymbol{m}_n \circ \boldsymbol{y}_n|\,\d\boldsymbol{x},
\end{equation*}
where we exploited \eqref{eqn:saturation-sequence}. Analogously, taking into account the injectivity of $\boldsymbol{y}$, we have
\begin{equation*}
    \int_{\im_{\rm T}(\boldsymbol{y},\Omega)}|\boldsymbol{m}|^2\,\d\boldsymbol{\xi}=\int_\Omega |\boldsymbol{m} \circ \boldsymbol{y}|^2\det D \boldsymbol{y}\,\d\boldsymbol{x}=\int_\Omega |\boldsymbol{m} \circ \boldsymbol{y}|\,\d\boldsymbol{x}.
\end{equation*}
Thus, from \eqref{eqn:compactness-composition}, we obtain
\begin{equation*}
    \lim_{n \to \infty} \|\chi_{\im_{\rm T}(\boldsymbol{y}_n,\Omega)}\boldsymbol{m}_n\|_{L^2(\R^N;\R^N)}=\|\chi_{\im_{\rm T}(\boldsymbol{y},\Omega)}\boldsymbol{m}\|_{L^2(\R^N;\R^N)}
\end{equation*}
which, together with \eqref{eqn:compactness-extension-m}, yields \eqref{eqn:improved-injective}. 
\end{proof}

We conclude the subsection with a remark explaining how the techniques employed here can be used to prove the strong convergence of the composition of magnetizations with deformations when the magnetic saturation constraint is formulated in the deformed configuration as in \cite{barchiesi.henao.moracorral}.The approach is described in the following remark.

\begin{remark}[Compactness for sphere-valued magnetizations]
	\label{rem:sphere}
	Let $(\boldsymbol{y}_n)\subset \mathcal{Y}_p(\Omega)$ and let $(\boldsymbol{m}_n)$ with $\boldsymbol{m}_n\in W^{1,2}(\im_{\rm T}(\boldsymbol{y}_n,\Omega);\mathbb{S}^{N-1})$. Suppose that \eqref{eqn:compactness-deformation} holds for some $\boldsymbol{y}\in\mathcal{Y}_p(\Omega)$. Additionally, suppose that \eqref{eqn:gradient-magnetization-bounded} holds true and that there exists  $\gamma \colon (0,+\infty)\to(0,+\infty)$ Borel satisfying \eqref{eqn:extreme-compression} such that
	\begin{equation}
	\label{eqn:gb}
	\sup_{n \in \N} \|\gamma(\det D \boldsymbol{y}_n)\|_{L^1(\Omega)}<+\infty.
	\end{equation}
	Then, there exists $\boldsymbol{m}\in W^{1,2}(\im_{\rm T}(\boldsymbol{y},\Omega);\mathbb{S}^{N-1})$ such that the convergences in \eqref{eqn:compactness-extension-m}--\eqref{eqn:compactness-extension-Dm} hold true.
	The identification of $\boldsymbol{m}$ as well as \eqref{eqn:compactness-extension-m}--\eqref{eqn:compactness-extension-Dm} have been achieved in \cite[Proposition 7.1]{barchiesi.henao.moracorral}. We only show how to prove the convergence of compositions.
	
	Let $\boldsymbol{x}_0\in\Omega$ be such that $\boldsymbol{y}$ is regularly approximately differentiable at $\boldsymbol{x}_0$ with $\det \nabla \boldsymbol{y}(\boldsymbol{x}_0)>0$. Also, let $r,r',r'',r_{\boldsymbol{x}_0}>0$ and $(\boldsymbol{y}_{n_k})$ be given by Lemma \ref{lem:topological-images-nested}. Define $B,B'$ and $B''$ be as in Step 2 of the proof of Theorem \ref{thm:compactness}, so that \eqref{eqn:nested}--\eqref{eqn:inverses-convergence-ae} hold true.
	As in \cite[Proposition 7.8]{barchiesi.henao.moracorral}, we 
	introduce the function $\widehat{\gamma}\colon (0,+\infty)\to (0,+\infty)$ by setting $\widehat{\gamma}(z)\coloneqq z\,\gamma(1/z)$. In this case,  \eqref{eqn:extreme-compression} yields
	\begin{equation*}
	\lim_{z \to +\infty} \frac{\widehat{\gamma}(z)}{z}=+\infty.
	\end{equation*}
	Applying Proposition \ref{prop:change-of-variable} and Theorem \ref{thm:deg=m}, for every $k \in \N$, we compute
	\begin{equation*}
	\begin{split}
	\int_{\im_{\rm T}(\boldsymbol{y}_{n_k},B)} \widehat{\gamma}(\det D\boldsymbol{y}_{n_k,B}^{-1})\,\d\boldsymbol{\xi}
	&= \int_{\im_{\rm T}(\boldsymbol{y}_{n_k},B)} \gamma\left(\frac{1}{\det D\boldsymbol{y}_{n_k,B}^{-1}}\right)\,\det D\boldsymbol{y}_{n_k,B}^{-1}\,\d\boldsymbol{\xi}\\
	&=\int_B \gamma(\det D \boldsymbol{y}_{n_k})\,\d\boldsymbol{x},
	\end{split}
	\end{equation*}
	where the right-hand side is uniformly bounded thanks to  \eqref{eqn:gb}.
	Here, we used the identity
	\begin{equation*}
	\text{$\det D\boldsymbol{y}_{n_k,B}^{-1}=(\det D \boldsymbol{y}_{n_k})^{-1}\circ \boldsymbol{y}_{n_k,B} ^{-1}$ a.e. in $\im_{\rm T}(\boldsymbol{y}_{n_k},B)$.}
	\end{equation*}
	In view of \eqref{eqn:nested}, there holds
	\begin{equation*}
	\sup_{k \in \N} \int_{\im_{\rm T}(\boldsymbol{y},B')} \widehat{\gamma}(\det D\boldsymbol{y}_{n_k,B}^{-1})\,\d\boldsymbol{\xi}\leq \sup_{k \in \N} \int_{\im_{\rm T}(\boldsymbol{y}_{n_k},B)} \widehat{\gamma}(\det D\boldsymbol{y}_{n_k,B}^{-1})\,\d\boldsymbol{\xi},
	\end{equation*}
	where 
	the sequence $(\det D\boldsymbol{y}_{n_k,B}^{-1})$ is equi-integrable on $\im_{\rm T}(\boldsymbol{y},B')$ by De la Vallée-Poussin Criterion \cite[Theorem 2.29]{fonseca.leoni}.
	At this point, Theorem \ref{thm:stability-inverse} ensures that, up to subsequence, there holds
	\begin{equation}
	\label{eqn:inverses-determinant}
	\text{$\det D\boldsymbol{y}_{n_k,B}^{-1}\wk \det D \boldsymbol{y}_B^{-1}$ in $L^1(\im_{\rm T}(\boldsymbol{y},B'))$.}
	\end{equation}
	Let $\boldsymbol{\varphi}\in C^0_{\rm c}(B'';\R^N)$.
	Applying Proposition \ref{prop:inverse-change-of-variable}, we compute
	\begin{equation*}
	\begin{split}
	\int_{B''} \boldsymbol{m}_{n_k} \circ \boldsymbol{y}_{n_k} \cdot \boldsymbol{\varphi}\,\d\boldsymbol{x}&=\int_{\im_{\rm G}(\boldsymbol{y}_{n_k},B'')} \boldsymbol{m}_{n_k} \cdot \boldsymbol{\varphi}\circ \boldsymbol{y}_{n_k,B}^{-1}\,\det D\boldsymbol{y}_{n_k,B}^{-1}\,\d\boldsymbol{\xi}\\
	&=\int_{\im_{\rm T}(\boldsymbol{y},B')} \boldsymbol{m}_{n_k} \cdot \boldsymbol{\varphi}\circ \boldsymbol{y}_{n_k,B}^{-1}\,\det D\boldsymbol{y}_{n_k,B}^{-1}\,\d\boldsymbol{\xi}.
	\end{split}
	\end{equation*}
	Recalling \eqref{eqn:compactness-m-local}, \eqref{eqn:inverses-convergence-ae} and \eqref{eqn:inverses-determinant}, we pass to the limit, as $k \to \infty$, in the previous equation by employing \cite[Proposition 2.61]{fonseca.leoni}. We obtain
	\begin{equation*}
	\begin{split}
	\lim_{k \to \infty} \int_{B''} \boldsymbol{m}_{n_k} \circ \boldsymbol{y}_{n_k} \cdot \boldsymbol{\varphi}\,\d\boldsymbol{x}&=\int_{\im_{\rm T}(\boldsymbol{y},B')} \boldsymbol{m} \cdot \boldsymbol{\varphi}\circ \boldsymbol{y}_B^{-1}\,\det D\boldsymbol{y}_B^{-1}\,\d\boldsymbol{\xi}\\
	&=\int_{\im_{\rm G}(\boldsymbol{y},B'')} \boldsymbol{m} \cdot \boldsymbol{\varphi}\circ \boldsymbol{y}_B^{-1}\,\det D\boldsymbol{y}_B^{-1}\,\d\boldsymbol{\xi}\\
	&=\int_{B''} \boldsymbol{m}\circ \boldsymbol{y}\,\cdot \boldsymbol{\varphi}\,\d\boldsymbol{x},
	\end{split}
	\end{equation*}
	where, in the last line, we used Proposition \ref{prop:inverse-change-of-variable}. This yields
	\begin{equation*}
	\text{$\boldsymbol{m}_{n_k} \circ \boldsymbol{y}_{n_k} \wk \boldsymbol{m}\circ \boldsymbol{y}$ in $L^2(\Omega;\R^N)$.}
	\end{equation*}
	In particular, as 
	\begin{equation*}
	\|\boldsymbol{m}_{n_k} \circ \boldsymbol{y}_{n_k}\|_{L^2(B'';\R^N)}=\|\boldsymbol{m} \circ \boldsymbol{y}\|_{L^2(B'';\R^N)}=\sqrt{\leb(B'')}
	\end{equation*}
	for every $k \in \N$, this actually entails the strong convergence in $L^2(B'';\R^N)$.
	
	By means of a diagonal argument analogous to the one employed in Step 2 of the proof of Theorem \ref{thm:compactness}, we select a subsequence for which the compositions converge strongly in $L^2(\Omega;\R^N)$ and, more generally, in $L^q(\Omega;\R^N)$ for every $1 \leq q<\infty$.
\end{remark}

\subsection{Existence of minimizers}
To prove the existence of minimizers of the magnetoelastic energy functional, we apply the Direct Method of the Calculus of Variations.
For future reference, we highlight the results concerning the coercivity and the lower semicontinuity of the magnetoelastic energy functional separately. We state the coercivity result in a slightly more self-contained form.

\begin{proposition}[Coercivity]
\label{prop:coercivity}
Let $(\boldsymbol{q}_n)\subset \mathcal{Q}$ with $\boldsymbol{q}_n=(\boldsymbol{y}_n,\boldsymbol{m}_n)$ be such that
\begin{equation}
\label{eqn:bound}
    \sup_{n \in \N} \left \{ \| \boldsymbol{y}_n\|_{W^{1,p}(\Omega;\R^N)}+\|D\boldsymbol{m}_n\|_{L^2(\im_{\rm T}(\boldsymbol{y}_n,\Omega);\rnn)}+\|\gamma(\det D \boldsymbol{y}_n)\|_{L^1(\Omega)} \right \}<+\infty,
\end{equation}
where $\gamma\colon (0,+\infty)\to (0,+\infty)$ is Borel and satisfies \eqref{eqn:gamma}.
Then, there exists $\boldsymbol{q}\in \mathcal{Q}$ with $\boldsymbol{q}=(\boldsymbol{y},\boldsymbol{m})$ such that, up to subsequences, the convergences in \eqref{eqn:compactness-deformation}--\eqref{eqn:improved-extension} hold true. In particular, if $(\boldsymbol{q}_n)\subset \mathcal{Q}^{\,\rm inj}$, then $\boldsymbol{q}\in \mathcal{Q}^{\,\rm inj}$ and the convergence in \eqref{eqn:improved-injective} holds true as well.  
\end{proposition}
\begin{proof}
By \eqref{eqn:bound}, there exists $\boldsymbol{y}\in W^{1,p}(\Omega;\R^N)$ such that, up to subsequences, there holds 
\begin{equation*}
    \text{$\boldsymbol{y}_n \wk \boldsymbol{y}$ in $W^{1,p}(\Omega;\R^N)$.}
\end{equation*}
 We need to prove that $\det D\boldsymbol{y}>0$ almost everywhere.
From \eqref{eqn:gamma} and \eqref{eqn:bound},  by De la Vallée-Poussin Criterion \cite[Theorem 2.29]{fonseca.leoni} and the Dunford-Pettis Theorem \cite[Theorem 2.54]{fonseca.leoni}, there exists $h \in L^1(\Omega)$ such that
\begin{equation*}
    \text{$\det D \boldsymbol{y}_n \wk h$ in $L^1(\Omega)$.}
\end{equation*}
By Theorem \ref{thm:closedness}, we infer that $\boldsymbol{y}$ satisfies \eqref{eqn:DIV} and  $h=\det D \boldsymbol{y}$ almost everywhere. As $\det D \boldsymbol{y}_n>0$ almost everywhere for every $n \in \N$, there holds $\det D\boldsymbol{y}\geq 0$ almost everywhere. 
Suppose by contraddiction that $\det D \boldsymbol{y}=0$ in $A$ for some measurable set $A \subset \Omega$ with $\leb(A)>0$. In this case, $\det D \boldsymbol{y}_n \to 0$ in $L^1(A)$ and, up to subsequences, $\det D \boldsymbol{y}_n \to 0$ almost everywhere in $A$. Thus, by \eqref{eqn:extreme-compression}, $\gamma(\det D \boldsymbol{y}_n)\to +\infty$ almost everywhere in $A$ so that, by the Fatou Lemma, we have
\begin{equation*}
    \liminf_{n\to \infty} \int_\Omega \gamma(\det D \boldsymbol{y}_n)\,\d\boldsymbol{x} \geq \liminf_{n\to \infty} \int_A \gamma(\det D \boldsymbol{y}_n)\,\d\boldsymbol{x}=+\infty.
\end{equation*}
This contraddicts \eqref{eqn:bound} and we conclude that $\det D \boldsymbol{y}>0$ almost evereywhere. Hence,  $\boldsymbol{y}\in \mathcal{Y}_p(\Omega)$.

We claim that $(1/\det D \boldsymbol{y}_n)$ is equi-integrable. To see this, define $\widetilde{\gamma}\colon (0,+\infty)\to (0,+\infty)$ by setting $\widetilde{\gamma}(z)=\gamma(1/z)$. By \eqref{eqn:gamma}, there holds 
\begin{equation*}
    \lim_{z \to +\infty} \widetilde{\gamma}(z)/z=+\infty.    
\end{equation*}
Also, by \eqref{eqn:bound}, we have
\begin{equation*}
    \sup_{n \in \N}\int_\Omega \widetilde{\gamma}(1/\det D \boldsymbol{y}_n)\,\d\boldsymbol{x}=\sup_{n \in \N}\int_\Omega \gamma(\det D \boldsymbol{y}_n)\,\d\boldsymbol{x}<+\infty,
\end{equation*}
so that $(1/\det D \boldsymbol{y}_n)$ is equi-integrable by De la Vallée-Poussin Criterion \cite[Theorem 2.29]{fonseca.leoni}. The proof is concluded by applying Theorem \ref{thm:compactness}.
\end{proof}

\begin{proposition}[Lower semicontinuity of the magnetoelastic energy]
\label{prop:lsc}
Let $(\boldsymbol{q}_n)\subset \mathcal{Q}$ with $\boldsymbol{q}_n=(\boldsymbol{y}_n,\boldsymbol{m}_n)$. Suppose that there exists $\boldsymbol{q}\in \mathcal{Q}$ with $\boldsymbol{q}=(\boldsymbol{y},\boldsymbol{m})$ such that the convergences in \eqref{eqn:compactness-deformation} and \eqref{eqn:compactness-extension-m}--\eqref{eqn:compactness-extension-Dm} hold true as well as the almost everywhere convergence in \eqref{eqn:compactness-composition}.
Then, there holds:
\begin{equation}
\label{eqn:lsc-magnetoelastic-energy}
    E(\boldsymbol{q})\leq \liminf_{n \to \infty} E(\boldsymbol{q}_n).
\end{equation}
\end{proposition}
\begin{proof}
By \eqref{eqn:compactness-deformation} and by the weak continuity of Jacobian minors, we have 
that all Jacobian minors of $\boldsymbol{y}_n$ converge weakly in $L^1(\Omega)$ to the corresponding Jacobian minor of $\boldsymbol{y}$. 
At this point, we can argue as in \cite[Theorem 2.4]{kruzik.stefanelli.zeman}.
From the weak convergence of the Jacobian minors of deformations in $L^1(\Omega)$ and from the almost everywhere convergence in \eqref{eqn:compactness-composition}, exploiting the polyconvexity assumption \eqref{eqn:W-polyconvex} and applying 
\cite[Theorem 5.4]{ball.currie.olver}, we obtain
\begin{equation}
\label{eqn:lsc-el}
    E^{\rm el}(\boldsymbol{q})\leq \liminf_{n \to \infty} E^{\rm el}(\boldsymbol{q}_n).
\end{equation}
From \eqref{eqn:compactness-extension-Dm}, by the lower semicontinuity of the norm, we immediately get
\begin{equation}
\label{eqn:lsc-exc}
    E^{\rm exc}(\boldsymbol{q})\leq \liminf_{n \to \infty} E^{\rm exc}(\boldsymbol{q}_n).
\end{equation}
For simplicity, for every $n \in \N$, set $u_n\coloneqq u_{\boldsymbol{m}_n}$. Thus, for every $n \in \N$, there holds
\begin{equation}
\label{eqn:lsc-maxwell-sequence}
    \forall\,v \in V^{1,2}(\R^N),\qquad \int_{\R^N} \left (-D u_n-\chi_{\im_{\rm T}(\boldsymbol{y}_n,\Omega)}\boldsymbol{m}_n \right ) \cdot D v \,\d\boldsymbol{\xi}=0.
\end{equation}
Testing \eqref{eqn:lsc-maxwell-sequence} with $v=u_n$ and applying the H\"{o}lder inequality, we see that
\begin{equation*}
    \sup_{n \in \N}\|Du_n\|_{L^2(\R^N;\R^N)}\leq \sup_{n \in \N}\|\chi_{\im_{\rm T}(\boldsymbol{y}_n,\Omega)}\boldsymbol{m}_n\|_{L^2(\R^N;\R^N)}<+\infty,
\end{equation*}
where we used \eqref{eqn:compactness-extension-m}. Thus, up to subsequences, there exists $u \in V^{1,2}(\R^N)$ such that
\begin{equation}
\label{eqn:weak-convergence-stray-fields}
    \text{$Du_n \wk Du$ in $L^2(\R^N;\R^N)$.}    
\end{equation}
Passing to the limit, as $n \to \infty$, in \eqref{eqn:lsc-maxwell-sequence} taking into account \eqref{eqn:compactness-extension-m}, we deduce
\begin{equation}
\label{eqn:lsc-maxwell}
    \forall\,v \in V^{1,2}(\R^N),\qquad \int_{\R^N} \left (-D u-\chi_{\im_{\rm T}(\boldsymbol{y},\Omega)}\boldsymbol{m} \right ) \cdot D v \,\d\boldsymbol{\xi}=0,
\end{equation}
and, in turn, we can assume $u=u_{\boldsymbol{m}}$. Then, by the lower semicontinuity of the norm, we obtain
\begin{equation}
\label{eqn:lsc-mag}
    E^{\rm mag}(\boldsymbol{q})\leq \liminf_{n \to \infty} E^{\rm mag}(\boldsymbol{q}_n).
\end{equation}
Finally, combining \eqref{eqn:lsc-el}--\eqref{eqn:lsc-exc} and \eqref{eqn:lsc-mag}, we establish \eqref{eqn:lsc-magnetoelastic-energy}. 
\end{proof}

The magnetostatic energy functional is actually continuous with respect to the notion of  convergnece provided by Theorem \ref{thm:compactness} in the case of globally invertible deformations. This fact is the content of the following remark. 

\begin{remark}[Continuity of the magnetostatic energy]
	With the same notation adopted in Proposition \ref{prop:lsc}, suppose that
	\eqref{eqn:improved-extension} and \eqref{eqn:weak-convergence-stray-fields} hold true.
	Then, there holds
	\begin{equation*}
	\text{$D \psi_{\boldsymbol{m}_n}\to D \psi_{\boldsymbol{m}}$ in $L^2(\R^N;\R^N)$}
	\end{equation*}
	and, in turn,
	\begin{equation*}
	\lim_{n \to \infty} E^{\rm mag}(\boldsymbol{q}_n)=E^{\rm mag}(\boldsymbol{q}).
	\end{equation*}
	To check this, we test the weak form of the Maxwell equation with the stray field potentials to see that, for every $n \in \N$, there holds
	\begin{equation*}
	\int_{\R^N} |D\psi_n|^2\,\d\boldsymbol{\xi}=\int_{\R^N} \chi_{\im_{\rm T}(\boldsymbol{y}_n,\Omega)}\boldsymbol{m}_n \cdot D \psi_n\,\d\boldsymbol{\xi}
	\end{equation*}
	and, analogously
	\begin{equation*}
	\int_{\R^N} |D\psi|^2\,\d\boldsymbol{\xi}=\int_{\R^N} \chi_{\im_{\rm T}(\boldsymbol{y},\Omega)}\boldsymbol{m} \cdot D \psi\,\d\boldsymbol{\xi}.
	\end{equation*}
	Hence, from \eqref{eqn:improved-extension} and \eqref{eqn:weak-convergence-stray-fields}, we obtain
	\begin{equation*}
	\lim_{n \to \infty} \int_{\R^N} |D \psi_n|^2\,\d\boldsymbol{\xi}=\int_{\R^N} |D \psi|^2\,\d\boldsymbol{\xi},
	\end{equation*}
	which, together with \eqref{eqn:weak-convergence-stray-fields}, entails the claim.
\end{remark}

We are finally able to establish the existence of minimizers for the magnetoelastic energy under Dirichlet boundary conditions on deformations. This constitutes the second main result of the section.

\begin{theorem}[Existence of mimimizers]
\label{thm:existence-minimizers}
Assume \eqref{eqn:gamma}--\eqref{eqn:W-polyconvex}. Let $\Gamma \subset \partial \Omega$ be measurable with $\haus(\Gamma)>0$ and let $\boldsymbol{d}\in L^p(\Gamma;\R^N)$. Suppose that 
\begin{equation*}
    \mathcal{A}\coloneqq\left \{{\boldsymbol{q}}=({\boldsymbol{y}},{\boldsymbol{m}})\in \mathcal{Q}:\:\text{$\tr_\Gamma({\boldsymbol{y}})=\boldsymbol{d}$ a.e. on $\Gamma$}\right \}    
\end{equation*}
is nonempty. Then, the functional $E$ admits a minimizer in $\mathcal{A}$. Analogously, if $\mathcal{A}^{\rm inj}\coloneqq \mathcal{A} \cap \mathcal{Q}^{\,\rm inj}$ is nonempty, then the functional $E$ admits a minimizer in $\mathcal{A}^{\rm inj}$.
\end{theorem}

\begin{proof}
The proof is a standard application of the Direct Method. Let $(\boldsymbol{q}_n)\subset \mathcal{A}$ be a minimizing sequence. In particular, for every $n \in \N$, there holds $\tr_\Gamma(\boldsymbol{y}_n)=\boldsymbol{d}$ almost everywhere on $\Gamma$.
We have 
\begin{equation*}
    \sup_{n \in \N} E(\boldsymbol{q}_n)<+\infty
\end{equation*}
and, in view of \eqref{eqn:coercivity}, this yields
\begin{equation}
\label{eqn:b1}
    \sup_{n \in \N} \left \{ \| D\boldsymbol{y}_n\|_{L^p(\Omega;\rnn)}+\|D\boldsymbol{m}_n\|_{L^2(\im_{\rm T}(\boldsymbol{y}_n,\Omega);\rnn)}+\|\gamma(\det D \boldsymbol{y}_n)\|_{L^1(\Omega)} \right \}<+\infty.
\end{equation}
By the Poincarè inequality with boundary term, for every $n \in \N$, we have
\begin{equation}
\label{eqn:b2}
    \|\boldsymbol{y}_n\|_{L^p(\Omega;\R^N)}\leq C \left (\|D\boldsymbol{y}_n\|_{L^p(\Omega;\rnn)}+\|\boldsymbol{d}\|_{L^p(\Gamma;\R^N)} \right),
\end{equation}
where the constant $C>0$ depends only on $\Omega$, $N$ and $p$. 
As \eqref{eqn:b1} together with \eqref{eqn:b2} yields \eqref{eqn:bound}, by Proposition \eqref{prop:coercivity} we deduce \eqref{eqn:compactness-deformation}--\eqref{eqn:improved-determinant}. In particular, \eqref{eqn:compactness-deformation} and the weak continuity of the trace operator entail $\tr_\Gamma(\boldsymbol{y})=\boldsymbol{d}$ almost everywhere on $\Gamma$, so that $\boldsymbol{q}\in \mathcal{A}$.
Finally, applying Proposition \ref{prop:lsc}, we obtain
\begin{equation*}
    E(\boldsymbol{q})\leq \liminf_{n \to \infty} E(\boldsymbol{q}_n)=\inf_{\widehat{\boldsymbol{q}}\in \mathcal{A}} E(\widehat{\boldsymbol{q}})
\end{equation*}
and, in turn, $\boldsymbol{q}$ is a minimizer of $E$ in $\mathcal{A}$. The proof under the constraint of global injectivity of deformations is totally analogous.
\end{proof}

\begin{remark}[Applied loads]
Let $\boldsymbol{f}\in L^{p'}(\Omega;\R^N)$, $\boldsymbol{g}\in L^{p'}(\Sigma;\R^N)$, where $\Sigma \subset \partial \Omega$ is measurable, and $\boldsymbol{h}\in L^2(\R^N;\R^N)$, represent applied body forces, surface forces and magnetic fields, respectively. The work of applied loads is determined by the functional $L \colon \mathcal{Q}\to \R$ defined by 
\begin{equation}
\label{eqn:applied-loads}
    L(\boldsymbol{q})\coloneqq \int_{\Omega} \boldsymbol{f}\cdot \boldsymbol{y}\,\d\boldsymbol{x}+\int_\Sigma \boldsymbol{g}\cdot \tr_\Sigma(\boldsymbol{y})\,\d\boldsymbol{a}+\int_{\im_{\rm T}(\boldsymbol{y},\Omega)} \boldsymbol{h}\cdot \boldsymbol{m}\,\d\boldsymbol{\xi},
\end{equation}
where $\boldsymbol{q}=(\boldsymbol{y},\boldsymbol{m})$. In particular, the energy contribution corresponding to external magnetic fields is of Eulerian type. The functional in \eqref{eqn:applied-loads}  is continuous with respect to the convergences in \eqref{eqn:compactness-deformation} and \eqref{eqn:compactness-extension-m}. Therefore, the proof of Theorem \ref{thm:existence-minimizers} can be easily adapted to prove the existence of minimizers of the total energy $E-L$ in $\mathcal{A}$ and $\mathcal{A}^{\,\rm inj}$, whenever these two classes are nonempty.
\end{remark}

The existence result provided by Theorem \ref{thm:existence-minimizers} can be adapted to cover the incompressible regime as well as inhomogeneous elastic energy densities.

Recalling Remark \ref{rem:sphere}, we realize that our  arguments can provide an alternative proof of \cite[Theorem 8.9]{barchiesi.henao.moracorral}, where the magnetic saturation constraint is formulated in the actual space. In this case, for the function $\gamma$ in \eqref{eqn:coercivity}, it is sufficient to satisfy \eqref{eqn:extreme-compression} in place of the first limit condition in \eqref{eqn:gamma}.

\section{Quasistatic setting}
\label{sec:quasistatic}
In this section, we study quasistatic evolutions driven by time-dependent applied loads and boundary conditions in presence of a rate-independent dissipation. 
The main result of the section is Theorem \ref{thm:existence-energetic-solution} which states the existence of energetic solutions for the quasistatic model. For simplicity, in this section, we do not specify our results for the case of globally invertible deformations. However, this requirement can be easily incorporated in the analysis thanks to Corollay \ref{cor:closure-injectivity}.

\subsection{The quasistatic model}
Let $T>0$ be the time horizon. Let $\Gamma, \Sigma \subset \partial \Omega$ be measurable with $\haus(\Gamma)>0$.
The map 
\begin{equation}
\label{eqn:regularity-boundary-datum}
    \boldsymbol{d}\in AC([0,T];L^p(\Gamma;\R^N))
\end{equation}
 constitutes the boundary datum. Applied loads are determined by the functions
\begin{equation}
\label{eqn:regularity-applied-loads}
    \begin{split}
        &\boldsymbol{f}\in AC([0,T];L^{p'}(\Omega;\R^N)), \quad \boldsymbol{g}\in AC([0,T];L^{p'}(\Sigma;\R^N)),\\
        &\hspace{25mm}\boldsymbol{h}\in AC([0,T];L^2(\R^N;\R^N)),
    \end{split}
\end{equation}
representing mechanical body and surface forces, and external magnetic fields, respectively. 

Recall \eqref{eqn:admissible-states} and \eqref{eqn:magnetoelastic-energy}. 
The total energy functional $\mathcal{E}\colon [0,T] \times  \mathcal{Q} \to \R$ is defined as
\begin{equation*}
    \mathcal{F}(t,\boldsymbol{q})\coloneqq E(\boldsymbol{q})-\mathcal{L}(t,\boldsymbol{q})+\mathcal{B}(t,\boldsymbol{q}).
\end{equation*}
The functionals $\mathcal{L},\mathcal{B}\colon [0,T]\times \mathcal{Q}\to \R$ are given by
\begin{equation}
\label{eqn:functional-L}
\mathcal{L}(t,\boldsymbol{q})\coloneqq \int_\Omega \boldsymbol{f}(t)\cdot \boldsymbol{y}\,\d\boldsymbol{x}+\int_{\Sigma} \boldsymbol{g}(t)\cdot \tr_\Sigma(\boldsymbol{y})\,\d\boldsymbol{a}+\int_{\im_{\rm T}(\boldsymbol{y},\Omega)} \boldsymbol{h}(t)\cdot \boldsymbol{m}\,\d\boldsymbol{\xi}
\end{equation}
and
\begin{equation}
\label{eqn:functional-B}
    \mathcal{B}(t,\boldsymbol{q})\coloneqq \int_\Gamma |\boldsymbol{d}(t)-\tr_\Gamma(\boldsymbol{y})|^p\,\d\boldsymbol{a},
\end{equation}
where $\boldsymbol{q}=(\boldsymbol{y},\boldsymbol{m})$. The functional in \eqref{eqn:functional-L} accounts for the work of applied loads, while
the functional in \eqref{eqn:functional-B} enforces the Dirichlet boundary condition $\tr_\Gamma(\boldsymbol{y})=\boldsymbol{d}(t)$ on $\Gamma$ in a relaxed sense.
Observe that, given \eqref{eqn:regularity-boundary-datum}--\eqref{eqn:regularity-applied-loads}, for every fixed $\boldsymbol{q}\in \mathcal{Q}$ the map $t \mapsto \mathcal{E}(t,\boldsymbol{q})$ belongs to $AC([0,T])$.

The dissipation distance $\mathcal{D}\colon \mathcal{Q}\times \mathcal{Q}\to [0,+\infty)$ is defined by
\begin{equation*}
    \mathcal{D}(\boldsymbol{q},\widehat{\boldsymbol{q}})\coloneqq \int_\Omega |\boldsymbol{m}\circ \boldsymbol{y}\det D\boldsymbol{y}-\widehat{\boldsymbol{m}}\circ \widehat{\boldsymbol{y}}\det D \widehat{\boldsymbol{y}}|\,\d\boldsymbol{x},
\end{equation*}
where $\boldsymbol{q}=(\boldsymbol{y},\boldsymbol{m})$ and $\widehat{\boldsymbol{q}}=(\widehat{\boldsymbol{y}},\widehat{\boldsymbol{m}})$. Given a map $\boldsymbol{q}\colon [0,T]\to \mathcal{Q}$, its variation with respect to $\mathcal{D}$ on the interval $[a,b]\subset [0,T]$ is defined as
\begin{equation}
\label{eqn:varD}
    \var_{\mathcal{D}}(\boldsymbol{q};[a,b])\coloneqq \sup \left \{\sum_{i=1}^m \mathcal{D}(\boldsymbol{q}(t^{i-1}),\boldsymbol{q}(t^i)):\:\text{$\Pi=(t^0,\dots,t^l)$ partition of $[a,b]$} \right \}.
\end{equation}
Here, a partition of $[a,b]$ is any finite ordered subset $(t^0,\dots,t^l)$ of $[0,T]$ with $a=t^0<\dots<t^l=b$ for some $l \in \N$. The quantity in \eqref{eqn:varD} represents the amount of energy dissipated along the evolution described by $\boldsymbol{q}$ within the time interval $[a,b]$.

The aim of this section is to prove the existence of energetic solutions for the quasistatic model. We recall the definition of such solutions \cite{mielke.roubicek}.

\begin{definition}[Energetic solution]
\label{def:energetic-solution}
A map $\boldsymbol{q}\colon [0,T]\to \mathcal{Q}$ is termed an energetic solution if the function $t \mapsto \partial_t \mathcal{F}(t,\boldsymbol{q}(t))$ belongs to $L^1(0,T)$ and the following two conditions hold:
\begin{enumerate}[(i)]
    \item \textbf{Global stability:} 
    \begin{equation}
    \label{eqn:global-stability}
        \forall\,t \in [0,T],\:\forall\, \widehat{\boldsymbol{q}}\in \mathcal{Q},\quad \mathcal{F}(t,\boldsymbol{q}(t))\leq \mathcal{F}(t,\widehat{\boldsymbol{q}})+\mathcal{D}(\boldsymbol{q}(t),\widehat{\boldsymbol{q}}).
    \end{equation}
    \item \textbf{Energy-dissipation balance:}
    \begin{equation}
    \label{eqn:energy-dissipation-balance}
        \forall\,t\in [0,T], \quad \mathcal{F}(t,\boldsymbol{q}(t))+\var_{\mathcal{D}}(\boldsymbol{q};[0,T])=\mathcal{F}(0,\boldsymbol{q}(0))+\int_0^t \partial_t \mathcal{F}(\tau,\boldsymbol{q}(\tau))\,\d\tau.
    \end{equation}
\end{enumerate}
\end{definition}

The strategy of the proof of the existence of energetic solutions is standard and consists of two main steps: first, for a fixed partition of the time interval, one solves the incremental minimization problem; then, one considers a sequence of partition of vanishing size and constructs the desired solution from the sequence piecewise constant interpolants corresponding to the solutions of the incremental minimization problems by compactness arguments. These two steps will be addressed in the next two subsections.

\subsection{Incremental minimization problem}
Before proceeding, we show that our model fulfils the basic assumptions of the theory of rate-independent systems \cite{mielke.roubicek}. Let $Z\subset (0,T)$ be the set of times for which the applied loads do not admit a pointwise derivative. Given \eqref{eqn:regularity-applied-loads}, there holds $\mathscr{L}^1(Z)=0$.

\begin{lemma}[Time-control of the total energy]
Assume \eqref{eqn:regularity-boundary-datum}--\eqref{eqn:regularity-applied-loads} and let $\mathcal{F}$ be defined as in \eqref{eqn:functional-E}--\eqref{eqn:functional-L}. The following two claims hold true.
\begin{enumerate}[(i)]
    \item There exist two constants $C_1,C_2>0$, depending only on the given data, such that, for every $t \in [0,T]$ and $\boldsymbol{q}\in \mathcal{Q}$ with $\boldsymbol{q}=(\boldsymbol{y},\boldsymbol{m})$,  there holds
        \begin{equation}
        \label{eqn:coercivity-total-energy}
        \begin{split}
        \mathcal{F}(t,\boldsymbol{q})&\geq C_1 \left \{\|D\boldsymbol{y}\|_{L^p(\Omega;\rnn)}^p+ \|\gamma(\det D \boldsymbol{y})\|_{L^1(\Omega)}+\|D\boldsymbol{m}\|_{L^2(\im_{\rm T}(\boldsymbol{y},\Omega);\rnn)}^2 \right \}\\
        &+C_1\,\|\boldsymbol{d}(t)-\boldsymbol{y}\|^p_{L^{p^\sharp}(\Gamma;\R^N)}-C_2.
        \end{split}
        \end{equation}
    \item There exist a constants $L>0$ and a function $\eta\in L^1(0,T)$, all depending only on the given data, such that, for every $\boldsymbol{q}\in\mathcal{Q}$, the following estimates hold:
    \begin{equation}
        \label{eqn:control-time-derivative}
         \forall, t\in (0,T)\setminus Z, \quad |\partial_t \mathcal{F}(t,\boldsymbol{q})|\leq \eta(t) \left (  \mathcal{F}(t,\boldsymbol{q})+L \right),
    \end{equation}
    \begin{equation}
        \label{eqn:gronwall}
         \forall,s,t\in (0,T), \quad \mathcal{F}(t,\boldsymbol{q})+L \leq \left(\mathcal{F}(s,\boldsymbol{q})+L \right)\mathrm{e}^{|H(t)-H(s)|},
    \end{equation}
    \begin{equation}
        \label{eqn:gronwall-derivative}
         \forall\, t \in (0,T)\setminus Z, \:\forall\,s\in [0,T], \quad |\partial_t\mathcal{F}(t,\boldsymbol{q})| \leq \eta(t)\left(\mathcal{F}(s,\boldsymbol{q})+C_0 \right)\mathrm{e}^{|H(t)-H(s)|}.
     \end{equation}
    Here, we define the function $H\in AC([0,T])$  by setting
    \begin{equation*}
    H(t)\coloneqq \int_0^t \eta(\tau)\,\d\tau.
    \end{equation*}  
    In particular, the constant $L>0$ is chosen in order to have
    \begin{equation}
    \label{eqn:L}
    \inf_{(t,{\boldsymbol{q}})\in [0,T]\times \mathcal{Q}} \mathcal{F}({t},\boldsymbol{q})\geq -\,L.
    \end{equation}  
\end{enumerate}
\end{lemma}
\begin{proof}
First,  by the Poincaré inequality with trace term, for every $\boldsymbol{y}\in W^{1,p}(\Omega;\R^N)$ and $t \in [0,T]$, there holds
\begin{equation}
\label{eqn:poincare}
    \|\boldsymbol{y}\|_{W^{1,p}(\Omega;\R^N)}\leq C_3 \left (\|D\boldsymbol{y}\|_{L^p(\Omega;\rnn)}+\|\boldsymbol{d}(t)-\tr_\Gamma(\boldsymbol{y})\|_{L^p(\Gamma;\R^N)}+\|\boldsymbol{d}\|_{C^0([0,T];L^p(\Gamma;\R^N))} \right),
\end{equation}
where the constant $C_3>0$ depends only on $\Omega$, $N$ and $p$. Second, define $\widetilde{\gamma}$ as in the proof of Proposition \ref{prop:coercivity} and recall that, by \eqref{eqn:gamma}, this function has superlinear growth at infinity. Applying Proposition \ref{prop:change-of-variable} and exploiting \eqref{eqn:magnetic-saturation}, for every $\boldsymbol{q}=(\boldsymbol{y},\boldsymbol{m})\in \mathcal{Q}$ we estimate
\begin{equation}
\label{eqn:m-gamma}
    \begin{split}
        \int_{\im_{\rm T}(\boldsymbol{y},\Omega)} |\boldsymbol{m}|^2\,\d\boldsymbol{\xi}&\leq \int_\Omega |\boldsymbol{m}\circ \boldsymbol{y}|^2\,\det D \boldsymbol{y}\,\d\boldsymbol{x}=\int_\Omega \frac{1}{\det D \boldsymbol{y}}\,\d\boldsymbol{x}\\
        &\leq C_4 \left (1+\left \|\widetilde{\gamma}\left (\frac{1}{\det D \boldsymbol{y}}\right)\right\|_{L^1(\Omega)} \right)\leq  C_4 \left (1+||\gamma(\det D \boldsymbol{y})||_{L^1(\Omega)} \right),
    \end{split}
\end{equation}
where the constant $C_4>0$ depends only on $\Omega$ and $\gamma$.

Recalling \eqref{eqn:coercivity} and employing \eqref{eqn:poincare}--\eqref{eqn:m-gamma}, claim (i) is checked by means of a standard application of H\"{o}lder and Young inequalities. 
For (ii), recall that, for every fixed $\boldsymbol{q}\in \mathcal{Q}$, 
the map $t\mapsto \mathcal{F}(t,\boldsymbol{q})$ belongs to $AC([0,T])$. In particular,  $t \in (0,T)\setminus Z$ and $\boldsymbol{q}=(\boldsymbol{y},\boldsymbol{m})\in \mathcal{Q}$, there holds 
\begin{equation}
\label{eqn:time-derivative-expression}
		\begin{split}
			\partial_t \mathcal{F}(t,\boldsymbol{q})&=-\partial_t \mathcal{L}(t,\boldsymbol{q})+\partial_t \mathcal{B}(t,\boldsymbol{q})\\
			&=
			-\int_\Omega \dot{\boldsymbol{f}}(t)\cdot \boldsymbol{y}\,\d\boldsymbol{x}-\int_{\Sigma} \dot{\boldsymbol{g}}(t)\cdot \tr_\Sigma(\boldsymbol{y})\,\d\boldsymbol{a}-\int_{\im_{\rm T}(\boldsymbol{y},\Omega)} \dot{\boldsymbol{h}}(t)\cdot \boldsymbol{m}\,\d\boldsymbol{\xi}\\
			&+p\int_{\Gamma} |\boldsymbol{d}(t)-\boldsymbol{y}|^{p-2} (\boldsymbol{d}(t)-\boldsymbol{y})\cdot \dot{\boldsymbol{d}}(t)\,\d\boldsymbol{a}.
		\end{split}
\end{equation}
Employing again H\"{o}lder and Young inequalities together with \eqref{eqn:coercivity-total-energy} and \eqref{eqn:poincare}--\eqref{eqn:m-gamma}, we find a function $\eta \in L^1(0,T)$ of the form
\begin{equation*}
	\eta(t)\coloneqq C \left(\|\dot{\boldsymbol{f}}(t)\|_{L^{p'}(\Omega;\R^N)}+ \|\dot{\boldsymbol{g}}(t)\|_{L^{p'}(\Sigma;\R^N)}+\|\dot{\boldsymbol{h}}(t)\|_{L^2(\R^N;\R^N)}+ \|\dot{\boldsymbol{d}}(t)\|_{L^p(\Sigma;\R^N)}       \right)
\end{equation*}
for some constant $C>0$ and another constant $0<L\leq C_2$, both  depending only on the given data, such that \eqref{eqn:control-time-derivative} holds true.
From this, by applying the Gronwall inequality, we deduce \eqref{eqn:gronwall}. Finally, combining \eqref{eqn:control-time-derivative} and \eqref{eqn:gronwall}, we obtain \eqref{eqn:gronwall-derivative}. In particular, \eqref{eqn:L}  follows from \eqref{eqn:coercivity-total-energy}.
\end{proof}

In the following, we will exploit the coercivity and lower semicontinuity properties below which are simple consequences of the results of Section \ref{sec:static}. 

\begin{lemma}[Coercivity and lower semicontinuity of the total energy]
	\label{lem:properties-total-energy}
	Assume \eqref{eqn:gamma}--\eqref{eqn:W-polyconvex} and \eqref{eqn:regularity-boundary-datum}--\eqref{eqn:regularity-applied-loads}. The following two claims hold true.
	\begin{enumerate}[(i)]
		\item \textbf{Coercivity:} Let $(t_n)\subset [0,T]$ and $(\boldsymbol{q}_n)\subset \mathcal{Q}$ with $\boldsymbol{q}_n=(\boldsymbol{y}_n,\boldsymbol{m}_n)$. Suppose that
		\begin{equation}
		\label{eqn:total-energy-bounded}
		\sup_{n \in \N} \mathcal{F}(t_n,\boldsymbol{q}_n)<+\infty.
		\end{equation}
		Then,  there exists $\boldsymbol{q}\in \mathcal{Q}$ with $\boldsymbol{q}=(\boldsymbol{y},\boldsymbol{m})$ such that, up to subsequences, the convergences in \eqref{eqn:compactness-deformation}--\eqref{eqn:improved-extension} hold true.
		\item \textbf{Lower semicontinuity:} Let $(t_n)\subset [0,T]$ and $t \in [0,T]$. Let $(\boldsymbol{q}_n)\subset \mathcal{Q}$ with $\boldsymbol{q}_n=(\boldsymbol{y}_n,\boldsymbol{m}_n)$ and $\boldsymbol{q}\in \mathcal{Q}$ with $\boldsymbol{q}=(\boldsymbol{y},\boldsymbol{m})$. Suppose that $t_n \to t$  and that the convergences in \eqref{eqn:compactness-deformation} and \eqref{eqn:compactness-extension-m}--\eqref{eqn:compactness-extension-Dm} hold true as well as the almost everywhere convergence in \eqref{eqn:compactness-composition}. Then, there holds
		\begin{equation}
		\label{eqn:total-energy-lsc}
		\mathcal{F}(t,\boldsymbol{q})\leq \liminf_{n \to \infty} \mathcal{F}(t_n,\boldsymbol{q}_n).
		\end{equation}
	\end{enumerate}
\end{lemma}
\begin{proof}
	From \eqref{eqn:total-energy-bounded}, by \eqref{eqn:coercivity-total-energy} and \eqref{eqn:poincare}, we deduce \eqref{eqn:bound}. Thus, claim (i) follows by Proposition \ref{prop:coercivity}. We prove claim (ii). First, observe that the functionals $\mathcal{L}(t,\cdot)$ and $\mathcal{B}(t,\cdot)$ are continuous with respect to the convergences in \eqref{eqn:compactness-deformation}-- \eqref{eqn:compactness-extension-m}. This, together with Proposition \ref{prop:lsc}, yields 
	\begin{equation*}
	\mathcal{F}(t,\boldsymbol{q})\leq \liminf_{n \to \infty} \mathcal{F}(t,\boldsymbol{q}_n).
	\end{equation*}
	Hence, in order to establish \eqref{eqn:total-energy-lsc}, it is sufficient to check that
	\begin{equation*}
	\lim_{n \to \infty} \left \{\mathcal{F}(t_n,\boldsymbol{q}_n)-\mathcal{F}(t,\boldsymbol{q}_n) \right \}= \lim_{n\to\infty} \left \{  \mathcal{L}(t,\boldsymbol{q}_n)-\mathcal{L}(t_n,\boldsymbol{q}_n)+\mathcal{B}(t_n,\boldsymbol{q}_n)-\mathcal{B}(t,\boldsymbol{q}_n)\right\} =0 .
	\end{equation*}
	This claim holds true since, from \eqref{eqn:compactness-deformation}--\eqref{eqn:compactness-extension-Dm} and \eqref{eqn:regularity-boundary-datum}--\eqref{eqn:regularity-applied-loads}, we have
	\begin{equation*}
	\lim_{n\to \infty} \mathcal{L}(t_n,\boldsymbol{q}_n)= \lim_{n\to \infty} \mathcal{L}(t,\boldsymbol{q}_n)= \mathcal{L}(t,\boldsymbol{q}), \qquad \lim_{n\to \infty} \mathcal{B}(t_n,\boldsymbol{q}_n)=\lim_{n\to \infty} \mathcal{B}(t,\boldsymbol{q}_n)=\mathcal{B}(t,\boldsymbol{q}).
	\end{equation*}
\end{proof}

We now introduce the incremental minimization problem.
Let $\Pi=(t^0,\dots,t^l)$ be a partition of $[0,T]$ and let $\boldsymbol{q}^0\in \mathcal{Q}$. The incremental minimization problem determined by $\Pi$ with initial datum $\boldsymbol{q}^0$ reads as follows: 
\begin{equation}
\label{eqn:IMP}
    \begin{split}
        &\hspace{2mm}\text{find $\boldsymbol{q}^1,\dots,\boldsymbol{q}^l\in \mathcal{Q}$ such that each $\boldsymbol{q}^i$ is a minimizer of}\\
        &\text{the functional $\boldsymbol{q}\mapsto \mathcal{F}(t^i,\boldsymbol{q})+\mathcal{D}(\boldsymbol{q}^{i-1},\boldsymbol{q})$ for $i \in \{1,\dots,l\}$.}
    \end{split}
\end{equation}

In the next result, we prove the existence of solutions of the incremental minimization problem and we collect their main properties.

\begin{proposition}[Solutions of the incremental minimization problem]
\label{prop:IMP}
Let $\Pi=(t^0,\dots,t^l)$ be a partition of $[0,T]$ and let $\boldsymbol{q}^0\in \mathcal{Q}$. Then, the incremental minimization problem determined by $\Pi$ with initial datum $\boldsymbol{q}^0$ admits a solution $(\boldsymbol{q}^1,\dots,\boldsymbol{q}^l)$. Moreover, if $\boldsymbol{q}^0$ satisfies
\begin{equation}
    \label{eqn:stability-initial-datum}
    \forall\,\widehat{\boldsymbol{q}}\in \mathcal{Q},\quad \mathcal{F}(0,\boldsymbol{q}^0)\leq \mathcal{F}(0,\widehat{\boldsymbol{q}})+\mathcal{D}(\boldsymbol{q}^0,\widehat{\boldsymbol{q}}),
\end{equation}
then the following hold:
\begin{equation}
    \label{eqn:IMP-stability}
    \forall\,i \in \{1,\dots,l\},\:\forall\,\widehat{\boldsymbol{q}}\in \mathcal{Q}, \quad \mathcal{F}(t^i,\boldsymbol{q}^i)\leq \mathcal{F}(t^i,\widehat{\boldsymbol{q}})+\mathcal{D}(\boldsymbol{q}^i,\widehat{\boldsymbol{q}}),
\end{equation}
\begin{equation}
\label{eqn:IMP-energy-dissipation-balance}
    \forall\,i \in \{1,\dots,l\}, \quad \mathcal{F}(t^i,\boldsymbol{q}^i)-\mathcal{F}(t^{i-1},\boldsymbol{q}^{i-1})+\mathcal{D}(\boldsymbol{q}^{i-1},\boldsymbol{q}^i)\leq \int_{t^{i-1}}^{t^i} \partial_t \mathcal{F}(\tau,\boldsymbol{q}^{i-1})\,\d\tau,
\end{equation}
\begin{equation}
\label{eqn:IMP-a-priori-estimates}
    \forall\,i \in \{1,\dots,l\}, \quad \mathcal{F}(t^i,\boldsymbol{q}^i)+L+\sum_{j=1}^i \mathcal{D}(\boldsymbol{q}^{j-1},\boldsymbol{q}^j)\leq \left ( \mathcal{F}(0,\boldsymbol{q}^0)+L\right)\mathrm{e}^{H(t^i)}.
\end{equation}
\end{proposition}
\begin{proof}
Once the solvability of the incremental minimization problem is shown, claims \eqref{eqn:IMP-stability}--\eqref{eqn:IMP-a-priori-estimates} follow by standard computations \cite[Proposition 2.1.4]{mielke.roubicek}. Let $i \in \{1,\dots,l\}$ and suppose that $\boldsymbol{q}^{i-1}\in \mathcal{Q}$ is given. Let $(\boldsymbol{q}_n)\subset \mathcal{Q}$ be a minimizing sequence for the functional $\boldsymbol{q}\mapsto \mathcal{F}(t^i,\boldsymbol{q})+\mathcal{D}(\boldsymbol{q}^{i-1},\boldsymbol{q})$. Clearly
\begin{equation*}
    \sup_{n \in \N} \mathcal{F}(t^i,\boldsymbol{q}_n)\leq \sup_{n \in \N} \left \{ \mathcal{F}(t^i,\boldsymbol{q}_n)+\mathcal{D}(\boldsymbol{q}^{i-1},\boldsymbol{q}_n) \right\}<+\infty.
\end{equation*}
Thus, by claim (i) of Lemma \ref{lem:properties-total-energy}, there exists $\boldsymbol{q}\in \mathcal{Q}$ with $\boldsymbol{q}=(\boldsymbol{y},\boldsymbol{m})$ such that, up to subsequences, the convergence in \eqref{eqn:compactness-deformation}--\eqref{eqn:improved-extension} hold true. Then, by claim (ii) of Lemma \ref{lem:properties-total-energy} and by the definition of the dissipation distance, we have
\begin{equation*}
\begin{split}
    \mathcal{F}(t^i,\boldsymbol{q})+\mathcal{D}(\boldsymbol{q}^{i-1},\boldsymbol{q})&\leq \liminf_{n \to \infty} \mathcal{F}(t^i,\boldsymbol{q}_n)+\lim_{n \to \infty}\mathcal{D}(\boldsymbol{q}^{i-1},\boldsymbol{q}_n)\\
    &=\liminf_{n \to \infty} \left \{\mathcal{F}(t^i,\boldsymbol{q}_n) +\mathcal{D}(\boldsymbol{q}^{i-1},\boldsymbol{q}_n)\right \}\\
    &=\inf_{\widehat{\boldsymbol{q}}\in\mathcal{Q}} \left \{\mathcal{F}(t^i,\widehat{\boldsymbol{q}}) +\mathcal{D}(\boldsymbol{q}^{i-1},\widehat{\boldsymbol{q}})\right \}.
\end{split}
\end{equation*}
This concludes the proof.
\end{proof}

In the next result, we introduce the piecewise-constant interpolants and we show that these satisfy time-discrete versions of the stability condition and of the upper energy-dissipation inequality.

\begin{proposition}[Piecewise-constant interpolants]
\label{prop:pc-interpolants}
Let $\Pi=(t^0,\dots,t^l)$ be a partition of $[0,T]$ and let $\boldsymbol{q}^0\in \mathcal{Q}$ satisfy \eqref{eqn:stability-initial-datum}. Suppose that $(\boldsymbol{q}^1,\dots,\boldsymbol{q}^l)$ is a solution of the incremental minimization problem determined by $\Pi$ with initial datum $\boldsymbol{q}^0$. Define the piecewise-constant interpolant $\boldsymbol{q}_\Pi\colon [0,T] \to \mathcal{Q}$ by setting
\begin{equation}
\label{eqn:pc-interpolant}
    \boldsymbol{q}_\Pi(t)\coloneqq 
    \begin{cases}
    \boldsymbol{q}^{i-1} & \text{if $t^{i-1}\leq t<t^i$ for some $i \in \{1,\dots,l\}$,}\\
    \boldsymbol{q}^l & \text{if $t=T$.}
    \end{cases}
\end{equation}
Then, the following hold:
\begin{equation}
\label{eqn:pc-interpolants-stability}
    \forall\,t \in \Pi,\:\forall\,\widehat{\boldsymbol{q}}\in \mathcal{Q},\quad \mathcal{F}(t,\boldsymbol{q}_\Pi(t))\leq \mathcal{F}(t,\widehat{\boldsymbol{q}})+\mathcal{D}(\boldsymbol{q}_\Pi(t),\widehat{\boldsymbol{q}}),
\end{equation}
\begin{equation}
\label{eqn:pc-interpolants-energy-dissipation-balance}
    \forall\,s,t\in \Pi:s<t,\quad \mathcal{F}(t,\boldsymbol{q}_\Pi(t))-\mathcal{F}(s,\boldsymbol{q}_\Pi(s))+\var_{\mathcal{D}}(\boldsymbol{q}_\Pi;[s,t])\leq \int_s^t \partial_t \mathcal{F}(\tau,\boldsymbol{q}_\Pi(\tau))\,\d\tau,
\end{equation}
\begin{equation}
\label{eqn:pc-interpolants-a-priori-estimate}
    \forall\,t \in [0,T],\quad \mathcal{F}(t,\boldsymbol{q}_\Pi(t))+L+\var_{\mathcal{D}}(\boldsymbol{q}_\Pi;[0,t])\leq \left (\mathcal{F}(0,\boldsymbol{q}^0)+L \right) \mathrm{e}^{H(t)}.
\end{equation}
\end{proposition}
\begin{proof}
Claims \eqref{eqn:pc-interpolants-stability} and \eqref{eqn:pc-interpolants-energy-dissipation-balance} follow from \eqref{eqn:IMP-stability}  and \eqref{eqn:IMP-energy-dissipation-balance}, respectively. To prove \eqref{eqn:pc-interpolants-a-priori-estimate}, let $t \in [0,T]$ and let $i \in \{1,\dots,l\}$ be such that $t^{i-1}\leq t<t^i$. Then, using  \eqref{eqn:gronwall},  \eqref{eqn:IMP-a-priori-estimates} and \eqref{eqn:pc-interpolant}, we obtain
\begin{equation*}
    \begin{split}
        \mathcal{F}(t,\boldsymbol{q}_\Pi(t))+L+\var_{\mathcal{D}}(\boldsymbol{q}_\Pi;[0,t])&=\mathcal{F}(t,\boldsymbol{q}^{i-1})+L+\sum_{j=1}^{i-1}\mathcal{D}(\boldsymbol{q}^{j-1},\boldsymbol{q}^j)\\
        &\leq \left(\mathcal{F}(t^{i-1},\boldsymbol{q}^{i-1})+L \right) \mathrm{e}^{H(t)-H(t^{i-1})}+\sum_{j=1}^{i-1}\mathcal{D}(\boldsymbol{q}^{j-1},\boldsymbol{q}^j)\\
        &\leq \left (\mathcal{F}(t^{i-1},\boldsymbol{q}^{i-1})+L+\sum_{j=1}^{i-1}\mathcal{D}(\boldsymbol{q}^{j-1},\boldsymbol{q}^j) \right) \mathrm{e}^{H(t)-H(t^{i-1})}\\
        &\leq \left(\mathcal{F}(0,\boldsymbol{q}^0)+L \right) \mathrm{e}^{H(t)}.
    \end{split}
\end{equation*}
\end{proof}

\subsection{Existence of energetic solutions} We are finally ready to present the main result of the section, namely the existence of energetic solutions for the quasistatic model.

\begin{theorem}[Existence of energetic solutions]
\label{thm:existence-energetic-solution}
Assume \eqref{eqn:gamma}--\eqref{eqn:W-polyconvex} and \eqref{eqn:regularity-boundary-datum}--\eqref{eqn:regularity-applied-loads}. Then, for every $\boldsymbol{q}^0 \in \mathcal{Q}$ satisfying \eqref{eqn:stability-initial-datum}, there exists a measurable map $\boldsymbol{q}\colon [0,T]\to \mathcal{Q}$ with $\boldsymbol{q}(t)=(\boldsymbol{y}(t),\boldsymbol{m}(t))$  which is an energetic solution according to Definition \ref{def:energetic-solution} and satisfies $\boldsymbol{q}(0)=\boldsymbol{q}^0$. Moreover, the map $t \mapsto \boldsymbol{m}(t)\circ \boldsymbol{y}(t)\det D \boldsymbol{y}(t)$ belongs to $BV([0,T];L^1(\Omega;\R^N))$.
\end{theorem}

The measurability of the map $\boldsymbol{q}$ in Theorem \ref{thm:existence-energetic-solution} is meant with respect to the Borel $\sigma$-algebra of $\mathcal{Q}$, namely the map
\begin{equation*}
	t \mapsto (\boldsymbol{y},\chi_{\im_{\rm T}(\boldsymbol{y},\Omega)}\boldsymbol{m},\chi_{\im_{\rm T}(\boldsymbol{y},\Omega)}D\boldsymbol{m})
\end{equation*}
is measurable from $[0,T]$ to $W^{1,p}(\Omega;\R^N)\times L^2(\R^N;\R^N)\times L^2(\R^N;\rnn)$, where the latter space is equipped with the product weak topology. 
 Eventually, by means of standard computations \cite{mielke.roubicek.stefanelli}, it can be shown that the map $t \mapsto \mathcal{F}(t,\boldsymbol{q}(t))$ belongs to $BV([0,T])$.

The remainder of the section is devoted to the proof of Theorem \ref{thm:existence-energetic-solution}. This employs the following version of the Helly Selection Principle, which is a special case of \cite[Theorem 3.2]{mainik.mielke}.

\begin{lemma}[Helly Selection Principle]
\label{lem:helly}
Let $Z$ be a Banach space and let $\mathcal{K}\subset Z$ be compact. Let $(\boldsymbol{z}_n)\subset BV([0,T];Z)$ be such that the following hold:
\begin{equation*}
    \forall\,n \in \N,\:\forall\,t \in [0,T], \quad \boldsymbol{z}_n(t)\in \mathcal{K}, 
\end{equation*}
\begin{equation}
    \sup_{n \in \N} \var_Z(\boldsymbol{z}_n;[0,T])<+\infty.
\end{equation}
Then, there exists $\boldsymbol{z}\in BV([0,T];Z)$ such that, up to subsequences, there holds:
\begin{equation}
    \text{$\forall\,t \in [0,T],\quad \boldsymbol{z}_n(t)\to\boldsymbol{z}(t)$ in $Z$.}
\end{equation}
\end{lemma}

We now proceed with the proof of Theorem \ref{thm:existence-energetic-solution}.

\begin{proof}[Proof of Theorem \ref{thm:existence-energetic-solution}]
We rigorously follow the scheme provided in \cite[Theorem 2.1.6]{mielke.roubicek}. The proof is subdivided into five steps.

\textbf{Step 1 (A priori estimates).} Let $(\Pi_n)$ with $\Pi_n=(t^0_n,\dots,t^{l_n}_n)$ be a sequence of partitions of $[0,T]$ such that $|\Pi_n|\coloneqq\max_{i \in \{1,\dots,l_n\}} (t^i_n-t^{i-1}_n)\to 0$, as $n \to \infty$. By Proposition \ref{prop:IMP}, for every $n \in \N$, the incremental minimization problem determined by $\Pi_n$ with initial datum $\boldsymbol{q}^0$ admits a solution. If $\boldsymbol{q}_n\coloneqq \boldsymbol{q}_{\Pi_n}$ denotes the corresponding piecewise-constant interpolant defined as in \eqref{eqn:pc-interpolant}, then, by Proposition \ref{prop:pc-interpolants}, there hold:
\begin{equation}
\label{eqn:stability-sequence}
    \forall\,t \in \Pi_n,\:\forall\,\widehat{\boldsymbol{q}}\in \mathcal{Q},\quad \mathcal{F}(t,\boldsymbol{q}_n(t))\leq \mathcal{F}(t,\widehat{\boldsymbol{q}})+\mathcal{D}(\boldsymbol{q}_n(t),\widehat{\boldsymbol{q}}),
\end{equation}
\begin{equation}
\label{eqn:energy-dissipation-balance-sequence}
    \forall\,s,t\in \Pi_n:s<t,\quad \mathcal{F}(t,\boldsymbol{q}_n(t))-\mathcal{F}(s,\boldsymbol{q}_n(s))+\var_{\mathcal{D}}(\boldsymbol{q}_n;[s,t])\leq \int_s^t \partial_t \mathcal{F}(\tau,\boldsymbol{q}_n(\tau))\,\d\tau,
\end{equation}
\begin{equation}
\label{eqn:a-priori-estimate-sequence}
    \forall\,t \in [0,T],\quad \mathcal{F}(t,\boldsymbol{q}_n(t))+L+\var_{\mathcal{D}}(\boldsymbol{q}_n;[0,t])\leq \left (\mathcal{F}(0,\boldsymbol{q}^0)+L \right) \mathrm{e}^{H(t)}.
\end{equation}
In particular, from \eqref{eqn:a-priori-estimate-sequence}, we deduce the existence of a constant $M>0$, depending only on the given data, such that
\begin{equation}
\label{eqn:a-priori-estimate}
    \sup_{n \in \N} \left \{ \sup_{t \in [0,T]} \mathcal{F}(t,\boldsymbol{q}_n(t))+L+\var_{\mathcal{D}}(\boldsymbol{q}_n;[0,T]) \right \}\leq M.
\end{equation}

\textbf{Step 2 (Compactness).} Set
\begin{equation*}
    \mathcal{H}\coloneqq \bigcup_{\widehat{t}\in [0,T]} \left \{\widehat{\boldsymbol{q}}=(\widehat{\boldsymbol{y}},\widehat{\boldsymbol{m}})\in \mathcal{Q}:\: \mathcal{F}(\widehat{t},\widehat{\boldsymbol{q}})\leq M\right \}.
\end{equation*}
and
\begin{equation*}
    \mathcal{K}\coloneqq \bigcup_{\widehat{t}\in [0,T]} \left \{\widehat{\boldsymbol{m}}\circ \widehat{\boldsymbol{y}}\det D \widehat{\boldsymbol{y}}:\:\widehat{\boldsymbol{q}}=(\widehat{\boldsymbol{y}},\widehat{\boldsymbol{m}})\in \mathcal{H}\right \}.
\end{equation*}
By \eqref{eqn:a-priori-estimate}, the sequence $(\boldsymbol{q}_n)$ takes values in $\mathcal{H}$. For every $n \in \N$, define $\boldsymbol{z}_n\colon [0,T]\to L^1(\Omega;\R^N)$ by setting
\begin{equation*}
	\boldsymbol{z}_n(t)\coloneqq \boldsymbol{m}_n(t)\circ \boldsymbol{y}_n(t)\det D\boldsymbol{y}_n(t).
\end{equation*}
By construction, the sequence $(\boldsymbol{z}_n)$ takes values in $\mathcal{K}$.  Also, by claim (i) of Lemma \ref{lem:properties-total-energy}, $\mathcal{K}$ is a compact subset of $L^1(\Omega;\R^N)$ while, by \eqref{eqn:a-priori-estimate}, there holds
\begin{equation*}
    \sup_{n \in \N}\var_{L^1(\Omega;\R^N)}(\boldsymbol{z}_n;[0,T])<+\infty.
\end{equation*}
Hence, thanks to Lemma \ref{lem:helly}, there exists  a map $\boldsymbol{z}\in BV([0,T];L^1(\Omega;\R^N))$ such that, up to subsequences, we have
\begin{equation}
\label{eqn:compactness-z}
    \text{$\forall\,t \in [0,T]$, \quad $\boldsymbol{z}_n(t)\to\boldsymbol{z}(t)$ in $L^1(\Omega;\R^N)$.}
\end{equation}
The construction of the candidate solution $\boldsymbol{q}\colon [0,T]\to \mathcal{Q}$ requires more work. First, by \eqref{eqn:coercivity-total-energy},  every $\widehat{\boldsymbol{q}}=(\widehat{\boldsymbol{y}},\widehat{\boldsymbol{m}})\in \mathcal{H}$ with $\mathcal{D}(\widehat{t},\widehat{\boldsymbol{q}})\leq M$ for some $\widehat{t}\in [0,T]$ satisfies
\begin{equation*}
    \|D\widehat{\boldsymbol{y}}\|_{L^p(\Omega;\rnn)}+\|\gamma(\det D \widehat{\boldsymbol{y}})\|_{L^1(\Omega)}+\|D\widehat{\boldsymbol{m}}\|_{L^2(\im_{\rm T}(\widehat{\boldsymbol{y}},\Omega);\rnn)}+\|\boldsymbol{d}(\widehat{t})-\tr_\Gamma(\widehat{\boldsymbol{y}})\|_{L^p(\Gamma;\R^N)}\leq C_5,
\end{equation*}
where the constant $C_5>0$ depends only on $M$ and on the given data. From this, setting 
\begin{equation*}
    \widehat{\boldsymbol{v}}\coloneqq \chi_{\im_{\rm T}(\widehat{\boldsymbol{y}},\Omega)}\widehat{\boldsymbol{m}}, \qquad \widehat{\boldsymbol{V}}\coloneqq \chi_{\im_{\rm T}(\widehat{\boldsymbol{y}},\Omega)}D\widehat{\boldsymbol{m}},
\end{equation*}
by \eqref{eqn:poincare}--\eqref{eqn:m-gamma}, we deduce
\begin{equation}
\label{eqn:X-space}
    \|\widehat{\boldsymbol{y}}\|_{W^{1,p}(\Omega;\R^N)}++\|\widehat{\boldsymbol{v}}\|_{L^2(\R^N;\R^N)}+\|\widehat{\boldsymbol{V}}\|_{L^2(\R^N;\rnn)}\leq C_6,
\end{equation}
where  the constant $C_6>0$  also depends only on $M$ and on the given data. In particular, this constant does not depend on $\widehat{t}$.
We define $\mathcal{X}$ as the set of triples
\begin{equation*}
	(\widehat{\boldsymbol{y}},\widehat{\boldsymbol{v}},\widehat{\boldsymbol{V}})\in W^{1,p}(\Omega;\R^N)\times L^2(\R^N;\R^N)\times L^2(\R^N;\rnn)	
\end{equation*}
satisfying
\eqref{eqn:X-space}. 
This is endowed with the product weak topology which makes it a compact metrizable space. Now, for every $n \in \N$, let $\boldsymbol{q}_n(t)=(\boldsymbol{y}_n(t),\boldsymbol{m}_n(t))$ for every $t \in [0,T]$ and define the maps $\boldsymbol{v}_n\colon [0,T]\to L^2(\R^N;\R^N)$ and $\boldsymbol{V}_n\colon [0,T]\to L^2(\R^N;\rnn)$ by setting
\begin{equation*}
    \boldsymbol{v}_n(t)\coloneqq \chi_{\im_{\rm T}(\boldsymbol{y}_n(t),\Omega)}\boldsymbol{m}_n(t),\qquad \boldsymbol{V}_n(t)\coloneqq \chi_{\im_{\rm T}(\boldsymbol{y}_n(t),\Omega)}D\boldsymbol{m}_n(t). 
\end{equation*}
By construction, the map $t \mapsto  (\boldsymbol{y}_n(t),\boldsymbol{v}_n(t),\boldsymbol{V}_n(t))$ takes values in $\mathcal{X}$ for every $n \in \N$. Consider the set-valued map $S\colon [0,T]\to \mathcal{P}(\mathcal{X})$ where $S(t)$ is defined as the set of all limit points of the sequence $((\boldsymbol{y}_n(t),\boldsymbol{v}_n(t),\boldsymbol{V}_n(t)))$ in $\mathcal{X}$. Equivalently, $S(t)$ is the limit superior of the set $\{(\boldsymbol{y}_n(t),\boldsymbol{v}_n(t),\boldsymbol{V}_n(t))\}$ in the sense of Kuratowski \cite[Definition 1.1.1]{aubin.frankowska}. Clearly, $S(t)$ is closed in $\mathcal{X}$ and nonempty for every $t \in [0,T]$.
Also, by \cite[Theorem 8.2.5]{aubin.frankowska}, the set-valued map $S$ is measurable \cite[Definition 8.1.1]{aubin.frankowska}. Therefore, by \cite[Theorem 8.1.3]{aubin.frankowska}, there exists a measurable selection of $S$, namely a measurable map $\boldsymbol{s}\colon [0,T]\to \mathcal{X}$ such that $\boldsymbol{s}(t)\in S(t)$ for every $t \in [0,T]$. Let $\boldsymbol{s}(t)=(\boldsymbol{y}(t),\boldsymbol{v}(t),\boldsymbol{V}(t))$. By definition of $S$, for every $t \in [0,T]$ there exist a sequence of indices $\left (n_k \right)$, possibly depending on $t$,  such that
\begin{equation}
\label{eqn:t-subseq-deformation}
    \text{$\boldsymbol{y}_{n_k}(t)\wk \boldsymbol{y}(t)$ in $W^{1,p}(\Omega;\R^N)$},
\end{equation}
\begin{equation}
\label{eqn:t-subseq-m}
    \text{$\boldsymbol{v}_{n_k}(t)\wk \boldsymbol{v}(t)$ in $L^2(\R^N;\R^N)$},
\end{equation}
\begin{equation}
\label{eqn:t-subseq-Dm}
    \text{$\boldsymbol{V}_{n_k}(t)\wk \boldsymbol{V}(t)$ in $L^2(\R^N;\rnn)$}.
\end{equation}
By claim (i) of Proposition \ref{lem:properties-total-energy}, we deduce several facts. First, $\boldsymbol{y}(t)\in \mathcal{Y}_p(\Omega)$ and there holds
\begin{equation}
\label{eqn:t-subseq-jacobian}
    \text{$\det D \boldsymbol{y}_{n_k}(t)\wk \det D\boldsymbol{y}(t)$ in $L^1(\Omega)$}.
\end{equation}
Second, there exists $\boldsymbol{m}(t)\in W^{1,2}(\im_{\rm T}(\boldsymbol{y}(t),\Omega);\R^N)$ satisfying 
\begin{equation*}
\text{$|\boldsymbol{m}(t)\circ \boldsymbol{y}(t)|\det D \boldsymbol{y}(t)=1$ a.e. in $\Omega$,}    
\end{equation*}
such that 
\begin{equation*}
    \boldsymbol{v}(t)\coloneqq \chi_{\im_{\rm T}(\boldsymbol{y}(t),\Omega)}\boldsymbol{m}(t),\qquad \boldsymbol{V}(t)\coloneqq \chi_{\im_{\rm T}(\boldsymbol{y}(t),\Omega)}D\boldsymbol{m}(t). 
\end{equation*}
In particular, setting $\boldsymbol{q}(t)\coloneqq (\boldsymbol{y}(t),\boldsymbol{m}(t))\in \mathcal{Q}$ for every $t \in [0,T]$,  the map $\boldsymbol{q}\colon [0,T]\to \mathcal{Q}$ is  measurable. Third, up to subsequences, there holds
\begin{equation*}
    \text{$\boldsymbol{z}_{n_k}(t)\to \boldsymbol{m}(t)\circ \boldsymbol{y}(t)\det D \boldsymbol{y}(t)$ in $L^1(\Omega;\R^N)$,}
\end{equation*}
which, together with \eqref{eqn:compactness-z}, yields
\begin{equation*}
    \text{$\boldsymbol{z}(t)=\boldsymbol{m}(t)\circ \boldsymbol{y}(t)\det D \boldsymbol{y}(t)$ a.e. in $\Omega$.}
\end{equation*}
For every $n \in \N$, define $\vartheta_n\colon [0,T]\to \R$ by setting $\vartheta_n(t)\coloneqq \partial_t \mathcal{F}(t,\boldsymbol{q}_n(t))$. As $\boldsymbol{q}_n$ is piecewise-constant, the function $\vartheta_n$ is measurable. Also, by \eqref{eqn:control-time-derivative} and \eqref{eqn:a-priori-estimate}, we have
\begin{equation}
\label{eqn:theta-control}
    |\vartheta_n(t)|\leq \eta(t)\left ( \mathcal{F}(t,\boldsymbol{q}_n(t))+L\right) \leq (M+L) \eta(t)
\end{equation}
for almost every $t \in (0,T)$, so that the sequence $(\vartheta_n)$ is equi-integrable. Thus, by the Dunford-Pettis Theorem \cite[Theorem 2.54]{fonseca.leoni}, up to subsequence, there holds
\begin{equation}
\label{eqn:weak-convergence-power}
    \text{$\vartheta_n \wk \vartheta$ in $L^1(0,T)$}
\end{equation}
for some $\vartheta \in L^1(0,T)$. If we define $\Bar{\vartheta} \colon (0,T)\to \R$ by setting 
\begin{equation*}
    \Bar{\vartheta}(t)\coloneqq \limsup_{n \to \infty} \vartheta_n(t),
\end{equation*}
then, by \eqref{eqn:theta-control}, we deduce $\Bar{\vartheta}\in L^1(0,T)$. Moreover, by the Reverse Fatou Lemma, there holds $\vartheta\leq \Bar{\vartheta}$ almost everywhere.

Let $t \in (0,T)\setminus Z$, where  $Z\subset (0,T)$ is the set in claim (ii) of Lemma \ref{lem:properties-total-energy}. Up to subsequences, we can assume that, for the sequence of indices $(n_k)$ for which\eqref{eqn:t-subseq-deformation}--\eqref{eqn:t-subseq-m} hold true, we also have
\begin{equation*}
    \Bar{\vartheta}(t)=\lim_{k \to \infty} \vartheta_{n_k}(t).
\end{equation*}
In this case, recalling \eqref{eqn:time-derivative-expression} and exploiting \eqref{eqn:t-subseq-deformation}--\eqref{eqn:t-subseq-m}, we deduce that $\bar{\vartheta}(t)=\partial_t \mathcal{F}(t,\boldsymbol{q}(t))$. Therefore, this equality holds for almost every $t \in (0,T)$.

\textbf{Step 3 (Stability).} The remaining steps of the proof are quite standard. We first prove that the map $\boldsymbol{q}$ satisfies the global stability condition \eqref{eqn:global-stability}. Fix $t \in [0,T]$ and let $(n_k)$ be the  sequence of indices for which \eqref{eqn:t-subseq-deformation}--\eqref{eqn:t-subseq-jacobian} hold true. For every $k \in \N$, set $\tau_k(t)\coloneqq \max \left \{s \in \Pi_{n_k}:\:s \leq t \right \}$. Since $|\Pi_{n_k}|\to 0$, as $k \to \infty$, we have $\tau_k(t)\to t$, as $k \to \infty$. Also, by \eqref{eqn:pc-interpolant}, we have $\boldsymbol{q}_{n_k}(t)=\boldsymbol{q}_{n_k}(\tau_k(t))$. By \eqref{eqn:pc-interpolants-stability}, there holds
\begin{equation*}
    \forall\,\widehat{\boldsymbol{q}}\in \mathcal{Q}, \quad \mathcal{F}(\tau_k(t),\boldsymbol{q}_{n_k}(t))\leq \mathcal{F}(\tau_k(t),\widehat{\boldsymbol{q}})+\mathcal{D}(\boldsymbol{q}_{n_k}(t),\widehat{\boldsymbol{q}}).
\end{equation*}
Passing to the limit, as $k \to \infty$, with the aid of claim (ii) of Proposition \ref{lem:properties-total-energy} taking into account \eqref{eqn:compactness-z},  \eqref{eqn:t-subseq-deformation}--\eqref{eqn:t-subseq-jacobian}, and the continuity of $\widehat{t} \mapsto \mathcal{F}(\widehat{t},\widehat{\boldsymbol{q}})$, we obtain
\begin{equation*}
        \mathcal{F}(t,\boldsymbol{q}(t))\leq \liminf_{k \to \infty} \mathcal{F}(\tau_k(t),\boldsymbol{q}_{n_k}(t))
        \leq \liminf_{k \to \infty} \left \{ \mathcal{F}(\tau_k(t),\widehat{\boldsymbol{q}})+\mathcal{D}(\boldsymbol{q}_{n_k}(t),\widehat{\boldsymbol{q}}) \right \}=\mathcal{F}(t,\widehat{\boldsymbol{q}})+\mathcal{D}(\boldsymbol{q}(t),\widehat{\boldsymbol{q}}).
\end{equation*}
This proves \eqref{eqn:global-stability} for $t \in [0,T]$ fixed.

\textbf{Step 4 (Upper energy-dissipation inequality).} We want to prove the following inequality
\begin{equation}
    \label{eqn:u-e-d}
    \forall\,t \in [0,T], \quad \mathcal{F}(t,\boldsymbol{q}(t))+\var_{\mathcal{D}}(\boldsymbol{q};[0,t])\leq \mathcal{E}(0,\boldsymbol{q}^0)+\int_0^t \partial_t \mathcal{F}(\tau,\boldsymbol{q}(\tau))\,\d\tau.
\end{equation}
Fix $t \in [0,T]$ and let $(n_k)$ and $\tau_k$ be as in Step 3.  We compute
\begin{equation*}
    \begin{split}
        |\mathcal{F}(t,\boldsymbol{q}_{n_k}(t))-\mathcal{F}(\tau_k(t),\boldsymbol{q}_{n_k}(\tau_k(t)))|&=|\mathcal{F}(t,\boldsymbol{q}_{n_k}(\tau_k(t)))-\mathcal{E}(\tau_k(t),\boldsymbol{q}_{n_k}(\tau_k(t)))|\\
        &\leq \int_{\tau_k(t)}^t |\partial_t \mathcal{F}(\tau,\boldsymbol{q}_{n_k}(\tau_k(t)))|\,\d\tau\\
        &\leq \left ( \mathcal{F}(\tau_k(t),\boldsymbol{q}_{n_k}(\tau_k(t)))+L \right) \int_{\tau_k(t)}^t \eta(\tau) \mathrm{e}^{H(\tau)-H(\tau_k(t))}\,\d\tau\\
        &\leq (M+L) \left(\mathrm{e}^{H(t)-H(\tau_k(t))}-1 \right),
    \end{split}
\end{equation*}
where we employed \eqref{eqn:gronwall-derivative} and \eqref{eqn:a-priori-estimate}.  From this, we immediately obtain
\begin{equation}
\label{eqn:ued-prelim}
    \mathcal{F}(t,\boldsymbol{q}_{n_k}(t))\leq \mathcal{F}(\tau_k(t),\boldsymbol{q}_{n_k}(\tau_k(t)))+(M+L) \left(\mathrm{e}^{H(t)-H(\tau_k(t))}-1 \right).
\end{equation}
Observe that $\var_{\mathcal{D}}(\boldsymbol{q}_{n_k};[0,t])=\var_{\mathcal{D}}(\boldsymbol{q}_{n_k};[0,\tau_k(t)])$. This, together with \eqref{eqn:pc-interpolants-energy-dissipation-balance} and \eqref{eqn:ued-prelim}, yields
\begin{equation}
\label{eqn:ued-prelim-2}
    \begin{split}
        \mathcal{F}(t,\boldsymbol{q}_{n_k}(t))+\var_{\mathcal{D}}(\boldsymbol{q}_{n_k};[0,t])&\leq \mathcal{F}(\tau_k(t),\boldsymbol{q}_{n_k}(\tau_k(t)))+\var_{\mathcal{D}}(\boldsymbol{q}_{n_k};[0,\tau_k(t)])\\
        &+(M+L) \left(\mathrm{e}^{H(t)-H(\tau_k(t))}-1 \right)\\
        &\leq\mathcal{F}(0,\boldsymbol{q}^0)+\int_0^{\tau_k(t)} \partial_t \mathcal{F}(\tau,\boldsymbol{q}_{n_k}(\tau))\,\d\tau\\
        &+(M+L) \left(\mathrm{e}^{H(t)-H(\tau_k(t))}-1 \right)\\
        &=\mathcal{F}(0,\boldsymbol{q}^0)+\int_0^{\tau_k(t)} \vartheta_{n_k}(\tau)\,\d\tau\\
        &+(M+L) \left(\mathrm{e}^{H(t)-H(\tau_k(t))}-1 \right).
    \end{split}
\end{equation}
By \eqref{eqn:compactness-z} and by the lower semicontinuity of the total variation, we have
\begin{equation}
\label{eqn:lsc-total-variation}
\begin{split}
    \var_{\mathcal{D}}(\boldsymbol{q};[0,t])&=\var_{L^1(\Omega;\R^N)}(\boldsymbol{z};[0,t])\\
    &\leq \liminf_{k \to \infty} \var_{L^1(\Omega;\R^N)}(\boldsymbol{z}_{n_k};[0,t])\\
    &=\liminf_{k \to \infty} \var_{\mathcal{D}}(\boldsymbol{q}_{n_k};[0,t]).
\end{split}
\end{equation}
Recall \eqref{eqn:t-subseq-deformation}--\eqref{eqn:t-subseq-jacobian}, \eqref{eqn:weak-convergence-power} and \eqref{eqn:lsc-total-variation}. Applying claim (ii) of Proposition \ref{lem:properties-total-energy}, we take the inferior limit, as $k \to \infty$, in \eqref{eqn:ued-prelim-2}  and  we obtain
\begin{equation*}
    \begin{split}
        \mathcal{F}(t,\boldsymbol{q}(t)) + \var_{\mathcal{D}}(\boldsymbol{q};[0,t])&\leq\liminf_{k \to \infty} \left \{ \mathcal{F}(t,\boldsymbol{q}_{n_k}(t))+\var_{\mathcal{D}}(\boldsymbol{q}_{n_k};[0,t])\right\}\\
        &\leq\mathcal{F}(0,\boldsymbol{q}^0)+\int_0^t \vartheta(\tau)\,\d\tau\\
        &\leq \mathcal{F}(0,\boldsymbol{q}^0)+\int_0^t \partial_t\mathcal{F}(\tau,\boldsymbol{q}(\tau))\,\d\tau.
    \end{split}
\end{equation*}
Here, we exploited that $\tau_k(t)\to t$, as $k \to \infty$, and that $\vartheta\leq \Bar{\vartheta}$ almost everywhere. This proves \eqref{eqn:u-e-d}.

\textbf{Step 5 (Lower energy-dissipation inequality).} In view of \eqref{eqn:u-e-d}, in order to establish \eqref{eqn:energy-dissipation-balance}, we are left to prove the following:
\begin{equation}
    \label{eqn:l-e-d}
    \forall\,t \in [0,T], \quad \mathcal{F}(t,\boldsymbol{q}(t))+\var_{\mathcal{D}}(\boldsymbol{q};[0,t])\geq \mathcal{F}(0,\boldsymbol{q}^0)+\int_0^t \partial_t \mathcal{F}(\tau,\boldsymbol{q}(\tau))\,\d\tau.
\end{equation}
This is deduced from \eqref{eqn:global-stability} by arguing as in \cite[Proposition 2.1.23]{mielke.roubicek}. 
\end{proof}

We conclude by mentioning that, so far, we are not able to impose time-dependent Dirichlet boundary conditions in the sense of traces. In particular, the strategy in \cite{francfort.mielke} seems not to be applicable to our problem. The obstacle is that, regardless to the regularity of the boundary datum, the magnetostatic energy is not differentiable in time when composed with it. However, the other energy terms are well behaved in this sense. Therefore, the strategy in \cite{francfort.mielke} in combination with our results could be used to study quasistatic evolutions for variational models of nematic elastomers like the ones in \cite{barchiesi.henao.moracorral} and \cite{henao.stroffolini} by imposing time-dependent boundary conditions in the sense of traces.

\section*{Appendix: Sobolev maps on the boundary of domains of class \texorpdfstring{$C^2$}{c2} }
\label{sec:appendix}

\setcounter{section}{0}
\setcounter{theorem}{0}

\setcounter{equation}{0}

\renewcommand{\thetheorem}{A.\arabic{theorem}}

\setcounter{equation}{0}
\renewcommand{\theequation}{A.\arabic{equation}}

In this appendix, we briefly recall the definition of Sobolev maps on the boundary of smooth domains. The aim is to fix notation and terminology. Then, we provide a detailed proof of Lemma \ref{lem:weak-convergence-boundary}.

\subsection*{Tangential differentiability}
Henceforth, $U\subset \subset \R^N$ denotes a domain of class $C^2$. In this case, $\partial U$ is an hypersurface of class $C^2$ without boundary. For every $\boldsymbol{x}_0\in\partial U$, there exist a local chart at $\boldsymbol{x}_0$, namely an injective immersion $\boldsymbol{\zeta}\in C^2(\closure{Q};\R^N)$, where $Q\subset \R^{N-1}$ is a bounded domain, such that $\boldsymbol{\zeta}(Q)$ is an relatively open set of $\partial U$ containing $\boldsymbol{x}_0$. In this case, the tangent space $T_{\boldsymbol{x}_0}\partial U$ is spanned by the unit vectors $(\boldsymbol{t}^{(1)}_U(\boldsymbol{x}_0),\dots,\boldsymbol{t}^{(N-1)}_U(\boldsymbol{x}_0))$, where, for every $i\in \{1,\dots,N-1\}$, we set
\begin{equation*}
\boldsymbol{t}^{(i)}_U(\boldsymbol{x}_0)\coloneqq \frac{D\boldsymbol{\zeta}(\boldsymbol{\zeta}^{-1}(\boldsymbol{x}_0))\boldsymbol{e}_i}{|D\boldsymbol{\zeta}(\boldsymbol{\zeta}^{-1}(\boldsymbol{x}_0))\boldsymbol{e}_i|}.
\end{equation*}
Here, $(\boldsymbol{e}_1,\dots,\boldsymbol{e}_{N-1})$ is the canonical basis of $\R^{N-1}$. 
We denote the projection operator onto the tangent space by $\boldsymbol{\Pi}_{T_{\boldsymbol{x}_0}\partial U}\colon \R^N \to T_{\boldsymbol{x}_0}\partial U$.
Given the map  from $\R^N$ to $\R^{N-1}$ defined by
\begin{equation*}
\boldsymbol{x}\mapsto \sum_{i=1}^{N-1} \left (\boldsymbol{x}\cdot \boldsymbol{t}^{(i)}_U(\boldsymbol{x}_0) \right)\boldsymbol{e}_i,
\end{equation*}
we denote by $\boldsymbol{\Lambda}_U(\boldsymbol{x}_0)\in\R^{(N-1)\times N}$ the matrix representing it with respect to the canonical basis of $\R^N$ and $\R^{N-1}$. In this case, $\boldsymbol{\Lambda}_U\in C^1(\partial U;\R^{(N-1)\times N})$. Moreover, 
there holds $\boldsymbol{\Lambda}_U(\boldsymbol{\zeta}(\boldsymbol{x}_0))D\boldsymbol{\zeta}(\boldsymbol{\zeta}^{-1}(\boldsymbol{x}_0))=\boldsymbol{J}$, where $\boldsymbol{J}\in\R^{(N-1)\times (N-1)}$ is the identity matrix.

We now recall the classical definition of tangential differentiability. We give the definition for vector-valued maps which are the ones of interest for our purposes.

\begin{definition}[Tangential differentiability]
	\label{def:tangential-differentiability}
	A map $\boldsymbol{y}\colon \partial U \to \R^N$  is termed to be of class $C^1$ if there exists a map $\boldsymbol{Y}\in C^1(\R^N;\R^N)$ such that $\boldsymbol{Y}\restr{\partial U}=\boldsymbol{y}$. In this case, we write  $\boldsymbol{y}\in C^1(\partial U;\R^N)$ and we define the tangential gradient of $\boldsymbol{y}$ at $\boldsymbol{x}_0\in \partial U$ as $D^{\partial U}\boldsymbol{y}(\boldsymbol{x}_0)\coloneqq D\boldsymbol{y}(\boldsymbol{x}_0) \left( \boldsymbol{I}-\boldsymbol{n}_U(\boldsymbol{x}_0)\otimes \boldsymbol{n}_U(\boldsymbol{x}_0)\right)$.
\end{definition}

Note that $D^{\partial U}\boldsymbol{y}(\boldsymbol{x}_0)\in \rnn$. Basically, the rows of this matrix are given by $\boldsymbol{\Pi}_{T_{\boldsymbol{x}_0}\partial U}(DY^i(\boldsymbol{x}_0))$ for $i \in \{1,\dots,N\}$. It is easily shown that the previous definiton is well posed in the sense that does not depend on the choice of the extension $\boldsymbol{Y}$. Clearly, for $\boldsymbol{y}\in C^1(\R^N;\R^N)$, we have
\begin{equation}
\label{eqn:A1}
D^{\partial U}\boldsymbol{y}\restr{\partial U}(\boldsymbol{x}_0)= D\boldsymbol{y}(\boldsymbol{x}_0) \left( \boldsymbol{I}-\boldsymbol{n}_U(\boldsymbol{x}_0)\otimes \boldsymbol{n}_U(\boldsymbol{x}_0)\right).
\end{equation}
Definition \ref{def:tangential-differentiability} is equivalent to the one generally used in differential geometry: a map $\boldsymbol{y}\colon \partial U \to\R^N$ belongs to $C^1(\partial U;\R^N)$ if and only if, for every $\boldsymbol{x}_0\in \partial U$, there exists a local chart $\boldsymbol{\zeta}\in C^2(\closure{A};\R^N)$ at $\boldsymbol{x}_0$ such that $\boldsymbol{\upsilon}\coloneqq \boldsymbol{y}\circ \boldsymbol{\zeta}\in C^1(A;\R^N)$. In this case, we have the identity
\begin{equation}
\label{eqn:A2}
D^{\partial U}\boldsymbol{y}(\boldsymbol{x}_0)=D\boldsymbol{\upsilon}(\boldsymbol{\zeta}^{-1}(\boldsymbol{x}_0))\boldsymbol{\Lambda}_U(\boldsymbol{x}_0).
\end{equation}
Given $\boldsymbol{y}\in C^1(\partial U;\R^N)$, its tangential divergence at $\boldsymbol{x}_0\in \partial U$ is simply defined as 
\begin{equation*}
\div^{\partial U} \boldsymbol{y}(\boldsymbol{x}_0)\coloneqq \tr D^{\partial U}\boldsymbol{y}(\boldsymbol{x}_0).
\end{equation*}

\subsection*{Divergence Theorem and Integration-by-parts Formula}
We know recall the Divergence Theorem on surfaces. Here, we limit ourselves to the case which is of interest for us, namely the case of boundaries of smooth domains. However, these results hold true for more general smooth hypersurfaces possibly with boundary.

Before proceeding, we need to introduce the mean curvature. Observe that, given the regularity of $U$, we have $\boldsymbol{n}_U \in C^1(\partial U;\R^N)$. The scalar mean curvature of $\partial U$ is defined as 
\begin{equation*}
H_U\coloneqq \div^{\partial U}\boldsymbol{n}_U.
\end{equation*}
The reader must be warned that different conventions of sign in the previous definition are also considered in the literature.

We  now state the Divergence Theorem. For a proof,	we refet to \cite[Chapter 3]{ambrosio},  \cite[Theorem 2, Chapter 2, Section 1.5]{cartesian.currents} or \cite[Theorem 11.8 and Remark 11.12]{maggi}.
\begin{theorem}[Divergence Theorem on boundaries]
	\label{thm:div}
	Let $U \subset \subset \R^N$ be a domain of class $C^2$ and let $\boldsymbol{y}\in C^1(\partial U;\R^N)$. Then, there holds:
	\begin{equation*}
	\int_{\partial U} \div^{\partial U}\boldsymbol{y}\,\d\boldsymbol{a}=\int_{\partial U} (\boldsymbol{y}\cdot \boldsymbol{n}_U)\,H_U\,\d\boldsymbol{a}.
	\end{equation*}
\end{theorem}

From Theorem \ref{thm:div}, we easily deduce the following identities.

\begin{corollary}[Integration-by-parts Formula on boundaries]
	\label{cor:ibp}
	Let $U \subset \subset \R^N$ be a domain of class $C^2$ and let $\boldsymbol{y}\in C^1(\partial U;\R^N)$. Then, for every $\boldsymbol{\varphi}\in C^1(\partial U;\R^N)$, the following formula holds:
	\begin{equation*}
	\int_{\partial U} (D^{\partial U}\boldsymbol{y})\boldsymbol{\varphi}\,\d\boldsymbol{a}=\int_{\partial U} \boldsymbol{y}(\boldsymbol{\varphi}\cdot \boldsymbol{n}_U)\,H_U\,\d\boldsymbol{a}-\int_{\partial U} \boldsymbol{y}\,\div^{\partial U}\boldsymbol{\varphi}\,\d\boldsymbol{a}.
	\end{equation*}
	Equivalently, for every $\boldsymbol{\Phi}\in C^1(\partial U;\rnn)$, there holds
	\begin{equation*}
	\int_{\partial U} D^{\partial U}\boldsymbol{y}:\boldsymbol{\Phi}\,\d\boldsymbol{a}=\int_{\partial U} (\boldsymbol{\Phi}\boldsymbol{n}_U)\cdot \boldsymbol{y}\,H_U\,\d\boldsymbol{a}-\int_{\partial U} \boldsymbol{y}\cdot \div^{\partial U}\boldsymbol{\Phi}\,\d\boldsymbol{a}.
	\end{equation*} 
\end{corollary}
In the last equation, the tangential divergence operator is suppose to act row-wise on matrix-valued maps. Namely, the rows of the vector field $\div^{\partial U}\boldsymbol{\Phi}$  are given by $\div^{\partial U}\boldsymbol{\Phi}^i$ for $i \in \{1,\dots,N\}$.

\subsection*{Weak tangential differentiability}

The notion of weak tangetial differentiability is conceptually analogous to the one of weak differentiability. In this case, one the classical Integration-by-parts Formula is replaced by the one provided by Corollary \ref{cor:ibp}. Again, we focus on the case of vector-valued maps.

\begin{definition}[Weak tangential differentiability]
	A map $\boldsymbol{y}\in L^1(\partial U;\R^N)$ is termed weakly tangentially differentiable if there exists $\boldsymbol{F}\in L^1(\partial U;\rnn)$ such that, for every $\boldsymbol{\varphi}\in C^1(\partial U;\R^N)$, there holds
	\begin{equation*}
	-\int_{\partial U} \boldsymbol{y}\,\div^{\partial U}\boldsymbol{\varphi}\,\d\boldsymbol{a}=\int_{\partial U} \boldsymbol{F}\boldsymbol{\varphi}\,\d\boldsymbol{a}-\int_{\partial U} \boldsymbol{y}(\boldsymbol{\varphi}\cdot \boldsymbol{n}_U)\,H_U\,\d\boldsymbol{a},
	\end{equation*}
	or, equivalently, for every $\boldsymbol{\Phi}\in C^1(\partial U;\rnn)$, there holds
	\begin{equation*}
	-\int_{\partial U} \boldsymbol{y}\cdot \div^{\partial U}\boldsymbol{\Phi}\,\d\boldsymbol{a}=\int_{\partial U} \boldsymbol{F}:\boldsymbol{\Phi}\,\d\boldsymbol{a}-\int_{\partial U} (\boldsymbol{\Phi}\boldsymbol{n}_U)\cdot \boldsymbol{y}\,H_U\,\d\boldsymbol{a}.
	\end{equation*}
	In this case, $D^{\partial U}\boldsymbol{y}\coloneqq \boldsymbol{F}$ is termed the weak tangential gradient of $\boldsymbol{y}$.
\end{definition}

The weak tangential gradient is unique up to sets of measure zero and, in view of Corollay \ref{cor:ibp}, coincides with the tangential gradient in the case of tangentially differentiable maps.

At this point, for every $1\leq p<\infty$, we define the Sobolev space $W^{1,p}(\partial U;\R^N)$ as the class of maps $\boldsymbol{y}\in L^p(\partial U;\R^N)$ admitting weak tangential gradient $D^{\partial U}\boldsymbol{y}$ which belongs to $L^p(\partial U;\rnn)$. Up to considering equivalence classes of maps that coincide almost everywhere, this is Banach space equipped with the norm
\begin{equation}
\label{eqn:norm}
\boldsymbol{y}\mapsto \left \{\|\boldsymbol{y}\|_{L^p(\partial U;\R^N)}^p+\|D^{\partial U}\boldsymbol{y}\|_{L^p(\partial U;\rnn)}^p   \right \}^{1/p}.
\end{equation} 
By exploiting the natural embedding of $W^{1,p}(\partial U;\R^N)$ in $L^p(\partial U;\R^N)\times L^p(\partial U;\rnn)$ given by $\boldsymbol{y}\mapsto (\boldsymbol{y},D^{\partial U}\boldsymbol{y})$, we see that elements of the dual of $W^{1,p}(\partial U;\R^N)$ admit a representation analogous to the one valid in the case of domains. Given $(\boldsymbol{y}_n)\subset W^{1,p}(\partial U;\R^N)$ and $\boldsymbol{y}\in W^{1,p}(\partial U;\R^N)$, we have that $\boldsymbol{y}_n \wk \boldsymbol{y}$ in $W^{1,p}(\partial U;\R^N)$ if and only if $\boldsymbol{y}_n \wk \boldsymbol{y}$ in $L^p(\partial U;\R^N)$ and $D^{\partial U}\boldsymbol{y}_n \wk D^{\partial U}\boldsymbol{y}$ in $L^p(\partial U;\rnn)$.

The following result shows that, in the case of domains of class $C^2$,  the definition of Sobolev maps on the boundary coincides with the usual definition of Sobolev maps on the boundary given by means of local charts \cite{necas}. This can be proved using standard techniques involving partitions of unity and the density of smooth functions in Sobolev spaces by exploiting \eqref{eqn:A1}--\eqref{eqn:A2}.

\begin{proposition}[Sobolev maps on the boundary]
	\label{prop:sobolev-boundary}
	Let $U \subset \subset \R^N$ be a domain of class $C^2$ and let $\boldsymbol{y}\in L^1(\partial U;\R^N)$. Then, $\boldsymbol{y} \in W^{1,p}(\partial U;\R^N)$ if and only if for every $\boldsymbol{x}_0$ and for every local chart $\boldsymbol{\zeta}\in C^2(\closure{A};\R^N)$ at $\boldsymbol{x}_0$, there holds $\boldsymbol{\upsilon}\coloneqq \boldsymbol{y}\circ \boldsymbol{\zeta}\in W^{1,p}(A;\R^N)$.
\end{proposition}

This proposition allows us to extend various results about Sobolev spaces on domain to Sobolev spaces on boundaries of smooth domains. Among these, we have the density of $C^1(\partial U;\R^N)$ in $W^{1,p}(\partial U;\R^N)$, the embeddings of Sobolev and Morrey, and the Poincaré inequality. In particular,  the density of smooth maps shows that the approach presented here is equivalent to the one usually adopted in the theory of Sobolev spaces on Riemannian manifolds \cite{hebey} which are defined as the closure of smooth functions with respect to the norm in \eqref{eqn:norm}. Finally, we stress that the validity of the Morrey embedding in $W^{1,p}(\partial U;\R^N)$ is essential for the analysis carried out in the present work. 

\subsection*{Convergence of Sobolev maps on boundaries}

We conclude this appendix by providing a detailed proof of Lemma \ref{lem:weak-convergence-boundary}.

\begin{proof}[Proof of Lemma \ref{lem:weak-convergence-boundary}]
	Let $\delta>0$ be such that the tubular neighborhood $T(\partial U,\delta)$ is defined.  Claim \eqref{eqn:sobolev-regularity-boundary} is established by mollification and by applying the Coarea Formula \cite[Equation (2.74)]{ambrosio.fusco.pallara} with the signed distance function. By the Rellich embedding, up to subsequences, we have  $\boldsymbol{y}_n \to \boldsymbol{y}$ in $L^p(T(\partial U,\delta);\R^N)$. Hence, by the Coarea Formula \cite[Equation (2.74)]{ambrosio.fusco.pallara}, we have
	\begin{equation}
	\label{eqn:strong-convergence-boundary}
	\text{$\boldsymbol{y}_n\to \boldsymbol{y}$ in $L^p(\partial U_\ell;\R^N)$ for every $\ell \in (-\delta,\delta)\setminus Z_1$,} 
	\end{equation}
	for some $Z_1 \subset (-\delta,\delta)$ with $\mathscr{L}^1(Z_1)=0$. For every $n \in \N$, define $f_n\colon (-\delta,\delta)\to \R$ and $f\colon (-\delta,\delta)\to \R$ by setting
	\begin{equation*}
	f_n(\ell)\coloneqq \int_{\partial U_\ell} |D^{\partial U_\ell}\boldsymbol{y}_n|^p\,\d\boldsymbol{a}, \qquad f(\ell)\coloneqq \liminf_{n \to \infty} f_n(\ell).
	\end{equation*}
	By the Fatou Lemma and by the Coarea Formula \cite[Equation (2.74)]{ambrosio.fusco.pallara}, $f \in L^1(-\delta,\delta)$ and, in turn, $f(\ell)<+\infty$ for every $\ell \in (-\delta,\delta)\setminus Z_2$ for some $Z_2 \subset (-\delta,\delta)$ with $\mathscr{L}^1(Z_2)=0$. Fix $\ell \in (-\delta,\delta)\setminus (Z_1 \cup Z_2)$ and select a subsequence indexed by $(n_k)$, possibly depending on $\ell$, such that $$\lim_{k \to \infty} f_{n_k}(\ell)=f(\ell)<+\infty.$$ Thus, for $k \gg 1$, the sequence $(D^{\partial U_\ell}\boldsymbol{y}_{n_k})\subset L^p(\partial U_\ell;\rnn)$ is bounded so that, up to subsequences, we have
	\begin{equation}
	\label{eqn:scb2}
	\text{$D^{\partial U_\ell}\boldsymbol{y}_{n_k} \wk \boldsymbol{F}$ in $L^p(\partial U_\ell;\rnn)$}
	\end{equation}
	for some $\boldsymbol{F}\in L^p(\partial U_\ell;\rnn)$.
	Now, given \eqref{eqn:sobolev-regularity-boundary}, for every $\boldsymbol{\Phi}\in C^1(\partial U_\ell;\rnn)$ and $k \in \N$ the  Integration-by-parts Formula yields
	\begin{equation*}
	-\int_{\partial U_\ell} \boldsymbol{y}_{n_k}\cdot \div^{\partial U_\ell} \boldsymbol{\Phi}\,\d\boldsymbol{a}=\int_{\partial U_\ell} D^{\partial U_\ell}\boldsymbol{y}_{n_k}:\boldsymbol{\Phi}\,\d\boldsymbol{a}-\int_{\partial U_\ell} H_{U_\ell}\,\boldsymbol{y}_{n_k}\cdot \boldsymbol{\Phi}\boldsymbol{n}_{U_\ell}\,\d\boldsymbol{a},
	\end{equation*}
	where $H_{U_\ell}$ denotes the scalar mean curvature of $\partial U_\ell$. 
	Passing to the limit, as $k\to \infty$, in the previous equation with the aid of \eqref{eqn:strong-convergence-boundary}--\eqref{eqn:scb2}, we deduce that  $\boldsymbol{F}=D^{\partial U_\ell}\boldsymbol{y}$. Therefore, $\boldsymbol{y}_{n_k}\wk \boldsymbol{y}$ in $W^{1,p}(\partial U_\ell;\R^N)$. Finally,  since $p>N-1$, we have ${\boldsymbol{y}}_{n_k} \to {\boldsymbol{y}}$ uniformly on $\partial U_\ell$ by the Morrey embedding.
\end{proof}

\section*{Acknowledgements}
I am thankful to Martin Kru\v{z}\'{i}k and Stefan Kr\"{o}mer  for  insightful discussions on the magnetic saturation constraint and its analytical treatment in the case of continuous and globally invertible deformations. I am also greateful for their hospitality at the Institute of Information Theory and Automation of the Czech Academy of Science in Prague. This work has been supported by the Austrian Science Fund (FWF) and the GA\v{CR} through the  grant I4052-N32/19-29646L and by the Federal Ministry of Education, Science and Research of Austria (BMBWF) through the OeAD-WTZ project CZ04/2019 and  M\v{S}MT \v{CR} project 8J19AT013.

\end{document}